\documentclass[a4paper,12pt,english]{amsart}

\usepackage[utf8]{inputenc}
\usepackage{lmodern}
\usepackage[english]{babel}
\usepackage{amsmath}
\usepackage{amsfonts}
\usepackage{hyperref}
\usepackage{dsfont}			
\usepackage{amsmath, amsthm, amssymb, amsthm,bbm,bm,mathrsfs}
\usepackage{todonotes}

\newtheorem{prop}{Proposition}[section]
\newtheorem{theorem}[prop]{Theorem}
\newtheorem{corollary}[prop]{Corollary}
\newtheorem{lemma}[prop]{Lemma}
\newtheorem{definition}[prop]{Definition}
\newtheorem{remark}[prop]{Remark}

\newtheorem{eg}[prop]{Example}

\newcommand{\0}{\mathbf{0}}
\newcommand{\1}{\mathbf{1}}
\newcommand{\A}{\mathcal{A}}
\newcommand{\B}{\mathcal{B}}
\newcommand{\cC}{\mathcal{C}}
\newcommand{\C}{\mathbb{C}}
\newcommand{\Cx}{\C\langle x_1,\ldots,x_d\rangle}
\newcommand{\D}{\mathcal{D}}
\newcommand{\cF}{\mathcal{F}}

\newcommand{\cL}{\mathcal{L}}
\newcommand{\M}{\mathcal{M}}
\newcommand{\bN}{\mathbb{N}}
\newcommand{\cN}{\mathcal{N}}
\newcommand{\cP}{\mathcal{P}}
\newcommand{\R}{\mathbb{R}}
\newcommand{\cR}{\mathcal{R}}
\newcommand{\K}{\mathcal{K}}
\newcommand{\U}{\mathcal{U}}
\newcommand{\X}{\mathcal{X}}
\newcommand{\Z}{\mathbb{Z}}

\newcommand{\diag}{\mathrm{diag}}
\newcommand{\dom}{\mathrm{dom}}

\newcommand{\id}{\mathrm{id}}
\newcommand{\im}{\mathrm{im}}
\newcommand{\lcm}{\mathrm{lcm}}

\newcommand{\Null}{\mathrm{Nullity}}

\newcommand{\rk}{\mathrm{rk}}
\newcommand{\rank}{\mathrm{rank}}
\newcommand{\tr}{\mathrm{tr}}
\newcommand{\Tr}{\mathrm{Tr}}
\newcommand{\diff}{\mathrm{d}}
\newcommand{\Subs}{\mathrm{Subs}}

\makeatletter
\def\moverlay{\mathpalette\mov@rlay}
\def\mov@rlay#1#2{\leavevmode\vtop{%
\baselineskip\z@skip \lineskiplimit-\maxdimen
\ialign{\hfil$#1##$\hfil\cr#2\crcr}}}
\makeatother
\def\plangle{\moverlay{(\cr<}}
\def\prangle{\moverlay{)\cr>}}

\makeindex

\title[Universality of free random variables]{Universality of free random variables:\\atoms for non-commutative rational functions}

\author[O. Arizmendi]{Octavio Arizmendi}
\thanks{O.A. was supported by Conacyt grant A1- S-9764}
\address{Department of Probability and Statistics. Centro de Investigaci\'{o}n en Matem\'{a}ticas, Guanajuato, M\'{e}xico.}
\email{octavius@cimat.mx}
\author[G. C\'ebron]{Guillaume C\'ebron}\thanks{G.C. and S.Y. were partly supported by the Project MESA (ANR-18-CE40-006) and by the Project STARS (ANR-20-CE40-0008) of the French National Research Agency (ANR)}
\address{Institut de Mathématiques de Toulouse; UMR5219; Université de Toulouse; CNRS; UPS, F-31062 Toulouse, France}
\email{guillaume.cebron@math.univ-toulouse.fr}
\author[R. Speicher]{Roland Speicher}
\thanks{O.A. and R.S. gratefully acknowledge financial support by the SFB-TRR 195 “Symbolic Tools in Mathematics and
their Application” of the German Research Foundation (DFG)}
\address{Saarland University, Department of Mathematics, 66041 Saarbr\"ucken, Germany}
\email{speicher@math.uni-sb.de}
\author[S. Yin]{Sheng Yin}
\thanks{S.Y acknowledges with appreciation the support by the National Natural Science Foundation of China (Grant No. 12031004) for his stay at Harbin.}
\address{Research Institute for Mathematical Sciences, Kyoto University, Kyoto 606-8502, Japan}
\email{yin@math.uni-sb.de}
\thanks{We thank Serban Belinschi, Hari Bercovici,  Marek Bo\.zejko, Beno\^it Collins, Jorge Garza-Vargas, and Dan-Virgil Voiculescu for constructive feedback on an earlier version of this work.}
\begin{document}

\maketitle

\begin{abstract}
Consider a tuple $(Y_1,\dots,Y_d)$ of normal operators in a tracial operator algebra setting with prescribed sizes of the eigenspaces for each $Y_i$. We address the question what one can say about the sizes of the eigenspaces for any non-commutative polynomial $P(Y_1,\dots,Y_d)$ in those operators?
We show that for each polynomial $P$ there are unavoidable eigenspaces, which occur in $P(Y_1,\dots,Y_d)$ for any $(Y_1,\dots,Y_d)$ with the prescribed eigenspaces for the marginals.
We will describe this minimal situation both in algebraic terms - where it is given by realizations via matrices over the free skew field and via rank calculations - and in analytic terms - where it is given by freely independent random variables with prescribed atoms in their distributions. 
The fact that the latter situation corresponds to this minimal situation allows to draw many new conclusions about atoms in polynomials of free variables. In particular, we give a complete description of atoms in the free commutator and the free anti-commutator. Furthermore, our results do not only apply to polynomials, but much more general also to non-commutative rational functions. Since many random matrix models become asymptotically free in the large $N$ limit, our results allow us to calculate the location and size of atoms in the asymptotic eigenvalue distribution of polynomials and rational functions in randomly rotated matrices.
\end{abstract}

\section{Introduction}

What can we say about the distribution of a non-commutative selfadjoint polynomial $P$ applied to selfadjoint operators $Y_1,\dots,Y_d$, if we fix the marginal distributions for all the $Y_i$? Of course, there can be many possibilities for those distributions of $P(Y_1,\dots,Y_d)$, depending on the actual choice of $(Y_1,\dots,Y_d)$; but it seems that there are some generic features (depending on $P$ and on the marginal distributions) which are common to all realizations. Progress in free probability theory in the last years supports the idea that actually the specific case where we choose all the variables to be freely independent realizes those generic features.

At the moment, questions addressing densities and their smoothness properties of the distributions are out of reach; however, for dealing with atoms there has been quite some progress in the last couple of years. In particular, it was shown that in the case of free variables $X_1,\dots,X_d$ atoms cannot be created via polynomials; i.e., if each $X_i$ has no atom in its distribution, then for non-constant polynomials $P$, $P(X_1,\dots,X_d)$ also has no atoms. 
We will now push all this much further to the general situation where we allow arbitrary atoms for the individual variables. In such a situation the distribution of $P(Y_1,\dots,Y_d)$ will in general also have atoms. Depending on $P$ and on the atoms for the marginal distributions, there will actually be unavoidable locations and sizes for atoms in the evaluations:  for any operator tuple
$(Y_1,\dots,Y_d)$ with the prescribed atoms for the marginals, the distribution of $P(Y_1,\dots,Y_d)$ will have atoms at those unavoidable locations with the mass of such an atom being at least of  the unavoidable size. One of our main results is that in the case of free variables $X_1,\dots,X_d$, the atoms of $P(X_1,\dots,X_d)$ are exactly at those unavoidable locations with masses given exactly by the unavoidable sizes. We emphasize that the unavoidable locations and sizes do not depend on any information about the non-atomic parts in the marginal distributions. 

One should note that many random matrix ensembles become asymptotically free in the large $N$ limit (for example, independent ensembles whose eigenspaces are rotated randomly against each other). Thus our results say that, in the asymptotic regime, polynomials in randomly rotated matrices have eigenspaces of minimal size; and for many situations we are actually able to calculate the size of those eigenspaces explicitly.

\subsection{Main results}
Let us now present the precise statements of our main results. For undefined notions the reader is referred to Section \ref{sect:preliminaries}. In particular, $\mu_X$ denotes the analytic distribution of the selfadjoint (or more general, normal) operator $X$ (see Def.~\ref{def:anal-distr}). The atomic part of this is denoted by $\mu^p_X$ (see Def. \ref{def:point spectrum}).
Our results do not only hold for non-commutative polynomials, but much more general also for non-commutative rational expressions in the variables. The basics about non-commutative rational expressions and functions will be recalled in Subsection \ref{subsect:free field}.
Note that our input operators are usually normal and bounded, but we do not require this for the outcome after evaluating non-commutative polynomials or rational functions in them. In the general case, the size of an ``atom'' is replaced by the general notion of the rank of an operator; this is defined in Def.~\ref{def:rank}.
For a normal operator $Y$ and a complex number $\lambda$, those concepts are related by
$\mu_Y(\{\lambda\})=\mu_Y^p(\{\lambda\})=1-\rank(\lambda-Y)$. We will use the notation $\mu_Y^p$ also in the general, not necessarily normal, case.

\begin{theorem}\label{thm:atomic}
Let $X_1,\dots,X_d$ and $Y_1,\dots,Y_d$  be normal variables in tracial $W^*$-probability spaces with $X_1,\dots,X_d$ being $*$-free and  such that, for each $1\leq i \leq d$, we have
$$\mu^p_{X_i}\leq\mu^p_{Y_i}.$$
Then, for each rational expression $R$ in $d$ non-commuting variables, $R(X_1,\dots,X_d)$ is well-defined whenever $R(Y_1,\dots,Y_d)$ is well-defined, and we have
$$\rank(R(X_1,\dots,X_d))\geq \rank(R(Y_1,\dots,Y_d));$$
in other words, for each such rational expression $R$ in $d$ non-commuting variables,
$$\mu^p_{R(X_1,\dots,X_d)}\leq  \mu^p_{R(Y_1,\dots,Y_d)}.$$
\end{theorem}
 
Our theorem includes in particular the statement that the location and the size of the unavoidable atoms do not depend on the non-atomic parts in the marginal distributions of the free tuple. We put this down as a separate theorem, as this will be a main step in the proof of the general statement.

\begin{theorem}\label{thm:nonatomic}
Let $(X_1,\dots,X_d)$ and $(Y_1,\dots,Y_d)$  be two $d$-tuples of $*$-free normal variables in tracial $W^*$-probability spaces such that, for each $1\leq i \leq d$, we have
$$\mu^p_{X_i}=\mu^p_{Y_i}.$$
Then, for each rational expression $R$ in $d$ non-commuting variables, $R(X_1,\dots,X_d)$ is well-defined if and only if $R(Y_1,\dots,Y_d)$ is well-defined, and we have
$$\rank(R(X_1,\dots,X_d))= \rank(R(Y_1,\dots,Y_d));$$
in other words, for each such rational expression $R$ in $d$ non-commuting variables,
$$\mu^p_{R(X_1,\dots,X_d)}= \mu^p_{R(Y_1,\dots,Y_d)}.$$
\end{theorem}

Note that this means that for determining the atoms in $R(X_1,\dots,X_d)$, for $X_1,\dots,X_d$ being $*$-free, we can model the non-discrete parts in $\mu_{X_i}$ by any non-atomic distribution; choosing semicircular distributions for this seems to be a good option.

Another consequence of Theorem \ref{thm:nonatomic} is that it allows us to define, for each non-commutative rational function $R$, a
corresponding rational convolution $R^\square$ on discrete sub-probability measures on $\C$, according to
$R^\square(\mu_{X_1}^p,\dots,\mu_{X_d}^p)=\mu^p_{R(X_1,\dots,X_d)}$.

We want to point out that our rational expressions are only allowed to contain the operators, but not their adjoints. The following example shows that an extension of Theorem \ref{thm:atomic} to the $*$-case cannot hold without additional assumptions. Consider, for $d=1$, the $*$-polynomial $R(X)=XX^*$ and take for $X$ a Haar unitary variable $U$. Then the distribution of the normal variable $U$ is the uniform distribution on the unit circle in $\C$ and thus has no atoms; hence the condition $\mu^p_U\leq \mu_Y^p$ is void and we can take for $Y$ any normal operator; in particular, a normal operator for which $R(Y)=YY^*$ has no atom. On the other hand the distribution of $R(U)=UU^*=1$ has an atom of weight 1 at $\lambda=1$, thus violating $\mu^p_{UU^*}\leq \mu^p_{YY^*}$. We expect that under stronger assumptions a version of Theorem \ref{thm:atomic} still holds for $*$-rational functions. We conjecture this to be the case at least when we assume that, for each $i=1,\dots,d$, the $*$-distribution of $X_i$ is the same as the $*$-distribution of $Y_i$. This will be addressed in forthcoming investigations.

The proof of our general results will proceed via modeling the situation of given marginal distributions, in the case of rational weights for the atoms, by a purely algebraic object, namely by matrices over the free field. The free field is the universal field of fractions of non-commutative polynomials.
Those matrices over the free field contain thus all algebraic information common to all realizations by operator tuples $(Y_1,\dots,Y_d)$ with the given atoms in their marginals.
An unavoidable atom at 0 in the distribution of $P(Y_1,\dots,Y_d)$ corresponds in this model to the non-invertibility 
of the matrices. 
By results of \cite{MSY19}, the special operator tuples $(X_1,\dots,X_d)$, where $X_1,\dots,X_d$ are free, yield an analytic model for this algebraic situation. The inner rank on the algebraic side is then the same as the von Neumann operator algebraic rank on the analytic side, which in turn corresponds to the size of the atom in $P(X_1,\dots,X_d)$ at 0. (By shifting $P$ by a constant, we can of course move the location of any atom to 0.) This shows then that, in the case of rational weights for the input atoms, the unavoidable atoms are realized, both in locations and sizes, exactly in the free situation. General approximation arguments allow to extend this to the general situation of arbitrary weights for the atoms.

Moreover, according to \cite[Theorem 5.24]{MSY19}, we know that the von Neumann rank is equal to the inner rank over some skew field whenever an operator tuple $(X_1,\dots,X_d)$ satisfies the strong Atiyah property.
So a natural question is that to what extend the von Neuman rank over an operator tuple $(X_1,\dots,X_d)$ can be identified with an algebraic rank such as inner rank.
In the case of rational weights, our result (Theorem \ref{thm:main}) indeed confirms that von Neumann rank over such tuples is equal to some normalized inner rank.
This rank equality explains the rationality of the von Neumann rank for the free tuples of operators with atoms of rational weights.

We want to emphasize that the restriction to a tracial setting for our operator tuples $(Y_1,\dots,Y_d)$ is crucial for the above arguments. Since
the evaluation map from the free field to algebras is only well-defined for stably finite algebras -- i.e., in our setting for finite von Neumann algebras --  the above modeling with the free field captures only such situations. It is of course an interesting question whether our results extend also to non-tracial situations. In order to address this, new insights are necessary and we will leave this for future work.

The insight that the atoms in the free situation correspond exactly to the unavoidable atoms leads to some very precise statements on atoms for selfadjoint polynomials in free selfadjoint variables $X_1,\dots,X_d$. If we have \textit{one} example of  an operator tuple $(Y_1,\dots,Y_d)$ for which we can calculate the size of an atom for $P(Y_1,\dots,Y_n)$, then this size is a bound from above for the size of the atom for $P(X_1,\dots,X_d)$. In particular,
if we find one realization $(Y_1,\dots,Y_d)$ of the marginal distributions for which we don't have an atom at $\lambda$ for $P(Y_1,\dots,Y_d)$ then also $P(X_1,\dots,X_n)$ cannot have an atom at $\lambda$. This has as direct consequence that the only possible locations for atoms of $P(X_1,\dots,X_d)$ are at the positions $P(\lambda_1,\dots,\lambda_d)$ where, for $i=1,\dots,d$, $\lambda_i$ is the location of an atom for $X_i$. 

For getting lower bounds on the size of atoms for $P(X_1,\dots,X_d)$, we need apriori a bound from below for the size of an atom for $P(Y_1,\dots,Y_d)$ when considering \emph{all} possibilities for $Y_1,\dots,Y_d$ with the prescribed atoms for the marginals. This usually has to rely on general arguments about the size of the intersection of the involved eigenspaces for the atoms. It is a bit surprising that in many cases we can match those general lower bounds with the upper bounds from special realizations and thus get the exact size of the atoms in the free situation.

The following theorem gives a general condition for a polynomial $P$ in two variables ensuring that we can calculate exactly the size of any atom of $P$.
\begin{theorem} \label{diagonalbound2}Let $X$ and $Y$ be free selfadjoint random variables and let $a\in\mathbb{R}$.  Let $P$ be a selfadjoint polynomial which satisfies the following: for all $\lambda,\rho$ such that $P(\lambda,\rho)=a$ then $P(\lambda,\tilde \rho)\neq a$ and $P(\tilde \lambda,\rho)\neq a$ whenever $\lambda\neq \tilde \lambda$ and $\rho\neq \tilde \rho$. 

Then $P(X,Y)$ has an atom at $a$ if and only there are $\lambda$ and $\rho$ such that $X$ has an atom at $\lambda$ of size $s$ and $Y$ has an atom at $\rho$ of size $t$, such that  $r=t+s-1>0$ and $P(\lambda,\rho)=a$. 

Furthermore, if $r>0$, and $s(a)$ and $t(a)$ denote the mass of these (unique) atoms, then the mass at $a$ is given by $s(a)+t(a)-1$.
\end{theorem}

In particular, this theorem  allows a description of the atoms of the free additive convolution and, when  $a\neq0$, also for the free multiplicative convolution, the free commutator and the free anti-commutator. With some additional work we can achieve a full description of the atoms for all these cases.

\subsection{How to model free variables}

Here we want to give the main ideas -- mostly in terms of a concrete example -- involved in the the proof of our results. To be definite, we will in the following talk about  atoms of distributions of (normal) operators. Of course, all of them have to live in a tracial von Neumann algebra setting.

A universal algebraic model for an operator $X$ without atoms in its distribution is given by a formal variable $x$. Rational functions of $x$ model then the corresponding rational function of any such operator. In the case of several operators the free (skew) field, generated by $d$ formal non-commutative variables $x_1,\dots,x_d$, is the universal algebraic model for any $d$-tuple of,
not necessarily commuting, operators without atoms in their distributions.

If we want models for operators which have atoms in their distributions, then we can model this via matrices over the free field. For example, if we want to model an operator $X$ with distribution $\frac 13\delta_1+\frac 13\delta_2+\frac 13\tilde\mu$, where $\tilde\mu$ is a probability measure without atoms, then the following $3\times 3$-matrix over the free field (generated by the formal variable $x_1$) will do the job: 
$$
A=\begin{pmatrix}
1&0&0\\
0&2&0\\
0&0&x_1
\end{pmatrix}.
$$

The weight for an atom at $\lambda$ corresponds in this algebraic setting to
the inner rank (a purely algebraic notion); $\rank(A-\lambda)=k$ corresponds to an atom at $\lambda$ of weight $1-k/3$. In this way, it is clear that with $n\times n$-matrices over the free field we can thus model only atoms whose weights are multiples of $1/n$.

Now assume that we want to model two operators $X$ and $Y$ with prescribed distributions of their atoms. For example, in addition to having an operator $X$ with distribution $\frac 13\delta_1+\frac 13\delta_2+\frac 13\tilde\mu$ as before we would also like to have another operator $Y$ with only one atom, of size $1/3$ at 0. This latter operator could be modeled by
$$\tilde B=\begin{pmatrix}
0&0&0\\
0&x_2&0\\
0&0&x_3
\end{pmatrix}.
$$
Even though we have used different formal variables to model this one, it is clear that $A$ and $\tilde B$, living in $3\times 3$-matrices over the free field generated by the variables $x_1,x_2,x_3$ are not universal models for pairs of operators with the prescribed atoms in their distributions. What we have modeled here are operators where there exist definite relations between their eigenspaces. In order to get rid of those specific relations we should rotate the two matrices against each other, again in a generic way. For this we replace $\tilde B$ by the matrix
$$
B=U\tilde BU^{-1}=
\begin{pmatrix}
u_{11}&u_{12}&u_{13}\\
u_{21}&u_{22}&u_{23}\\
u_{31}&u_{32}&u_{33}
\end{pmatrix}
\begin{pmatrix}
0&0&0\\
0&x_2&0\\
0&0&x_3
\end{pmatrix}
\begin{pmatrix}
u_{11}&u_{12}&u_{13}\\
u_{21}&u_{22}&u_{23}\\
u_{31}&u_{32}&u_{33}
\end{pmatrix}^{-1}.
$$
The matrix $U$ 
has here 9 other formal variables as entries, and the matrix $U^{-1}$ is calculated over the free field generated by $u_{11},\dots,u_{33}$, thus its entries are rational functions in the $u_{11},\dots,u_{33}$; note that the matrix $U$ is invertible over the free field. The matrices $A$ and $B$, living in matrices over the free field generated by the 12 variables $x_1,x_2,x_3,u_{11},\dots,u_{33}$, are our wanted universal algebraic model for two operators with the prescribed atomic distributions. Of course, we have to show that any concrete pair of such operators is a specialization of $A$ and $B$. This is essentially the diagonalization of the given operators with respect to the projections onto their eigenspaces, but since we have to stay within our tracial von Neumann algebra setting and also to deal with not necessarily rational weights of the atoms the technical details are a bit more involved.

If we are now interested in the possible atoms appearing in rational functions $R(X,Y)$ for arbitrary operators $X$ and $Y$ with the prescribed weights for their atoms, then doing this calculation for $R(A,B)$ over the free field will yield exactly the unavoidable atoms and weights which are common to all $R(X,Y)$. In particular, in order that $R(X,Y)$ is well-defined (as an unbounded operator), $R(A,B)$ must be well-defined over the free field. Clearly, not all rational functions make sense when we have atoms; for example, $A^{-1}$ makes sense in the free field, but $B^{-1}$ does not. Of course, there are usually choices for $X$ and $Y$ for which $R(X,Y)$ is not defined, even though $R(A,B)$ makes sense.

From the above representation as $3\times 3$-matrices over the free field, it is also clear that in this case, where the marginal atoms are of weight $1/3$, the atoms of any meaningful rational function in $A$ and $B$ can also have only weights which are multiples of $1/3$. This will in general make a relation to the so-called Atiyah property.

Up to this point we have given universal algebraic models for operator tuples with prescribed atoms in the marginals, relying on the fact that the free field yields via specialization any possible realisation for such situations. This reduction to the free field shows in particular that the non-atomic parts in the prescribed distributions do not come into the calculations. All non-atomic distributions are in the free field just represented by formal variables and thus do not influence the search for the unavoidable atoms in rational functions. We can, and actually will (see Proposition \ref{Matrix_form:general_case}), even put part of the prescribed atoms into those formal variables which allows us to approximate general weights for atoms via rational weights in the matrices over the free field.

Doing concrete rank calculations over the free field, in order to find the unavoidable atoms, is however not easy  -- in particular, as we also have to deal with terms like $U^{-1}$. Note that the occurrence of $U^{-1}$ means that even if we are only interested in polynomials in our operators, in our calculations in the free field we have nevertheless to deal  with rational functions in the $u_{ij}$.

 We will now use another, much more surprising, connection between algebraic and analytic models. Namely, there are actually specific analytic situations which are not just a specialization, but which are isomorphic to the free field. This means that the rank calculations for the weights of unavoidable atoms in the free field can also be done in those analytic settings, using analytic tools. As those analytic settings correspond to free random variables, we have actually quite some tools, provided by free probability theory, for analytic calculations.

The starting point for this correspondence between the algebraic and the analytic side is the characterization from \cite{MSY19} that the free field generated by formal variables $x_1,\dots,x_d$ is isomorphic to the division closure (in the affiliated unbounded operators) of operators $X_1,\dots,X_d$ if and only if those operators maximize a specific quantity $\Delta(X_1,\dots,X_d)$. Such a maximal $\Delta$ is for example achieved if the $X_1,\dots,X_d$ are free and each $X_i$ has no atoms in its distribution. Whereas $x_1,x_2,x_3$ in the above example can be chosen as such free variables (with some arbitrary distribution without atoms), we also have to deal with the matrix $U$. 
We could just replace all entries of $U$ by free operators with, say, semicircular distribution and then try to calculate the distribution of those matrices, using operator-valued free probability tools. However, this will get out of hand soon, in particular, for big sizes of the considered matrices and it is not clear whether this brings a clear advantage over rank calculations in the free field. What we do instead is to stay in the scalar-valued setting and make all relevant operators there free. For this we need that the matrix $U$ itself has to be a Haar unitary matrix which also should be free from $A$ and $\tilde B$. This can be achieved by taking the entries $u_{ij}$ as given by a free compression of a Haar unitary by matrix units. It takes then some effort to show that in this case $\Delta(u_{11},\dots,u_{33})$ is also maximal and that furthermore the $u_{ij}$ are also free from $x_1,x_2,x_3$. Having proved all this gives then that $\Delta(x_1,x_2,x_3,u_{11},\dots,u_{33})$ is also maximal, and thus
we can use this analytic setting to calculate the unavoidable atoms. As it is set up now, we have that $U$ is a Haar unitary which is $*$-free from $\tilde B$ and from $A$. But this means that $A$ and $B$ themselves are free and each of them has the prescribed distribution for the atoms (and where we have the freedom to fill in for the non-atomic parts any non-atomic distribution, for example, semicirculars). So in the end we see that if we choose our operators to be free and with the prescribed distribution, then calculating atoms for any rational function in them will give the unavoidable atoms for the prescribed atoms of the marginals. 

Note that we can now also turn this connection around and use it to get some information about atoms of polynomials or rational functions of free variables. In particular, we see from the fact that in the free field the non-atomic parts don't play a role this is now also true for taking rational functions in free variables $X_1,\dots,X_d$. The atoms for such rational functions $R(X_1,\dots,R_d)$ depend only on the atoms of the operators $X_i$.

Another advantage of knowing that the situation of free random variables corresponds to the ``minimal'' situation concerning atoms is that this allows to get results on atoms for polynomial free convolutions. By comparing the free situation with any other suitably chosen example, we can exclude atoms for the free case if they do not show up in the considered example.

\subsection{Organization of the article}
There are six more sections after the introduction and they are organised as follows.
In Section \ref{sect:preliminaries} we introduce the necessary notions with some basic relevant facts used in the later sections.

In Section \ref{sect:matrix models} we establish several important lemmas and propositions concerning non-commutative random variables.
More precisely, it contains two subsections: in the first one we bring non-commutative random variables into some matrix forms that will be used for our main results in Section \ref{sec:calculation} and \ref{sec:comparison}; in the second one, we prove that the free compressions of free Haar unitaries have maximal $\Delta$ and thus we can identify them with formal variables.

In Section \ref{sec:calculation}, we first establish algebraic models for $\ast$-free normal random variables with atoms of rational weights.
Some consequences of our first main result are given in the last two subsections of this section.
Subsection \ref{subsect:Atiyah} concerns the link between our result and the strong Atiyah conjecture for groups and Atiyah properties for non-commutative variables.
In Subsection \ref{subsect:approximation of free variables} we give the proof of Theorem \ref{thm:nonatomic} with some applications.

In Section \ref{sec:comparison} we compare the rank and atoms for matrices in polynomials as well as rational functions between free variables and general random variables.
The first subsection is devoted to the proof of Theorem \ref{thm:atomic}.
As a first consequence, the universality of the rational closure of free variables is stated in Subsection \ref{subsect:universality}.
Then we show, in Subsection \ref{subsect:matrix approximation}, that (restricting to polynomials) the minimal situation realized by free variables actually can be obtained via matrices.

Section \ref{sect:application} collects applications and examples that can be studied with the help of our main results.
Comparing to specific examples we obtain general  criteria to bound the size of the atoms of various general polynomials.
By using these methods, in Section \ref{sect:commutator},  we  give a complete description of atoms for the free commutator and the free anti-commutator and provide an application on the spectrum of universal covers.

\section{Preliminaries}\label{sect:preliminaries}

\subsection{Tracial $W^*$-probability spaces, affiliated unbounded operators, analytic and point spectrum distributions}
In this section we will define the basic non-commutative probabilistic notions and recall some relevant facts around them.

\begin{definition}
A \emph{tracial $W^*$-probability space} $(\M,\tau)$ consists of
a von Neumann algebra $\M$ and a faithful and normal trace $\tau:\M\to\C$. Elements in $\M$ are often addressed as \emph{(non-commutative) random variables}.
\end{definition}

In the setting of a tracial $W^{\ast}$-probability space $\left(\M,\tau\right)$, we will usually be interested in the closed and densely defined linear operators affiliated with $\M$.
The set of all such affiliated unbounded operators is known to be a $\ast$-algebra containing $\M$. In a sense it is the non-commutative analogue of the set of all measurable functions of the elements from $\M$; correspondingly, we will denote it by $L^0(\M)$.
Moreover, the algebra $L^0(\M)$ is stably finite (see, for example, \cite[Lemma 5.3]{MSY19}), which will allow us to consider the evaluations of rational functions in elements from $L^0(\M)$.

Recall that an operator $X$ is called \emph{normal} if we have $XX^*=X^*X$. Normal operators can be identified with classical (complex-valued) random variables, and thus we associate to each normal operator $X$ in a $W^\ast$-probability space $(\M,\tau)$ its distribution, which is a probability measure on $\C$ with compact support, and which is determined by the $*$-moments of $X$. In the tracial setting we can also allow affiliated unbounded operators; for those the appearing probability measures do not need to have compact support or even finite moments. 
We record these facts here as the following definition.

\begin{definition}\label{def:anal-distr}
Let $(\M,\tau)$ be a
tracial $W^*$-probability space. 
\begin{enumerate}
\item For a normal operator $X\in \M$, the \emph{analytic distribution} of $X$ with respect to $\tau$
is defined as the unique Borel probability measure $\mu_X$ on $\C$ that satisfies 
$$\int_\C z^k \overline{z}^l\, d\mu_X(z) = \tau(X^k (X^\ast)^l)\qquad\text{for all $k,l\in\mathbb{N}_0$.}$$
\item For a normal unbounded operator $Y\in L^0(\M)$ its \emph{analytic distribution $\mu_Y$} is defined as the Borel probability measure on $\C$ that is obtained by $\mu_Y = \tau \circ E_Y$, where $E_Y$ denotes the resolution of identity associated to $Y$; this is well-defined, since $E_Y$ takes its values in $\M$ as $Y$ is a normal operator affiliated to $\M$; this extends the concept of analytic distributions for bounded normal operators from (1).
\end{enumerate}
\end{definition}

Whereas the analytic distribution only makes sense for normal operators, the atomic part of such a distribution has a much more algebraic meaning and can also be extended to general operators in our setting. This is made precise in the following definitions. 

\begin{definition}
A \emph{discrete sub-probability measure} on $\C$ is of the form
$$\sum_{\lambda\in\C} \alpha_\lambda \delta_\lambda, $$
where we have that $0\leq \alpha_\lambda\leq 1$ for all $\lambda$ and that
$\sum_{\lambda\in\C} \alpha_\lambda\leq 1$. Hence in particular, only countably many $\alpha_\lambda$ are different from zero. The case that all $\alpha_\lambda$ are equal to zero is also allowed. The corresponding discrete sub-probability measure will then be denoted by 0.
\end{definition}

\begin{definition}\label{def:point spectrum}
Let $(\M,\tau)$ be a
tracial $W^*$-probability space.
Let $X\in L^0(\M)$ be given.
Its \emph{point spectrum} is defined as
\[
\sigma_p(X):=\{\lambda\in\C:\ker(\lambda-X)\neq\{0\}\}.
\]
Furthermore, we define the \emph{point spectrum distribution} of $X$ as the discrete sub-probability measure
$$\mu_X^p:=\sum_{\lambda\in\sigma_p(X)}\tau(p_{\ker(\lambda-X)})\delta_\lambda,
$$
where $p_{\ker(Y)}$ stands for the orthogonal projection onto the kernel of an operator $Y\in L^0(\M)$.
\end{definition}

Note that each $\lambda\in\sigma_p(X)$ is an \emph{eigenvalue} of $X$, in the sense that there exists a nonzero eigenvector $\xi\in L^2(\M)$ such that $X\xi=\lambda\xi$.
Hence we will also call the subspace $\ker(\lambda-X)$ of $L^2(\M)$ an \emph{eigenspace} of $X$. 
For a given operator $X\in\M$ its point spectrum is the same as its spectrum in $L^0(\M)$ since an operator is invertible in $L^0(M)$ if and only if it has trivial kernel; of course, the spectrum with respect to $\M$ is in general larger.
Therefore, if two operators $X$ and $Y$ generate isomorphic von Neumann algebras, then $\mu_X^p=\mu_Y^p$.
In particular, this is true for operators $X$ and $Y$ with the same $\ast$-distribution.

The weight $\tau(p_{\ker(\lambda-X)})$ for the eigenvalue $\lambda$ in the definition of $\mu_X^p$ is a measure, as a normalized kind of dimension function, for the size of the corresponding eigenspace;
note that in the setting of type $II$ von Neumann algebras those eigenspaces are infinite-dimensional and their dimensions are measured with respect to the size of the whole Hilbert space. It is the basic ingredient in the setting of type $II$ von Neumann algebras that we have a continuous dimension, taking on all values in the interval $[0,1]$, for the size of infinite-dimensional subspaces of the ambient Hilbert space.

Moreover, if $X$ in $L^0(\M)$ is normal, then $\lambda\in\sigma_p(X)$ if and only if $\lambda$ is an \emph{atom} of the analytic distribution $\mu_X$, i.e., if $\mu_X(\{\lambda\})>0$; in which case we have
$\mu_X(\{\lambda\})=\tau(p_{\ker(\lambda-X)})$.
Therefore, for a normal operator $X$, we will also say that $\lambda\in\sigma_p(X)$ is an \emph{atom of $X$}.
Clearly, for a normal $X$ we can decompose the measure $\mu_X$ as the sum of an atomic part (which is a discrete measure supported on $\sigma_p(X)$ and which coincides with the point spectrum distribution $\mu_X^p$ of $X$) and a non-atomic part (which is a measure which has no atoms).

If $X\in \M$ is an operator which is not normal, a canonical probability measure $\mu_X$ called the \emph{Brown measure} can still be constructed on the spectrum of $X$, generalizing the definition of the analytic distribution of normal variables. Thanks to \cite{Haagerup2009}, we always have the inequality $\mu^p_X\leq \mu_X$. In particular, each $\lambda\in\sigma_p(X)$ is an atom of the Brown measure $\mu_X$. However, this inequality can be strict, and the weight of the atom for the Brown measure can be larger than the weight in the point spectrum distribution.

\subsection{Free random variables}

The situation corresponding to unavoidable atoms will be given by free random variables. Free probability theory and its fundamental notion of freeness or free independence are by now well-established non-commutative analogues of classical probability theory and its notion of independence. We refer to the standard references \cite{MS17,NS06,VDN92} for the basic notions and results from free probability theory. For our constructions only a few properties of free random variables will be important. We will not explicitly address random matrices in the following, but at least we want to point out that often random matrix models become freely independent in the large size limit, thus our results have quite some implications for atoms in the asymptotic eigenvalue distribution  of polynomials or rational functions in random matrices.

The notion of freeness refers to a setting of a non-commutative probability space $(\M,\tau)$ and given subalgebras $A_i$.

\begin{definition}
Let $(\M,\tau)$ be a $W^\ast$-probability space and let $I$ be an index set.

1)  Let, for each $i\in I$, $\A_i\subset \M$, be a unital subalgebra.
The subalgebras $(\A_i)_{i\in I}$ are called \emph{free} or \emph{freely independent}, 
if $\tau(a_1\cdots
a_k)=0$ whenever we have: $k$ is a positive integer; $a_j\in\A_{i(j)}$ (with $i(j)\in
I$) for all $j=1,\dots,k$; $\tau(a_j)=0$ for all $j=1,\dots,k$; and neighboring elements
are from different subalgebras, i.e., $i(1)\not=i(2)$, $i(2)\not= i(3)$,$\dots$,
$i(k-1)\not=i(k)$.

2) Let, for each $i\in I$, $X_i\in\M$. The random variables $(X_i)_{i\in I}$ are called
\emph{free} or \emph{freely independent}, if their generated unital subalgebras are free,
i.e., if $(\A_i)_{i\in I}$ are free, where, for each $i\in I$, $\A_i$ is the unital
subalgebra of $\M$ which is generated by $X_i$.

3) Let, for each $i\in I$, $X_i\in\M$. The random variables $(X_i)_{i\in I}$ are called
\emph{$*$-free} or \emph{$*$-freely independent}, if their generated unital $*$-subalgebras are free,
i.e., if $(\A_i)_{i\in I}$ are free, where, for each $i\in I$, $\A_i$ is the unital
$*$-subalgebra of $\M$ which is generated by $X_i$.
\end{definition}

For selfadjoint $X_i$, there is no difference between freeness and $*$-freeness. In the case of general normal $X_i$, we will usually consider $*$-free variables.

Let us point out that the notion of freeness, which looks quite algebraic, goes well with von Neumann algebras. Namely, if unital subalgebras $\A_i$ are free in $(\M,\tau)$, then also their generated von Neumann subalgebras are free. 

A crucial fact on freeness is that if we have a tuple $X=(X_1,\dots,X_d)$ of free selfadjoint operators, then for any selfadjoint non-commutative polynomial $P$, the analytic distribution of the evaluation $P(X):=P(X_1,\dots,X_d)$ (which is again a selfadjoint operator) is uniquely determined by the marginal analytic distributions $\mu_{X_i}$.
Hence, for each selfadjoint polynomial $P$ in
$d$ non-commuting variables we have a corresponding polynomial free convolution $P^\square$, mapping $d$-tuples of probability measures on $\R$ to another probability measure on $\R$, corresponding to
$P^\square(\mu_{X_1},\dots,\mu_{X_d})=\mu_{P(X_1,\dots,X_d)}$.
However, there are no explicit descriptions of this functional relationship and getting qualitative and quantitative properties of $\mu_{P(X)}$ out of corresponding information for the marginal distributions is one of the big problems. Up to now such concrete descriptions were only known for the case of the sum and the product of free random variables (corresponding to the operations of additive and multiplicative, respectively, free convolution). On the conceptual side, there was quite some progress by relating polynomials in free independent variables via linearization to the operator-valued version of additive convolution. This general approach goes back to the work \cite{BMS17}, and was brought to fruition for the problem of dealing with atoms in the recent work \cite{BBL2021}. However, it is hard to get explicit results on atoms from those approaches.

Our starting point here to address general polynomials (and also rational functions) in free variables is the crucial recent insight \cite{MSY19} that in the case of marginal distributions without atoms, rational functions in such free variables can be modelled by an abstract algebraic object, the so-called free field. In the next subsection we will recall the basic properties of this object. By realizing marginal distributions with atoms (of rational weights) via matrices over the free field we will then be able to get our results on atoms of functions of free variables, not only for polynomials, but also for non-commutative rational functions.

Our results show that the polynomial convolutions $P^\square$, as well as their generalizations to rational functions, can actually also be considered as mappings within the class of discrete sub-probability measures; and in this setting they also make sense not just for selfadjoint polynomials and measures on $\R$, but actually for arbitrary rational functions and discrete sub-probability measures on $\C$. For each non-commutative rational function $R$, we have a
corresponding rational convolution $R^\square$ on discrete sub-probability measures on $\C$, according to
$R^\square(\mu_{X_1}^p,\dots,\mu_{X_d}^p)=\mu^p_{R(X_1,\dots,X_d)}$.

Note that it is crucial for our arguments that the operators $X_1,\dots,X_d$ are normal, even though the evaluation $R(X_1,\dots,X_d)$ does not have to be so. Indeed, if we start from non-normal operators it will be hard to find polynomials which preserve normality.

\subsection{The category of $\Cx$-algebras, rational and division closures}\label{sec:category of Cx-algebras}
We denote by $\Cx$ the algebra of non-commutative polynomials in the indeterminates $x_1,\dots,x_d$.

In this article, we are interested in $d$-tuples of elements in various algebras.
It is natural to represent such a tuple as a $\Cx$-algebra in the category of $\Cx$-algebras.
The following definition is based on the so-called \emph{$\cR$-ring}, where $\cR$ is another ring with some homomorphism (see, for example, \cite[Section 7.2]{Coh06}).
We change the name accordingly since we are mainly interested in algebras instead of rings.

\begin{definition}\label{def:subhomomorphism}
A $\Cx$-\emph{algebra} is a unital complex algebra $\A$ together with a homomorphism $\phi:\Cx\rightarrow\A$.
We denote
\[
\Sigma_\phi:=\coprod_{n=1}^\infty\{A\in M_n(\Cx):\phi(A)\text{ is invertible in }M_n(\A)\},
\]
where we also denote by $\varphi$ the matricial amplification $\varphi((a_{ij})_{i,j=1}^n):=(\varphi(a_{ij}))_{i,j=1}^n$ for any homomorphism $\varphi$.

Let $\A$ and $\B$ be $\Cx$-algebras with homomorphisms $\phi_\A:\Cx\rightarrow\A$ and $\phi_\B:\Cx\rightarrow\B$.
\begin{enumerate}
\item A \emph{$\Cx$-algebra homomorphism} $f:\A\rightarrow\B$ is a homomorphism such that $f\circ \phi_\A=\phi_\B$.
\item A \emph{($\Cx$-algebra) subhomomorphism} is a $\Cx$-algebra homomorphism $f:\A_f\rightarrow\B$, where $\A_f$ is a $\Cx$-subalgebra of $\A$ such that
  \begin{itemize}
  \item its homomorphism $\phi_f:\Cx\rightarrow\A_f$ agrees with $\phi_\A$
  \item\label{it:invertiblity}
 and $\Sigma_{\phi_\B}\subset \Sigma_{\phi_f}$. 
  \end{itemize}
\item Two subhomomorphisms from $\A$ to $\B$ are called \emph{equivalent} if they agree on an $\Cx$-subalgebra $\A_0$ of $\A$ such that their common restriction to $\A_0$ is again a subhomomorphism.
\item A \emph{specialization} from $\A$ to $\B$ is an equivalence class of subhomomorphisms from $\A$ to $\B$.
\end{enumerate}
\end{definition}

Note that the definition of subhomomophism (Item \eqref{it:invertiblity} of Definition \ref{def:subhomomorphism}) is slightly different from the one in \cite[Section 7.2]{Coh06} since subhomomorphisms therein are considered for epic skew fields instead of algebras (see Section~\ref{subsect:free field}). However, the two definitions coincide if one restrict oneself to epic skew fields: if $\K$ and $\cL$ are two epic skew fields, a subhomorphism $f$ from $\K$ to $\cL$ with $f:\K_f\rightarrow\cL$ satisfies that an element $x\in\K_f$ such that $x\notin\ker f$ is invertible in $\K_f$.

In the sequel, for a unital complex algebra $\A$, the $\Cx$-algebra structure usually comes from the evaluation of some tuple $X=(X_1,\dots,X_d)$ over $\A$.
Namely, we usually take the homomorphism $\phi:\Cx\rightarrow\A$ as $\phi(P)=P(X)$, $\forall P\in\Cx$.
In particular, we also write $\Sigma_X:=\Sigma_\phi$ in such a case.

The following construction is an important example of $\Cx$-algebra that will be considered later.

\begin{definition}
Let $X=(X_1,\dots,X_d)$ be a tuple of elements in a unital complex algebra $\A$.
The \emph{rational closure of $X$ in $\A$} is the set, denoted by $\A_X$, of all entries of inverses of evaluations of matrices in $\Sigma_X$ at $X$.
\end{definition}

For a tuple $X=(X_1,\dots,X_d)$ over a unital complex algebra $\A$, we denote by $\C\langle X\rangle$ the subalgebra of $\A$ which is generated by $X_1,\dots,X_d$.
A rational closure $\A_X$ of $X$ in $\A$ is always a subalgebra of $\A$ and contains $\C\langle X\rangle$ (see \cite[Section 7.1]{Coh06} for more details).

Let $X$ and $Y$ be two tuples over unital complex algebras $\A$ and $\B$ respectively.
If $f:\A_f\rightarrow\B_Y$ is a subhomormorphism from a $\A_X$ to $\B_Y$ with the homomorphism $\phi_f:\Cx\rightarrow\A_f$, then $f$ is actually surjective due to the property $\Sigma_Y\subset\Sigma_{\phi_f}$.
Hence $f$ naturally induce an isomorphism from $\A_f/\ker(f)$ onto $\B_Y$.
Moreover, if $\B_Y=\A_X$, then $\A_f=\A_X$ and $f$ reduces to the identity map, as follows.

\begin{prop}
\label{prop:isomorphism}Let $X=(X_1,\dots,X_d)$ be a tuple over a unital complex algebra $\A$. The identity map is the only one specialization from the rational closure $\A_X$ to itself.
\end{prop}
\begin{proof}Let $\phi:\Cx\to\A_X$ be given by the evaluation map at $X$ inside $\A_X$. Let $f:\A_f\to \A_X$ be a subhomorphism from $\A_X$ to itself, with $\A_f\subset \A_X$ and $\phi_f:\Cx\to\A_X$ given by the evaluation map at $X$ inside $\A_X$.

We have $\Sigma_{\phi}\subset \Sigma_{\phi_f}$ because of the subhomorphism $f:\A_f\to \A_X$, and as a consequence, all entries of evaluations of inverses of matrices in $\Sigma_{\phi}$ at $X$ are in fact entries of evaluations of inverses of matrices in $\Sigma_{\phi_f}$, and by consequence belong to $\A_f$.
As a consequence, $\A_X\subset \A_f$ which implies that  $\A_X=\A_f$.

The identity map and the homomorphism $f:\A_X\to \A_X$ coincides on $\mathbb{C}\langle X\rangle$ so they also coincide on the rational closure $\A_X$ of $X$. Indeed, if $c\in \A_X$, there exists some $m\times m$ matrix $A\in\Sigma_{X}$ and some integers $1\leq i,j\leq m$ such that $c$ is the $(i,j)$-entry of $A(X)^{-1}$.
Now, we have
$$I_m=f(I_m)=f(A(X)A(X)^{-1})=f(A(X))f(A(X)^{-1})=A(X)f(A(X)^{-1}),$$
which means that
$f(A(X)^{-1})=A(X)^{-1},$
and $f(c)=c$.
\end{proof}

Moreover, if $f:\A\rightarrow\B$ is an $\Cx$-isomorphism for two $\Cx$-algebras $\A$ and $\B$, then for any tuple $X$ over $\A$ and $Y$ over $\B$, $f$ naturally induce an isomorphism from $\A_X$ onto $\B_Y$.
This allows us to give a criterion for the isomorphism of rational closures that are given by tuples over tracial $W^\ast$-probability spaces.
That is, if two tuples have the same $\ast$-distribution, then they generate isomorphic von Neuamnn algebras and thus have isomorphic rational closures.
We record this observation as the following remark.

\begin{remark}\label{rem:rational closure by dist}
Let $(\M_1,\tau_1)$ and $(\M_2,\tau_2)$ be two tracial $W^\ast$-probability spaces.
Suppose that $X=(X_1,\dots,X_d)\in \M_1^d$ and $Y=(Y_1,\dots,Y_d)\in \M_2^d$ are two tuples that have the same $\ast$-distribution.
Then the rational closures $L^0(\M_1)_{X}$ and $L^0(\M_2)_{Y}$ are isomorphic.
\end{remark}

A rational closure $\A_X$ has a property called \emph{division closed}, i.e., if $a\in\A_X$ is invertible in $\A$, then $a^{-1}\in\A_X$.
The smallest subalgebra with such a property is another important example of $\Cx$-algebras.

\begin{definition}
Let $X=(X_1,\dots,X_d)$ be a tuple over a unital complex algebra $\A$. 
The \emph{division closure} of $X$ in $\A$ is the smallest division closed subalgebra of $\A$ containing $\C\langle X\rangle$.
We denote it by $\C\plangle X\prangle$.
\end{definition}

A priori the division closure $\C\plangle X\prangle$ is contained in the rational closure $\A_X$ for a given tuple $X$.
However, they may agree under certain circumstance, for example, when one of them is a skew field (see, for example, \cite[Proposition 4.9]{MSY19}).

As $\Cx$-subalgebras inside rational closures, the isomorphism of division closures follows from that of rational closures.

\begin{remark}\label{rem:isomorphic division closure}
Let $\A_1$ and $\A_2$ be two unital complex algebras with the division closure $\D_i:=\C\plangle X_i\prangle$ for some tuple $X_i$ over $\A_i$ , $i=1,2$.
If $f:\A_1\rightarrow\A_2$ is an $\Cx$-isomorphism, then $f(D_1)=D_2$.
\end{remark}

\begin{proof}
It's not difficult to see that $f$ preserves the division closedness.
So $f(\D_1)$ is division closed, and thus $\D_2\subset f(\D_1)$.
Hence $f^{-1}(\D_2)\subset \D_1$.
Note that $f^{-1}(\D_2)$ is division closed, we have $\D_1\subset f^{-1}(\D_2)$.
Therefore, $\D_1\subset f^{-1}(\D_2)\subset \D_1$, from which the conclusion follows.
\end{proof}

\subsection{Free field}\label{subsect:free field}

In this subsection, we want to introduce the free field that is constituted by non-commutative rational functions.
We will also introduce some closely related notions such as rational expressions and linear representations.

Let us point out that the free field is actually a non-commutative algebra, hence ``free skew field'' might be a more precise name (at least for people coming from a commutative setting).

It is well-known that commutative polynomials can be embedded uniquely into the field of fractions, which is a field consisting of the quotients of polynomials.
However, the non-commutative analogue of the field of fractions is highly non-trivial.
In order to introduce this object, let us first give some necessary notations.

\begin{definition}
A $\Cx$-algebra $\K$ is called a $\Cx$-\emph{field} if $\K$ is a skew field.

Let $\K$ be a $\Cx$-\emph{field} with $\phi:\Cx\rightarrow\K$. 
\begin{enumerate}
\item We call $\K$ \emph{epic} if $\K$ is generated by the image $\phi(\Cx)$, i.e., there is no proper subfield of $\K$ containing $\phi(\Cx)$.
\item An epic $\Cx$-field $\K$ is called \emph{field of fractions of} $\Cx$ if the homomorphism $\phi$ is injective.
\end{enumerate}
\end{definition}

One may expect that, like in the commutative case, a (skew) field of fractions of $\Cx$ (if it exists) is unique.
However, it turns out that there exist fields of fractions of $\Cx$ which are not isomorphic (see, for example, \cite[Exercise 7.2.13]{Coh06}).
To understand the non-commutative counterpart of commutative rational functions, we need some universal property for $\Cx$-fields with the help of subhomomorphisms.

\begin{definition}
An epic $\Cx$-field $\U$ is called a \emph{universal $\Cx$-field} if for any epic $\Cx$-field $\K$ there is a unique specialization from $\U$ to $\K$.
If $\U$ is in addition a field of fractions of $\Cx$, then we call $\U$ the \emph{universal field of fractions of} $\Cx$.
\end{definition}

There indeed exists a universal field of fractions of $\Cx$, which we call the \emph{free field} and denote it by $\C\plangle x_1,\dots,x_d\prangle$.
We won't go further into the details of the constructions of the free field and refer the interested reader to \cite[Chapter 7]{Coh06} for more details.
An element in the free field is called a \emph{non-commutative rational function}.

A non-commutative rational function can be represented by \emph{non-commutative rational expressions}.
They are syntactically valid combinations of $\C$ and symbols $x_1,\ldots,x_d$ with $+$, $\cdot$, ${}^{-1}$, and $()$, which are respectively corresponding to addition, multiplication, taking inverse, and ordering these operations.
Non-commutative rational expressions are formal objects or expressions rather than functions.
For example, $x_1 + (-1) \cdot x_1$ and $0$ are two distinct non-commutative rational expressions though they represent the same non-commutative rational function.
We refer here to \cite{HW15} for a rigorous definition of non-commutative rational expressions that is based on the graph theory and by which arithmetic operations for rational expressions can be interpreted as operations on graphs.
Moreover, in \cite{Ami66}, Amitsur constructed the free field $\C\plangle x_1,\dots,x_d\prangle$ as equivalence classes of non-commutative rational expressions (see also \cite[Section 2]{K-VV12}).
For simplicity, we usually omit "non-commutative" and call them rational expressions unless otherwise indicated.
The same convention will also used for non-commutative polynomials and non-commutative rational functions.

Now we can recursively define the domain of rational expressions.
\begin{definition}
Let $\A$ be a unital complex algebra. For each rational expression $R$ in indeterminates $x_1,\dots,x_d$, we define its $\A$\emph{-domain} $\dom_{\A}(R)\subseteq\A^d$ together with its \emph{evaluation} $R(X)$ for any tuple $X=(X_1,\dots,X_d)\in\dom_{\A}(r)$ by the following rules:
\begin{enumerate}
\item If $R=\lambda$ for any $\lambda\in\C$, we define $\dom_{\A}(R)=\A^{d}$ and $R(X)=\lambda$.
\item If $R=x_i$ for some $1\leq i\leq d$, we define $\dom_{\A}(R)=\A^{d}$ and $R(X)=X_{i}$;
\item For two rational expressions $R_{1}$, $R_{2}$, we define
\[\dom_{\A}(R_{1}\cdot R_{2})=\dom_{\A}(R_{1}+R_{2})=\dom_{\A}(R_{1})\cap\dom_{\A}(R_{2})\]
and
\begin{align*}
(R_1\cdot R_2)(X)=R_1(X)\cdot R_2(X),\\
(R_1+R_2)(X)=R_1(X)+R_2(X).
\end{align*}
\item For a rational expression $R$, we define
\[
\dom_{\A}(R^{-1})=\{X\in\dom_{\A}(R): R(X)\text{ is invertible in }\A\}
\]
and
\[R^{-1}(X)=(R(X))^{-1}.\]
\end{enumerate}
\end{definition}

A very useful tool for studying rational expressions or functions is the so-called linearization or linear representation.
It is well-known and used in many different theories of mathematics (see, for example, \cite{CR99, HMS18}).
In order to introduce this tool, let us first define linear matrices.

\begin{definition}
A matrix $A$ in $M_{m}(\C\langle x_{1},\dots,x_{d}\rangle)$ (which we also identify with $M_m(\C)\otimes\C\langle x_{1},\dots,x_{d}\rangle$) is called \emph{linear} if it is of the form 
\[
A=A_{0}\otimes1+A_{1}\otimes x_{1}+\cdots+A_{d}\otimes x_{d},
\]
where $A_{0},A_{1}\dots,A_{d}$ are $m\times m$ matrices over $\C$.
Note that many authors also call such a matrix \emph{affine linear} since it allows a constant term. 
\end{definition}

The following notion of linear representation of rational expressions is taken from \cite{HMS18} and will be important in the sequel.

\begin{definition}\label{def:formal linear representation}
Let $R$ be a rational expression.
A \emph{formal linear representation of $R$ of dimension $m$} is a tuple $\rho=(u,Q,v)$ consisting of an $m\times m$ linear matrix $Q$ over $\C\left<x_1,\dots,x_d\right>$, a row $u\in M_{1,m}(\C)$ and a column $v\in M_{m,1}(\C)$ for some integer $m$ such that for any complex algebra $\A$
\begin{itemize}
\item $\dom_\A(R)\subset\dom_\A(Q^{-1})$, where
\[
\dom_\A(Q^{-1}):=\{X\in A^d:Q(X)\text{ is invertible in }M_m(\A)\}
\]
\item for any tuple $X=(X_1,\dots,X_d)\in\dom_\A(R)$, we have $R(X)=-uQ(X)^{-1}v$.
\end{itemize}
We denote its \emph{display} by
\[
L:=\begin{pmatrix}0 & u\\v & Q\end{pmatrix}.
\]
\end{definition}

For each rational expression, we can construct a formal linear representation, see \cite[Section 5.2]{HMS18}.
We remark that rational functions also have linear representations which are more or less the same as the ones of rational expressions.
But we won't go into details of their linear representations since they are not needed in this article.
We refer the interested reader to \cite{CR99, Mai2015} for more details.

\subsection{Sylvester rank function}

We will need two rank functions in the sequel.
One is defined algebraically for general rings and the other one is defined via von Neumann algebra theory.
However, they share some common characteristics, which are captured by the notion of \emph{algebraic rank function} (\cite{Mal80}) or \emph{Sylvester matrix rank function} (\cite[Section 7]{Sch85}).
Let us introduce this notion as \emph{Sylvester rank function}.
Since we are mainly interested in objects like von Neumann algebras and free fields, we will state all definitions for unital complex algebras.

\begin{definition}\label{def:Sylvester rank}
For a unital complex algebra $\A$, a \emph{Sylvester rank function on $\A$} is an $\R^+$-valued function $\rk$ on the set of rectangular matrices over $\A$ satisfying:
\begin{enumerate}
\item\label{it:rank of constant} $\rk(0)=0$ and $\rk(1)=1$;
\item\label{it:rank of product} $\rk(AB)\leq\min(\rk(A),\rk(B))$;
\item\label{it:rank of direct sum} $\rk(\begin{pmatrix}A & \0\\\0 & B\end{pmatrix})=\rk(A)+\rk(B)$;
\item\label{it:rank of triangle} $\rk(\begin{pmatrix}A & C\\\0 & B\end{pmatrix})\geq\rk(A)+\rk(B)$.
\end{enumerate}
Moreover, if $\rk(A)>0$ for any non-zero matrix $A$, then we say $\rk$ is \emph{faithful}.
\end{definition}

Item \ref{it:rank of triangle} is not really needed in our study but we keep it for the sake of completeness.
We record here some basic properties of a Sylvester rank function.

\begin{lemma}\label{lem:properties of Sylvester rank}
Let $\A$ be a unital complex algebra with a Sylvester rank function $\rk$.
\begin{enumerate}
\item If $A,B$ are matrices over $\A$ such that $B$ is invertible, then $\rk(AB)=\rk(A)$.
\item If $A\in M_n(\A)$ is invertible, then $\rk(A)=n$.
In particular, $\rk(I_n)=n$.
\item\label{it:subadditivity of Sylvester rank} $\rk$ is subadditive, i.e., for any rectangular matrices $A,B$,
\[
\rk(A+B)\leq\rk(A)+\rk(B).
\]
\end{enumerate}
\end{lemma}

\begin{proof}
Let $A,B$ be matrices over $\A$ such that $B$ is invertible.
According to Item \ref{it:rank of product} of Definition \ref{def:Sylvester rank}, $\rk(AB)\leq\rk(A)$.
Since $A=(AB)B^{-1}$, we apply Item \ref{it:rank of product} again to see that $\rk(A)\leq\rk(AB)$.
Hence we have $\rk(AB)=\rk(A)$.

Combining the fact $\rk(1)=1$ and Item \ref{it:rank of direct sum} of Definition \ref{def:Sylvester rank}, we see that $\rk(I_n)=n$.
Now, let $A\in M_n(\A)$ be an invertible matrix, then by the above paragraph we have $\rk(A)=\rk(I_n)=n$ since $AA^{-1}=I_n$.

Finally, we consider the matrix factorization
\[
\begin{pmatrix}A+B & \0\\\0 & \0\end{pmatrix}
=\begin{pmatrix}\1 & \1\\\0 & \0\end{pmatrix}
\begin{pmatrix}A & \0\\\0 & B\end{pmatrix}
\begin{pmatrix}\1 & \0\\\1 & \0\end{pmatrix}.
\]
Then the rank of the left hand side is $\rk(A+B)$ thanks to Item \ref{it:rank of direct sum} of Definition \ref{def:Sylvester rank}, while the rank of the right hand side is smaller than or equal to
\[
\rk\big(\begin{pmatrix}A & \0\\\0 & B\end{pmatrix}\big)=\rk(A)+\rk(B)
\]
thanks to Item \ref{it:rank of product} and \ref{it:rank of direct sum}.
This completes the proof of the lemma.
\end{proof}

In general a square matrix $A$ with maximal $\rk(A)$ may not be invertible.
Let us introduce the following notion of regularity for Sylvester rank function which connects the maximality of a Sylvester rank function with the invertibility of matrices.

\begin{definition}
Let $\A$ be a unital complex algebra with a Sylvester rank function $\rk$.
If for any integer $n$ and any matrix $A\in M_n(\A)$ with $\rk(A)=n$, $A$ is invertible, then we say $\rk$ is \emph{regular}.
\end{definition}

We know that for each rational expression we can associate a formal linear representation. The explicit link between the rank of a polynomial expression and the rank of its formal linearization has been established in \cite[Lemma 3.4]{BBL2021}. Let us reformulate this rank equality in our more general situation, for rational expressions and an arbitrary tuple of random variables.
\begin{lemma}\label{lem:inner rank of display}
Let $\A$ be a complex algebra with a Sylvester rank function $\rk$ and let $X=(X_1,\dots,X_d)$ be a given tuple over $M_n(\A)$ for some integer $n$.
Let $R$ be a rational expression and $X\in\dom_{M_n(\A)}(R)$.
For any formal linear representation $\rho=(u,Q,v)$ of $R$ of dimension $m$ (see Definition \ref{def:formal linear representation}), we have
\[
\rk(R(X))+nm=\rk(L(X)).
\]
where $L$ is the display of $\rho$.
\end{lemma}

\begin{proof}
Since $X\in\dom_{M_n(\A)}(R)$, we have $Q(X)$ is invertible and $R(X)=-uQ(X)^{-1}v$.
Thus the factorization
\begin{align*}
\begin{pmatrix}R(X) & \0\\\0 & I_{mn}\end{pmatrix}
&=\begin{pmatrix}-uQ(X)^{-1}v & \0\\\0 & I_{mn}\end{pmatrix}\\
&=\begin{pmatrix}I_n & -uQ(X)^{-1}\\\0  & Q(X)^{-1}\end{pmatrix}
\begin{pmatrix}\0 & u\\v & Q(X)\end{pmatrix}
\begin{pmatrix}I_n & \0\\-Q(X)^{-1}v & I_{mn}\end{pmatrix}
\end{align*}
follows.
Hence
\[
\rk\big(\begin{pmatrix}R(X) & \0\\\0 & I_{mn}\end{pmatrix}\big)
=\rk\big(\begin{pmatrix}\0 & u\\v & Q(X)\end{pmatrix}\big)
\]
as the rank function $\rk$ is preserved by invertible matrices.
So we see $\rk(L(X))=\rk(R(X))+\rk(I_{nm})=\rk(R(X))+nm$, which completes the proof.
\end{proof}

\begin{prop}\label{prop:comparison rational closure}
Let $\A$ and $\B$ be two unital complex algebras with faithful and regular Sylvester rank functions $\rk_\A$ and $\rk_\B$ respectively.
Let $X=(X_1,\dots,X_d)\in \A^d$, $Y=(Y_1,\dots,Y_d)\in \B^d$ be two tuples satisfying that for any matrix $A$ over $\C\langle x_1,\dots,x_d\rangle$,
\[
\rk_\B(A(Y))\leq \rk_\A(A(X)).
\]
Then we have the following.
\begin{enumerate}
\item There is a unique specialization from the rational closure $\A_{X}$ to $\B_{Y}$.
\item For any rational expression $R$ in variables $x_1,\dots,x_d$ such that $R(Y)$ is well-defined, $R(X)$ is also well-defined and
\[\rk_\B(R(Y))\leq \rk_\A(R(X)).\]
\end{enumerate}
\end{prop}

\begin{proof}
Let us first give a proof of the first item which is inspired by the proof of \cite[Theorem 6.3]{Jaikin2019}. Let $Z:=(X_1\oplus Y_1,\ldots,X_d\oplus  Y_d)\in \cC:=\A\oplus \B$.
For all $A\in M_n(\Cx)$, $$A(Z)=A(X)\oplus A(Y),$$
which means that $\Sigma_{Z}= \Sigma_{X}\cap \Sigma_{Y}$.
So we see that the rational closure $\cC_Z\subset\A_X\oplus\B_Y$ and thus we can consider the projection maps $\pi_1:\mathcal{C}_Z \to \A_X$ and $\pi_2:\mathcal{C}_Z \to \B_Y$.

Let $c$ be an element in $\cC_Z$.
There exists some $m\times m$ matrix $A\in\Sigma_{Z}$ and some integers $1\leq i,j\leq m$ such that $c$ is the $(i,j)$-entry of $A(Z)^{-1}$.
That is, we can write $c=e_iA(Z)^{-1}e_j^T$, where $(e_1,\dots,e_m)$ stands for the canonical basis of $\C^m$.
Suppose that $\pi_1(c)=e_1A(X)^{-1}e_j^T=0$.
According to the proof of Lemma \ref{lem:inner rank of display},
\[
\rk_\A(\pi_1(c))+m=\rk_\A\big(\begin{pmatrix}0 & e_i\\e_j^T & A(X)\end{pmatrix}\big)\text{ and }\rk_\B(\pi_2(c))+m=\rk_\B\big(\begin{pmatrix}0 & e_i\\e_j^T & A(Y)\end{pmatrix}\big).
\]
(Lemma \ref{lem:inner rank of display} addresses linear matrices but its proof actually works for matrices that are not necessarily linear.)
So by our assumption,
\[
\rk_\B(\pi_2(c))\leq\rk_\A(\pi_1(c))=0.
\]
Because of the faithfulness of $\rk$, it allows us to say that $\pi_2(c)=0$.
Consequently, $c=(0,0)$.
Hence $\pi_1$ is injective and it is possible to define an $\Cx$-algebra isomorphism $g:\im(\pi_1)\to \mathcal{C}_Z$.

Now, we consider $f:=\pi_2\circ g:\im(\pi_1)\to \B_Y$ and thus $\Sigma_{\phi_f}=\Sigma_Z$.
Then we want to see that it is a subhomomorphisms from $\A_X$ to $\B_Y$, i.e., $\Sigma_Y\subset\Sigma_Z$.
If $A\in \Sigma_{Y}$, we know that, by the regularity of $\rk_\B$,
$$n=\rk_\B(A(Y))\leq \rk_\A(A(X)).$$
Thanks to the regularity of $\rk_\A$, $A\in \Sigma_{X}$ and thus $\Sigma_Y\subset\Sigma_X$.
It follows that $\Sigma_Y\subseteq\Sigma_Z$, as desired.

Finally, in order to see that the above subhomomorphism induces a unique specialization, let $f':\A_{f'}\rightarrow\B_Y$ be another subhomomorphism with the homomorphism $\phi_{f'}:\Cx\rightarrow\A_{f'}$.
Then $f'$ is equivalent to $f$ since $\A_{f'}\subset\im(\pi)$ and $f'$ agrees with $f$ on $\A_{f'}$.
The inclusion $\A_{f'}\subset\im(\pi)$ is due to the fact that $\Sigma_Y\subset\Sigma_{\phi_{f'}}\subset\Sigma_X$ and thus $\Sigma_{\phi_{f'}}\subset\Sigma_Z$.
Moreover, $f'$ and $f$ agree since $f'(A(X)^{-1})=f(A(X)^{-1})=A(Y)^{-1}$ when $A\in\Sigma_{\phi_{f'}}$.

Now, let $\mathcal{R}$ be the set of rational expressions such that $R(Y)$ and $R(X)$ are well-defined, and
$\rk_\B(R(Y))\leq \rk_\A(R(X)).$
This set contains the constant and the variables $x_1,\ldots,x_d$. Moreover, it is stable by sums and products: it is clear for the well-definedness, and the inequality
$\rk_\B(R(Y))\leq \rk_\A(R(X))$
is a consequence of Lemma~\ref{lem:inner rank of display} as there exist formal linear representations for sums and products.

In order to conclude that $\mathcal{R}$ contains all rational expressions such that $R(Y)$ is well-defined, it suffices to prove that, for all $R\in \mathcal{R}$, whenever $R(Y)$ is invertible, then $R(X)$ is invertible and 
\[\rk_\B(R(Y)^{-1})\leq \rk_\A(R(X)^{-1}).\]
So let  $R\in \mathcal{R}$ such that $R(Y)$ is invertible. We have
\[1=\rk_\B(R(Y))\leq \rk_\A(R(X))\]
which implies that $R(X)$ is invertible, because of the regularity of $\rk$. The inequality
\[\rk_\B(R(Y)^{-1})\leq \rk_\A(R(X)^{-1})\]
is again a consequence of Lemma~\ref{lem:inner rank of display} as there exist formal linear representations for the inverse of any rational expression.
\end{proof}

\begin{corollary}\label{cor:isomorphic rational closure}
Let $\A$ and $\B$ be two unital complex algebras with faithful and regular Sylvester rank functions $\rk_\A$ and $\rk_\B$ respectively.
Let $X=(X_1,\dots,X_d)\in \A^d$, $Y=(Y_1,\dots,Y_d)\in \B^d$ be two tuples satisfying that for any matrix $A$ over $\C\langle x_1,\dots,x_d\rangle$,
\[
\rk_\A(A(X))=\rk_\B(A(Y)).
\]
Then the rational closures $\A_{X}$ and $\B_{Y}$ are isomorphic. Moreover, for any rational expression $R$ in variables $x_1,\dots,x_d$, $R(Y)$ is well-defined if and only if $R(X)$ is also well-defined and
\[\rk_\B(R(Y))= \rk_\A(R(X)).\]
\end{corollary}

Not that if each tuple is closed under taking the adjoint (for example if they are self-adjoint variables), it is also possible thanks to \cite[Theorem 6.3]{Jaikin2019} to go beyond and prove an isomorphism between the regular closures (which include the rational closures).
\begin{proof}
Thanks to Lemma~\ref{prop:comparison rational closure}, we know that there is a specialization from $\A_X$ to $\B_{Y}$ and another one from $\B_Y$ to $\A_{X}$.
Thanks to Proposition~\ref{prop:isomorphism}, those two specialisations are inverses of each other, and by consequence, are isomorphisms of algebras. The last part of the corollary is a consequence of the last part Proposition~\ref{prop:isomorphism}.
\end{proof}

\begin{lemma}\label{lem:comparison of Sylvester rank}
Let $\A$ be a unital complex algebra with a Sylvester rank function $\rk$, and let $X=(X_1,\dots,X_d)$ and $Y=(Y_1,\dots,Y_d)$ be given tuples over $\A$.
\begin{enumerate}
\item For a given matrix $A\in M_n(\C\left<x_1,\dots,x_d\right>)$, there exists a constant $c:=c(A)$ such that
\[
\rk(A(X)-A(Y))\leq c\max_{i=1,\dots,d}\{\rk(X_i-Y_i)\}
\]
and $c\leq n (d+2d^2+\cdots k d^k)$, where $k$ is the degree of $A$, i.e., the maximal degree of all entries of $A$.
Moreover, if $A$ is linear, i.e., $k=1$, we have
$$\rk(A(X)-A(Y))\leq n\cdot (\rk(X_1-Y_1)+\ldots+\rk(X_d-Y_d)).$$
\item Let $R$ be a rational expression such that $X,Y\in\dom_{\A}(R)$.
For any formal linear representation $(u,Q,v)$ of $R$ of dimension $n$, we have
\[
\rk(R(X)-R(Y))\leq n \cdot(\rk(X_1-Y_1)+\ldots+\rk(X_d-Y_d)).
\]
\end{enumerate}
\end{lemma}

\begin{proof}
First we want to prove Item (1).
We denote $L:=\max_{i=1,\dots,d}\{\rk(X_i-Y_i)\}$ for the sake of simplicity.
Suppose that $A=M\otimes x_{i_1}\cdots x_{i_k}$ is a matrix-valued monomial of degree $k$ for indices $i_1,\dots,i_k\in\{1,\dots,d\}$.
Then we want to show that $\rk(A(X)-A(Y))\leq nkL$.
Note that we can write
\[
A(X)-A(Y)=(M\otimes 1)\cdot(I_n\otimes(\prod_{j=1}^kX_{i_j}-\prod_{j=1}^kY_{i_j})),
\]
as a consequence of Item \eqref{it:rank of product} and \eqref{it:rank of direct sum} of Definition \ref{def:Sylvester rank}, we have
\[
\rk(A(X)-A(Y))\leq n\cdot\rk(\prod_{j=1}^kX_{i_j}-\prod_{j=1}^kY_{i_j}).
\]
Since
\[
\prod_{j=1}^kX_{i_j}-\prod_{j=1}^kY_{i_j}=\sum_{j=1}^k X_{i_{1}}\cdots X_{i_{j-1}}(X_{i_j}-Y_{i_j})Y_{i_{j+1}}\cdots Y_{i_{k}},
\]
by Item \eqref{it:rank of product} of Definition \ref{def:Sylvester rank} and Item \eqref{it:subadditivity of Sylvester rank} Lemma \ref{lem:properties of Sylvester rank},
\[
\rk(\prod_{j=1}^kX_{i_j}-\prod_{j=1}^kY_{i_j})\leq k\cdot\rk(X_{i_j}-Y_{i_j})\leq kL,
\]
as desired.
Therefore, if $A$ is a general matrix in $M_n(\Cx)$, the desired inequality follows by counting the monomials in $A$ with degree smaller or equal than $k$.
Moreover, the particular inequality for the linear case with respect to $\sum_{i=1}^d\rk(X_i-Y_i)$ can be seen by a similar argument.

Next, we want to prove Item (2).
Since $X,Y\in\dom_{\A}(R)$, $Q(X)$ and $Q(Y)$ are invertible, and we have $R(X)=-uQ(X)^{-1}v$ and $R(X)=-uQ(Y)^{-1}v$.
We write
$$\rk(R(X)-R(Y))=\rk(u(Q(X)^{-1}-Q(Y)^{-1})v)\leq \rk(Q(X)^{-1}-Q(Y)^{-1})$$
due to the matrix factorization
\[
\begin{pmatrix}u(Q(X)^{-1}-Q(Y)^{-1})v &\0 \\ \0 &\0_{n-1}\end{pmatrix}
=\begin{pmatrix}u \\ \0\end{pmatrix}(Q(X)^{-1}-Q(Y)^{-1})
\begin{pmatrix}v & \0\end{pmatrix}.
\]
We compute
$$\rk(Q(X)^{-1}-Q(Y)^{-1})=\rk(Q(Y)^{-1}(Q(Y)-Q(X))Q(X)^{-1})\leq \rk(Q(Y)-Q(X))$$
since the rank is preserved by invertible matrices. We conclude thanks to the first item for linear matrices.
\end{proof}

\subsection{Inner rank and von Neumann rank}

The first example of a Sylvester rank function is based on the inner rank, which is defined as follows.

\begin{definition}
Let $\A$ be a unital complex algebra.
For $A\in M_{m,n}(\A)$, $m,n\in\bN$, we define the \emph{inner rank} $\rho_\A(A)$ by
$$\rho_\A(A):=\min\{r\geq 1:A=PQ, P\in M_{m\times r}(\A), Q\in M_{r\times n}(\A) \}$$
and $\rho_\A(\0)=0$ for any rectangular matrix $\0$.
In addition, we call $A$ \emph{full} if $\rho_\A(A)=\min\{m,n\}$.
\end{definition}

The first two items in Definition \ref{def:Sylvester rank} follows from the definition of inner rank.
But an inner rank may not be a Sylvester rank function in general.
In the sequel, inner rank will majorly be considered on the unital complex algebra $\C\langle x_1,\dots,x_d\rangle$ of non-commutative polynomials.
This algebra is known as a \emph{Sylvester domain} (see, for example, its definition and properties \cite[Section 5.5]{Coh06}).
On a Sylvester domain, Item \ref{it:rank of direct sum} is satisfied and it's not difficult to see that Item \ref{it:rank of triangle} also holds.
So $\rho_{\C\langle x_1,\dots,x_d\rangle}$ is a faithful Sylvester function.

Moreover, the inner rank on $\C\langle x_1,\dots,x_d\rangle$ extends to the free field $\C\plangle x_1,\dots,x_d\prangle$ (which is actually a property of a Sylvester domain).
So we will not distinguish these two inner ranks and usually denote them by $\rho$.
On the free field, the inner rank $\rho$ is actually regular, i.e., full matrices are invertible over the free field.
Let us record here this important fact as a remark.

\begin{remark}
The inner rank $\rho$ on the free field $\C\plangle x_1,\dots,x_d\prangle$ is a faithful and regular Sylvester rank function.
\end{remark}

Another important example of Sylvester rank functions is an analytic rank on von Neumann algebras, which is defined as follows.

\begin{definition}\label{def:rank}
Let $(\M,\tau)$ be a tracial $W^*$-probability space. For $A\in M_n(L^0(\M))$, we define its \emph{von Neumann rank} by
$$\rank(A):=\Tr_n\otimes \tau(p_{\overline{\im(A)}}),$$
where $p_{\overline{\im(A)}}$ is the orthogonal projection onto the closure of the image of $A$ and $\Tr_n$ is the unnormalized trace on $M_n(\C)$.
\end{definition}

Recall that in Definition \ref{def:point spectrum} the weight of each $\lambda$ in the point spectrum $\sigma_p(X)$ of $X\in L^0(\M)$ is defined with the help of kernel.
So we also have
\[
\mu_X^p=\sum_{\lambda\in\sigma_p(X)}(1-\rank(\lambda-X))\delta_\lambda,
\]
thanks to the well-known equality $\rank(Y)=1-\tau(p_{\ker(Y)})$, that holds for any operator $Y\in L^0(\M)$ (see, for example, \cite[Lemma 5.2]{MSY19}).
Therefore, for two operator $X,Y\in L^0(\M)$ we have $\mu_X^p\leq\mu_Y^p$ if and only if $\rank(\lambda-Y)\leq\rank(\lambda-X)$ for each $\lambda\in\C$.

It is probably well-known to experts that the von Neumann rank is indeed a Sylvester rank function on $L^0(\M)$.
We state this here as a lemma and include a proof for reader's convenience.

\begin{lemma}\label{lem:vN rank is Sylvester}
Let $(\M,\tau)$ be a tracial $W^*$-probability space.
Then $\rank$ is a faithful and regular Sylvester rank function on $L^0(\M)$.
\end{lemma}

\begin{proof}
Clearly, Item \eqref{it:rank of constant}, \eqref{it:rank of direct sum}  and \eqref{it:rank of triangle} of Definition \ref{def:Sylvester rank} follows directly from the definition of the von Neumann rank.
Suppose that $A,B\in M_n(L^0(\M))$ and let us prove Item \eqref{it:rank of product}.
First, we see that $\rank(AB)\leq \rank(A)$ since $\im(AB)\subset\im(A)$.
Similarly, we have $\rank((AB)^\ast)=\rank(B^\ast A^\ast)\leq\rank(B^\ast)$.
Recall that $\rank(X)=\rank(X^\ast)$ since $\tau$ is tracial, we infer that $\rank(AB)\leq\rank(B)$.
So desired inequality $\rank(AB)\leq\min(\rank(A),\rank(B))$ follows.

Moreover, $\rank$ is clearly faithful.
And the regularity amounts to say that $A\in M_n(L^0(\M))$ is invertible if and only if $\ker(A)\neq\{0\}$, which is a well-known fact (see, for example, \cite[Lemma 5.5]{MSY19}).
\end{proof}

\begin{lemma}\label{lemma:approximation}
Let $X=(X_1,\ldots,X_d)$ be a tuple of $*$-free normal variables in some tracial $W^*$-probability space $(\M,\tau)$. Then, there exists a constant $C_X>0$ such that, for all $N\geq 1$ (enlarging the $W^*$-probability space if necessary), there exists a tuple of $*$-free normal variables  $X'=(X_1',\ldots,X_d')$ such that, for all $1\leq i \leq d$:
\begin{itemize}
    \item the weights of the atoms of $\mu_{X_i'}$ are multiples of $1/N$ which are smaller than the corresponding weights of the atoms of $\mu_{X_i}$;
    \item $\rank(X_i-X_i')\leq C_X/N$.
\end{itemize}
\end{lemma}

\begin{proof}
Let us prove the result for one variable (the general case can be deduced by taking the free product of $d$ $*$-algebras).

By the spectral theorem, we write $X=\int_{\sigma(X)}\lambda dE_{\lambda}$, where $E_{\lambda}$ is the spectral resolution of $X$.
Suppose that $\lambda_1,\dots,\lambda_{l}$ are the distinct atoms of $X$ and $q_i:=E_{\lambda_i}$ is the orthogonal projection onto the eigenspace $\ker(\lambda_i-X)$ for each $i=1,\dots,l$.
Moreover, we set $q_0:=1-\bigvee_{i=1}^l q_i$, then $\tau(q_0)=1-\sum_{i=1}^l\tau(q_i)$ since $q_{i}$'s are orthogonal with each other.
Therefore, we could consider $X$ in the following form
\[
X=q_0Xq_0+\sum_{i=1}^l\lambda_i q_i. 
\]
We consider a variable with no atoms $S\in \M$ (let say a semicircular variable), and $l$ Bernoulli variables $\beta_i$ of respective sizes $\lfloor N\tau(q_i)\rfloor/N\tau(q_i)$ in some probability space $L^{\infty}(\Omega,\mathcal{F},\mathbb{P})$. Then we define
$$X':=q_0Xq_0\otimes 1+\sum_{i=1}^l(\lambda_i q_i\otimes \beta_i+S\otimes (1-\beta_i))\in \M\otimes L^{\infty}(\Omega,\mathcal{F},\mathbb{P}),$$
whose atoms are $\lambda_1,\dots,\lambda_{l}$ with respective weights $\lfloor N\tau(q_i)\rfloor/N$, where $\lfloor c\rfloor$ stands for the largest integer which is smaller than or equal to the number $c\in\R$.
Identifying $\M$ as a subalgebra of the finite $W^*$-probability space $\M\otimes L^{\infty}(\Omega,\mathcal{F},\mathbb{P})$, we can compute the rank
\begin{equation*}
    \rank(X-X')
    \leq \sum_{i=1}^l \rank(1-\beta_i) 
    \leq \sum_{i=1}^l \left(1- \frac{\lfloor N\tau(q_i)\rfloor}{N\tau(q_i)}\right)
    \leq \frac{1}{N}\left(\sum_{i=1}^l \frac{1}{\tau(q_i)}\right).
\end{equation*}
\end{proof}

Let $\mu_1,\mu_2$ be two probability measure on $\R$.
Their \emph{Kolmogorov distance} is defined as
\[
d_{Kol}(\mu_1,\mu_2):=\sup_{t\in\R}|\cF_{\mu_1}(t)-\cF_{\mu_2}(t)|,
\]
where $\cF_\mu$ is the cumulative distribution function for a probability measure $\mu$.
Let $X,Y\in L^0(\M)$ be selfadjoint random variables for some $W^\ast$-probability space $(\M,\tau)$.
Then the Kolmogorov distance of their analytic distributions $\mu_X,\mu_Y$ is controlled by the rank of their difference, as follows.

\begin{lemma}{\cite[Lemma 25]{CMMPY2021}}\label{lem:Kol distance}
Let $X,Y\in L^0(\M)$ be selfadjoint random variables, where $(\M,\tau)$ is some $W^\ast$-probability space.
Then
$$d_{Kol}(\mu_{X},\mu_{Y})\leq \rank(X-Y).$$
\end{lemma}
Therefore, if $R$ is a rational expression such that $X,Y\in\dom_{\A}(R)$, and such that $R(X)$ and $R(Y)$ are selfadjoint, then
$$d_{Kol}(\mu_{R(X)},\mu_{R(Y)})\leq \rank(R(X)-R(Y)),$$ and it allows to state the following approximation theorem.

\begin{theorem} \label{BVrational}
Let $X=(X_1,\ldots,X_d)$ and $Y=(Y_1,\ldots,Y_d)$ be two tuples of free selfadjoint random variables in some $W^\ast$-probability space.
\begin{itemize}
\item  Let $A\in M_n(\C\left<x_1,\dots,x_d\right>)$. There exists a constant $c:=c(A)$ such that
$$|\rank(A(X)-A(Y))|\leq c \max_{i=1,\dots,d}\{d_{Kol}(\mu_{X_i},\mu_{Y_i})\}.$$
\item Let $R$ be a self-adjoint rational expression in $d$ noncommuting variables such that $X,Y\in\dom_{\A}(R)$. There exists a constant $c:=c(R)$ such that
$$d_{Kol}(\mu_{R(X)},\mu_{R(Y)})\leq c \max_{i=1,\dots,d}\{d_{Kol}(\mu_{X_i},\mu_{Y_i})\}.$$\end{itemize}
\end{theorem}

\begin{proof}
Let $\delta=\max\{d_{Kol}(\mu_{X_i},\mu_{Y_i}):i=1,\ldots,d\}$. As in the proof of \cite[Theorem 4.11]{BV93} we may  suppose $\{X_i,Y_i\}$ live in a $W$*-probability space $(\mathcal{M},\tau)$ so that there exist projections $\{p_i\}^n_{i=1}\in \mathcal{M}$ such that $\tau(p_i)\geq 1-\delta$, $p_iX_ip_i=p_iY_ip_i$ and $\{X_i,Y_i,p_i\}_i$ is a free family. Decomposing $X_i$ as $p_iX_ip_i+(1-p_i)X_ip_i+p_iX_i(1-p_i)+(1-p_i)X_i(1-p_i)$, we  can compute
\begin{align*}
    \rank(X_i-Y_i)&\leq \rank((1-p_i)(X_i-Y_i)p_i)\\
    &+\rank(p_i(X_i-Y_i)(1-p_i))+\rank((1-p_i)(X_i-Y_i)(1-p_i))\\
    &\leq 3\cdot\rank(1-p_i)\leq 3\delta.
\end{align*}
As a consequence of Lemma \ref{lem:comparison of Sylvester rank}, we have
\begin{align*}
|\rank(A(X)-A(Y))| &\leq K\max_{i=1,\dots,d}\{\rank(X_i-Y_i)\}
\leq 3  K\delta
\end{align*}
for a certain constant $K$.
In particular, for any formal linear representation $(u,Q,v)$ of $R$ of dimension $n$, we have
$$\rank(R(X)-R(Y))\leq n \cdot(\rank(X_1-Y_1)+\ldots+\rank(X_d-Y_d))\leq 3 nd\delta,$$
thanks to Lemma \ref{lem:comparison of Sylvester rank} again.
Then the desired the result follows thanks to Lemma \ref{lem:Kol distance}.
\end{proof}

\subsection{The quantity $\Delta$}The free field $\C\plangle x_1,\dots,x_d\prangle$ can seem quite abstract. However, it is possible to represent it as a concrete algebra of operators. In \cite{MSY19}, a link was discovered between the free field and the rational closure of tuples of operators maximizing a certain quantity $\Delta$. Let us introduce this quantity and recall some results of \cite{MSY19}.

Let us denote by $\mathcal{F}(L^2(\M,\tau))$ the $M-M$ Hilbert bimodule of finite-rank operators on $L^2(\M,\tau)$, and by $J$ Tomita’s conjugation operator, i.e., the conjugate-linear map $J:L^2(\M,\tau)\to L^2(\M,\tau)$ such that $J(x)=x^*$ for $x\in \M$. Considering a tuple $X=(X_1,\ldots,X_d)$ of noncommutative random variables in $\M$, the quantity $\Delta(X)$ was introduced in \cite{CS05} as
$$
    \Delta(X)=d-\overline{\dim \left\{(T_1,\ldots,T_d)\in \mathcal{F}(L^2(\M,\tau))^d:\sum_{i=1}^d[T_i,JX_iJ]=0\right\}}^{HS}
$$
where the closure is  taken with respect to the Hilbert-Schmidt norm. 
The quantity $\Delta(X)$ is maximal whenever $
    \Delta(X)=d$, or equivalently whenever there do not exist a $d$-tuple $(T_1,\ldots,T_d)\in \mathcal{F}(L^2(\M,\tau))^d$ such that $\sum_{i=1}^d[T_i,X_i]=0$.

\begin{theorem}{\cite[Theorem 1.1]{MSY19}}\label{th:Delta_maximal}
Let $X=(X_1,\dots,X_d)$ be a $d$-tuple of operators in a tracial $W^*$-probability space $(\M,\tau)$. The following are equivalent:
\begin{enumerate}
    \item $\Delta(X)=d$;
    \item the evaluation map sending $P\in\Cx$ to $P(X)$ extends to an isomorphism from $\C\plangle x_1,\dots,x_d\prangle$ to the rational closure $L^0(\M)_X$ of $X$;
    \item for any $n\in\mathbb{N}$ and any matrix $A\in M_n(\C\left<x_1,\dots,x_d\right>)$, we have
    $$\rank(A(X))=\rho(A).$$
\end{enumerate}
\end{theorem}

Note that the rank equality in Item (3) of the above theorem actually holds over the free field provided that one of the equivalent properties holds for a given tuple $X$.
That is, Item (3) can be replaced by Item (3'): for any $M_n(\C\plangle x_1,\dots,x_d\prangle)$, we have $\rank(A(X))=\rho(A)$.
This is not explicitly stated in \cite[Theorem 1.1]{MSY19} but it can be easily read out from the proof of \cite[Theorem 5.6]{MSY19}.
We summarize it as a remark that will needed in Section \ref{sec:calculation}.
\begin{remark}\label{rem:Delta_maximal}
For a $d$-tuple $X$ over some $W^\ast$-probability space with $\Delta(X)=d$, the evaluation map is a rank-preserving homomorphism from $\C\plangle x_1,\dots,x_d\prangle$ to $L^0(\M)$.
\end{remark}

We also have the following additivity of the quantity $\Delta$.
\begin{prop}[Theorem 3.3 of \cite{CS05}]\label{prop:delta_free}
Let $X_1,\ldots,X_d,X_{d+1},\ldots X_{d'}$ be in a tracial $W^*$-probability space $(\M,\tau)$.
Whenever $X_1,\ldots,X_d$ is $*$-free from $X_{d+1},\ldots X_{d'}$, we have
$$\Delta(X_1,\ldots,X_{d'})=\Delta(X_1,\ldots,X_{d})+\Delta(X_{d+1},\ldots,X_{d'}).$$
\end{prop}
\begin{proof}As remarked in \cite[Theorem 16]{CMMPY2021}, the proof of \cite[Theorem 3.3]{CS05}, which is written for self-adjoint variables, is valid for not necessarily self-adjoint $*$-free variables.\end{proof}

\section{Matrix models of random variables}
\label{sect:matrix models}

\subsection{Matrix model of normal random variables}

The goal of this subsection is to bring normal random variables to a matrix form of which the information of atoms can be easily read out.
This matrix form will play an important role in Section \ref{sec:calculation} and \ref{sec:comparison}.

Clearly, if $Y_1,\dots,Y_d$ are random variables in the $W^\ast$-probability space $(M_n(\C),\tr_n)$ for some $n\in\bN$, then they are already in a matrix form (with atoms of rational weights).
So in the following we focus on the random variables that have non-trivial non-atomic parts.
We will present two constructions that give us a matrix model for a given random variable.

In the first one, we consider normal random variable whose atoms are of rational weights.
For such a variable, we will construct a diagonal matrix such that they have the same distribution.
If we have multiple such random variables, the diagonal matrices produced by this procedure will lose the information on the relation between the given random variables.
However, with the help of freely independent Haar unitaries, we can further build a freely independent tuple of matrices which copy the distribution of each of the given random variables.

In the second one, we will show that for a normal variable $Y$, its matrix amplification $I_n\otimes Y$ can be decomposed into a form of $UDU^\ast$, where $U$ is unitary and $D$ is a diagonal matrix containing the information of atoms.
Comparing to the first construction, we can keep the information of the relation between the given random variables.

The first construction is given as the following lemma.

\begin{lemma}\label{Matrix_form:free_case}
Let $Y$ be a normal random variable in some $W^\ast$-probability space $(\M,\tau)$ such that there exists $n\in \mathbb{N}$ with $\mu_Y(\{\lambda\})\in \frac{1}{n}\mathbb{N}$ for each $\lambda\in \mathbb{C}$.
Then there is some $W^\ast$-probability space $(\cN,\varphi)$ such that
the distribution of $Y$ is the one of
\[
D:=
\begin{pmatrix}\alpha_1&\cdots & 0&&&\\
\vdots &\ddots &\vdots&&\0&\\
0 &\cdots &\alpha_{m}&&&\\
&& &\gamma_1 &\cdots & 0\\
&\0& &\vdots &\ddots &\vdots\\
&& &0 &\cdots &\gamma_{m_0}
\end{pmatrix}\in (M_n(\cN),\tr_n\circ\varphi),
\]
where $\gamma_1,\dots,\gamma_{m_0}$ are $*$-free normal random variables without atoms,  and $\{\alpha_1,\dots,\alpha_{m}\}\in \mathbb{C}$ is the (multi)-set of atoms of $Y$, in such a way that the point spectrum distribution $\mu^p_{Y}$ of $Y$ is
$\frac{1}{n}(\delta_{\alpha_1}+\ldots+\delta_{\alpha_m})$.
\end{lemma}
\begin{proof}
First, recall that we can write $Y=q_0Yq_0+\sum_{i=1}^l\lambda_i q_i$, where $\lambda_1,\dots,\lambda_l\in\sigma_p(Y)$, $q_1,\dots,q_l$ are orthogonal projections onto the eigenspaces of $Y$ and $q_0:=1-\sum_{i=1}^lq_i$ (see the proof of Lemma \ref{lemma:approximation}).
We denote $\M_0:=q_0\M q_0$ and $\tau_0(\cdot):=\frac{1}{\tau(q_0)}\tau(\cdot)$, which form a compressed $W^\ast$-probability space of $(\M,\tau)$.
We consider $\gamma:=q_0Yq_0$ as a random variable in $\M_0$.
Since $Y$ commutes with $q_0$, we see that $\gamma$ is normal as $Y$ is normal.
Moreover, $\lambda-\gamma$ is injective on $\im (q_0)$ for any $\lambda\in\C$ otherwise there is some subspace of $\im (q_0)$ which is an eigenspace of $Y$.
So the analytic distribution of $\gamma\in(\M_0,\tau_0)$ has no atom.

Next, let $n$ be an integer such that $\tau(q_1),\dots,\tau(q_l)$ can be written as fractions with their denominators to be $n$.
We denote by $m_i:=n\tau(q_i)$ for each $i=1,\dots,l$, $m:=\sum_{i=1}^lm_i$ and $m_0:=n-m$.
Now we define $(\cN,\varphi):=(\M_0,\tau_0)^{\ast m_0}$ to be the free product of $m_0$ copies of $(\M_0,\tau_0)$.
Moreover, let $\gamma_1,\dots,\gamma_{m_0}\in\cN$ to be $m_0$ copies of $\gamma\in\M_0$ such that $\gamma_1,\dots,\gamma_{m_0}$ are freely independent (and of course, they are normal random variables with non-atomic analytic distributions).
Let us set $D\in M_n(\cN)$ to be a diagonal matrix whose diagonal entries consists of
\begin{itemize}
\item $m$ constants $\alpha_1,\ldots,\alpha_m\in \mathbb{C}$, which are given by $m_i$ copies of the constant $\lambda_i$ for each $1\leq i \leq \ell$;
\item $\gamma_1,\dots,\gamma_{m_0}$.
\end{itemize}
Then we see that $D$ and $Y$ have the same $\ast$-distribution since for any polynomial $p\in\C\left<x,x^\ast\right>$,
\allowdisplaybreaks
\begin{align*}
\varphi(p(D,D^\ast))&=\sum_{i=1}^m\frac{1}{n}p(\alpha_i,\overline{\alpha_i})+\frac{1}{n}\sum_{j=1}^{m_0}\varphi(p(\gamma_j,\gamma^\ast_j))\\
&=\sum_{i=1}^l\frac{m_i}{n}p(\lambda_i,\overline{\lambda_i})+\frac{1}{n}\sum_{j=1}^{m_0}\varphi(p(\gamma_j,\gamma^\ast_j))\\
&=\sum_{i=1}^l\frac{m_i}{n}(p(\lambda_i,\overline{\lambda_i}))+\frac{m_0}{n}\tau_0(p(\gamma,\gamma^\ast))\\
&=\sum_{i=1}^l\tau(q_i)p(\lambda_i,\overline{\lambda_i})+\tau(q_0)\tau_0(p(\gamma,\gamma^\ast))\\
&=\sum_{i=1}^l\tau(q_i)p(\lambda_i,\overline{\lambda_i})+\tau(q_0p(Y,Y^\ast)q_0)\\
&=\tau(p(Y,Y^\ast)).
\end{align*}
\end{proof}

Since we consider random variables with non-trivial non-atomic part, we will assume the von Neumann algebra to be type $II_1$.
Then we can embed it, with also preserving the trace, into a type $II_1$ \emph{factor}.
Thanks to the factoriality, we can construct a unitary matrix by partial isometries that come from equivalent projections.
Then we use this unitary matrix to digonalize our random variables. The details of these statements are given in the following proposition and its proof. 

\begin{prop}
\label{Matrix_form:general_case}
Let $Y_1,\dots,Y_d$ be normal random variables in some tracial $W^\ast$-probability space.
Then they can be embedded into a tracial $W^\ast$-probability space $(\M,\tau)$ such that $\M$ is a factor.
Let $n\in\bN$ be given.
Then for each $Y_k$, there exist a unitary matrix $U_k\in M_n(\M)$ and a diagonal matrix $D_k\in M_n(\M)$ such that
\begin{itemize}
\item $U_kD_kU_k^\ast=I_n\otimes Y_k$ for each $k=1,\dots,d$;
\item for every atom $\lambda$ of $Y$ with weight $w$, there are $\lfloor nw\rfloor$ diagonal elements of $D_k$ appear as $\lambda$.
\end{itemize}
\end{prop}

\begin{proof}
Thanks to \cite[Proof of Theorem 2.6]{Haagerup2000} or \cite[Theorem A.1.5]{O2012}, a type $II_1$ von Neumann algebra with a tracial state can be embedded into a type $II_1$ factor such that the embedding preserves also the trace. (Note that the normal trace on a factor is unique.)
Therefore, let us assume that $(\M,\tau)$ is a tracial $W^\ast$-probability space such that $\M$ is a factor.
It's enough to prove the statement for one normal random variable $Y\in\M$.
First, we write
\[
Y=q_0Yq_0+\sum_{i=1}^l\lambda_i q_i. 
\]
as in the proof of Lemma \ref{lemma:approximation}, where $\lambda_1,\dots,\lambda_l\in\sigma_p(Y)$, $q_1,\dots,q_l$ are orthogonal projections onto the eigenspaces of $Y$ and $q_0:=1-\sum_{i=1}^lq_i$.

Next, for each $i=0,1,\dots,l$, let $m_i:=\lfloor n\tau(q_i)\rfloor$.
We expect to find $m_i$ subprojections of $q_i$ with trace equal to $\frac{1}{n}$.
This can be done as follows.
We know that there exists a projection $p_{\frac{1}{n}}\in\M$ with $\tau(p_{\frac{1}{n}})=\frac{1}{n}$ since $\M$ is a type $II_1$ factor (see, for example, \cite[Proposition 8.5.3]{KR86}).
If $m_i\geq1$, i.e., $\frac{1}{n}\leq\tau(q_i)$, then by the comparison theory of the factor we infer from $\tau(p_{\frac{1}{n}})\leq\tau(q_i)$ that $p_{\frac{1}{n}}\precsim q_i$ (see, for example, \cite[Section 8.5]{KR86}).
This implies that there exists a projection $p^i_1\leq q_i$ with $\tau(p^i_1)=\tau(p_{\frac{1}{n}})=\frac{1}{n}$.
Similarly we can find a projection $p^i_2\leq q_i-p^i_1$ with $\tau(p^i_2)=\frac{1}{n}$ if $m_i\geq2$.
We repeat this procedure until there are $m_i$ orthogonal subprojections $p^i_1,\dots,p^i_{m_i}$ of $q_i$ with equal trace $\frac{1}{n}$.
Therefore, we write $q_i=\sum_{j=1}^{m_i}p^i_{j}+p^i_0$, where $p^i_0:=q_i-\sum_{j=1}^{m_i}p^i_j$ stands for the remaining part.
Note that $p^i_1,\dots,p^i_{m_l},p^i_0$ commute with $Y$, we see that
\[
q_iYq_i=(\sum_{j=1}^{m_i}p^i_j+p^i_0)Y(\sum_{j=1}^{m_0}p^i_j+p^i_0)=\sum_{j=1}^{m_0}p^i_jYp^i_j+p^i_0Yp^i_0.
\]
Therefore, we could write
\[
Y=\sum_{i=0}^l\sum_{j=0}^{m_i}p^i_jYp^i_j.
\]
Let us set $p^{l+1}:=\sum_{i=0}^l p_0^i$.
Then we have 
\[
\tau(p^{l+1})=1-\sum_{i=0}^l\frac{m_i}{n}=\frac{n-\sum_{i=0}^lm_i}{n}.
\]
We denote $m_{l+1}:=n-\sum_{i=0}^l m_i$.
Furthermore, we divide also $p^{l+1}$ into $m_{l+1}$ subprojections with trace $\frac{1}{n}$ and denote them by $p_1^{l+1},\dots,p_{m_{l+1}}^{l+1}$.
So we finally write $Y$ as
\[
Y=\sum_{i=0}^{l+1}\sum_{j=1}^{m_i}p^i_jYp^i_j.
\]
For simplicity's sake, we rewrite the indices of $p^i_j$ as $1,\dots,n$ and thus
\[
Y=\sum_{j=1}^{n}p_jYp_j.
\]

Now let us consider the $W^\ast$-probability space $(M_n(\M),\tr_n\circ\tau)$.
For each pair of projections $(p_i,p_j)$ ($i<j$) we can find an partial isometry $u_{ij}\in\M$ such that $u_{ij}u_{ij}^\ast=p_j$ and $u_{ij}^\ast u_{ij}=p_i$ since $p_i$ and $p_j$ are equivalent relative to $\M$ (recall that the equivalence of projections is equivalent to the equality of traces of projections, see, for example \cite[Section 8.5]{KR86}).
Let $u_{ii}:=p_i$ and $u_{ij}:=u_{ji}^\ast$ if $i>j$, then we have $u_{ij}u_{ji}=p_j$ for $i,j=1,\dots,n$.
We define
\[
U:=\big(u_{ij}\big)_{i,j=1}^n,
\]
which is a selfadjoint unitary operator in $M_n(\M)$.
Moreover, let us define
\[
D:=
\begin{pmatrix}
\alpha_1 & 0 & \cdots & 0\\
0 & \alpha_2 & \cdots & 0\\
\vdots & \vdots & \ddots & \vdots\\
0 & 0 & \cdots & \alpha_n
\end{pmatrix}\in M_n(\M),
\]
where $\alpha_j:=\sum_{i=1}^nu_{ji}Yu_{ij}$.
Note that if $p_j$ satisfies $Yp_j=\lambda p_j$ for $\lambda\in\sigma_p(Y)$, then
\[
\alpha_j=\lambda\sum_{i=1}^nu_{ji}u_{ij}=\lambda\sum_{i=1}^np_i=\lambda.
\]
Moreover, if $\lambda$ has weight $w$, then there are $\lfloor nw\rfloor$ $p_i$'s that are subprojections of the projection $p_{\ker(\lambda-Y)}$.
So $\lambda$ appears $\lfloor nw\rfloor$ many times in $\alpha_j$'s.

Finally, from a direct verification 
\begin{align*}
UDU&=\big(\sum_{j=1}^nu_{kj}\alpha_ju_{jl}\big)_{k,l=1}^n\\
&=\big(\sum_{i,j=1}^nu_{kj}u_{ji}Yu_{ij}u_{jl}\big)_{k,l=1}^n\\
&=\big(\sum_{j=1}^nu_{kj}u_{jl}Yu_{lj}u_{jl}\big)_{k,l=1}^n\\
&=\big(\delta_{kl}\sum_{j=1}^np_jYp_j\big)_{k,l=1}^n\\
&=I_n\otimes Y,
\end{align*}
we see that $UDU\in M_n(\M)$ is indeed a matrix model for random variable $Y$.
\end{proof}

\subsection{Non-commutative derivations and $\Delta$}
Let $(\mathcal{M},\tau)$ be a tracial $W^*$-probability space. A non-commutative derivation
$\delta$ on $\mathcal{M}$ with domain $\mathcal{A}\subset \mathcal{M}$ is a (possibly unbounded) linear map
$$\delta:L^2(\mathcal{M},\tau)\supseteq \mathcal{A} \to L^2(\mathcal{M},\tau)\otimes L^2(\mathcal{M},\tau),$$
such that $\mathcal{A}$ is a $*$-algebra which is weakly dense in $\mathcal{M}$, and such that, for all $X_1,X_2\in \mathcal{A}$, we have the Leibniz rule
$$\delta(X_1X_2)=\delta(X_1)\cdot X_2+X_1\cdot \delta(X_2).$$
For all $\eta\in L^2(\mathcal{M},\tau)\otimes L^2(\mathcal{M},\tau)$, we say that $\delta^*(\eta)$ exists (or $\eta$ is in the domain of $\delta^*$) if $\delta^*(\eta)$ is a vector of $L^2(\mathcal{M})$ such that, for all $X\in\mathcal{A}$, we have
$$\langle \eta,\delta X\rangle=\langle \delta^*(\eta),X\rangle .$$
For all $\eta\in \mathcal{A}\otimes \mathcal{A}$, the map $A\mapsto A\sharp \eta$, given by $(A_1\otimes A_2)\sharp (x\otimes y)=(A_1x\otimes y A_2)$ on tensors, extends continuously to the space  $L^2(\mathcal{M},\tau)\otimes L^2(\mathcal{M},\tau)$. Because we have
$\langle \delta(X)\sharp \eta,1\otimes 1\rangle =\langle \delta(X),\eta^*\rangle,$
the existence of $\delta^*(\eta^*)$ is equivalent to the existence of $(\delta^\eta)^*(1\otimes 1)$ for the non-commutative derivation $\delta^\eta:X\mapsto \delta(X)\sharp \eta$. We have the following consequence.

\begin{lemma}{\cite[Lemma 4.3]{Mai2015}}\label{lemma:Mai}Let $\delta$ be a non-commutative derivation on $\mathcal{M}$ with domain $\mathcal{A}\subset \mathcal{M}$. The following are equivalent:
\begin{itemize}
    \item $\delta^*(1\otimes 1)$ exists;
    \item $\delta^*(\eta)$ exists for all $\eta\in \mathcal{A}\otimes \mathcal{A}$;
    \item $(\delta^\eta)^*(1\otimes 1)$ exists for all $\eta\in \mathcal{A}\otimes \mathcal{A}$ (where $\delta^\eta$ is given by $X\mapsto \delta(X)\sharp \eta$).
\end{itemize}
\end{lemma}
Let $P_\Omega:L^2(\mathcal{M},\tau)\to L^2(\mathcal{M},\tau)$ be the orthogonal projection onto the trace vector $1$. The Hilbert space $L^2(\mathcal{M},\tau)\otimes L^2(\mathcal{M},\tau)$ is isometrically isomorphic to the space of Hilbert-Schmidt operators on $L^2(\mathcal{M},\tau)$ via the map $A\sharp P_\Omega$, given by $(A_1\otimes A_2) \sharp P_\Omega=A_1P_\Omega A_2$ for all $A_1\otimes A_2\in \mathcal{A}\otimes \mathcal{A}$. The existence of $\delta^*(1\otimes 1)$ allows to define some unbounded operators on $L^2(\mathcal{M},\tau)$ with very interesting properties, thanks to the following theorem.

\begin{theorem}{\cite[Theorem 1]{S2006}}\label{theorem:derivation}
Let $\partial$ be a non-commutative derivation on $\mathcal{M}$ with domain $\mathcal{A}\subset \mathcal{M}$. If $\partial^*(1\otimes 1)$ exists, there exists a closable unbounded operator  $D: L^2(\mathcal{M},\tau)\to L^2(\mathcal{M},\tau)$ such that $[D,X]=(\partial X)\sharp P_\Omega$ for all $X\in \mathcal{A}$.
\end{theorem}

\begin{proof}
We recall briefly here the construction (it is written for self-adjoint variables in \cite{S2006}, however the proof is valid also in our context, since $\mathcal{A}$ is a dense $*$-algebra in $\mathcal{M}$). Let us consider the densily defined operator
$$D:=(\id\otimes \tau)\circ \partial:L^2(\mathcal{M},\tau)\supseteq \mathcal{A} \to L^2(\mathcal{M},\tau)$$
which is such that, for all $A\in \mathcal{A}$, we have
$[D,A]=\partial A\sharp P_{\Omega}.$
Now, the existence of $\partial^*(1\otimes 1)$ allows us to define an adjoint
$$D^*:L^2(\mathcal{M},\tau)\supseteq \mathcal{A} \to L^2(\mathcal{M},\tau)$$
by the formula
$D^*(A):=A\cdot \partial^*(1\otimes 1)-(\partial(A^*)\sharp P_\Omega)^*(1)=A\cdot\partial^*(1\otimes 1)-\tau\otimes \id[(\partial(A^*))^*].$
As a consequence, $D$ is closable (see also \cite[Theorem 8 and Exercise 3]{MS17}).
\end{proof}

We use the previous theorem in order to give a criterion for the maximality of $\Delta$.
\begin{prop}\label{prop:Deltamaximal}Let $X_1,\ldots,X_d$ be in a tracial $W^*$-probability space $\mathcal{M}:=W^*(X_1,\ldots,X_d)$ endowed with the state $\tau$. Assume that there exist non-commutative derivations $\partial_i$ ($1\leq i\leq d$) on $\mathcal{M}$ with domain $\mathbb{C}\langle X_1,\ldots,X_d,X_1^*,\ldots,X_d^*\rangle,$
such that $\partial_i(X_j)=\delta_{i,j}1\otimes 1$ and such that $\partial_i^*(1\otimes 1)$ exists. Then $\Delta(X_1,\ldots,X_d)=d$.
\end{prop}

In our setting, $\partial_i$ is defined on $\mathbb{C}\langle X_1,\ldots,X_d,X_1^*,\ldots,X_d^*\rangle$ and not only on $\mathbb{C}\langle X_1,\ldots,X_d\rangle$ (as it is usually the case for self-adjoint variables). The condition in the Lemma above is exactly the existence of a conjugate system if $X_i=X_i^*$. 
\begin{proof}
Thanks to Theorem~\ref{theorem:derivation}, the existence of the derivation $\partial_i$ allows to define an (unbounded) dual system $(D_1,\ldots,D_d)$ in the following sense. For each $1\leq i \leq d$, the densely defined operator
$D_i:L^2(\mathcal{M},\tau)\to L^2(\mathcal{M},\tau)$
is such that, for all $A\in \mathbb{C}\langle X_1,\ldots,X_d,X_1^*,\ldots,X_d^*\rangle$, we have
$[D_i,A]=(\partial_i A)\sharp P_{\Omega}.$
In particular, we have
$$[D_i,X_j]=\delta_{ij}P_{\Omega}.$$
In order to conclude that $\Delta(X_1,\ldots,X_d)=d$, it suffices now to use \cite[Proposition 6.2]{MSY19}, which is also valid if the dual system is in the algebra of affiliated operators (the proof works \it{mutatis mutandis}).
\end{proof}

\begin{lemma}
If a unitary variable $U$ has a $L^3$-density with respect to the Lebesgue measure, then there exists a non-commutative derivation $\partial$ on $W^*(U)$ with domain $\mathbb{C}[ U,U^{-1}]$
such that $\partial(U)=1\otimes 1$ and such that $\partial^*(1\otimes 1)$ exists. Furthermore,  $\Delta(U)=1$.
\end{lemma}
\begin{proof}Thanks to \cite[Proposition 8.7]{Voi99}, $\diff^*(1\otimes 1)$ exists for the unique non-commutative derivation $\diff$ on $W^*(U)$ with domain $\mathbb{C}[ U,U^{-1}]$
such that $\diff(U)=1\otimes U$. Using Lemma \ref{lemma:Mai}, it implies that $\partial^*(1\otimes 1)$ exists for the non-commutative derivation $\partial:X\mapsto \diff(X)\sharp (1\otimes U^{-1})$, which is such that $\partial(U)=(1\otimes U)\sharp (1\otimes U^{-1})=1\otimes 1$. Proposition \ref{prop:Deltamaximal} yields $\Delta(U)=1$.
\end{proof}

An important realization of distributions in matrix form is via the free compression by  matrix units. Precise definitions and more information on this can be found in \cite[Lecture 14]{NS06}. 

\begin{theorem}\label{thm:free-compression}Let $X^{(1)},\ldots,X^{(d)}$ be in a tracial $W^*$-probability space. If $(X_{ij}^{(1)})_{i,j=1}^n$, $\ldots$, $(X_{ij}^{(d)})_{i,j=1}^n$ are given by the free compressions of $X^{(1)},\ldots,X^{(d)}$ by matrix units, and if
there exists non-commutative derivations $\partial^{(\ell)}$ ($1\leq \ell \leq d$) on $W^*(X^{(1)},\ldots,X^{(d)})$ with domain $\mathbb{C}\langle X^{(1)},\ldots,X^{(d)},(X^{(1)})^*,\ldots,(X^{(d)})^*\rangle$
such that $\partial^{(\ell)}(X^{(\ell')})=\delta_{\ell,\ell'}1\otimes 1$ and such that $(\partial^{(\ell)})^*(1\otimes 1)$ exists, then there exists the non-commutative derivations $\partial_{ij}^{(\ell)}$ on $\mathcal{M}:=W^*(X_{ij}^{(\ell)}:i,j,\ell)$ with domain $\mathbb{C}\langle X_{ij}^{(\ell)},(X_{ij}^{(\ell)})^*:i,j,\ell\rangle $
such that $\partial_{ij}^{(\ell)}(X_{kl}^{(\ell')})=\delta_{i,k}\delta_{j,l}\delta_{\ell,\ell'}1\otimes 1$ and such that $(\partial_{ij}^{(\ell)})^*(1\otimes 1)$ exists. Furthermore, $\Delta(X_{ij}^{(\ell)}:i,j,\ell)=dn^2$.
\end{theorem}

\begin{proof} The general case being a straightforward adjustment of the case $d=1$, let us prove the result for $d=1$.

Without loss of generality, we can assume that
$$X_{ij}=E_{1i}XE_{j1}\in (E_{11}W^*(X)*M_n(\mathbb{C})E_{11},n(\varphi*\tr_n))=(\mathcal{M},\tau),$$
where $W^*(X)$ is a finite $W^*$-probability space with state $\varphi$. 
We extend $\partial$ to a non-commutative derivative
$\diff$ on $W^*(X)*M_n(\mathbb{C})$ with domain $\mathbb{C}\langle X,X^*\rangle * M_n(\mathbb{C})$ by imposing $\diff_{|\mathbb{C}\langle X,X^*\rangle}=\partial$ and $\diff_{|M_n(\mathbb{C})}=0$. Thanks to \cite[Proposition 5.14]{Voi99}, we know that
$$(\tau*\tr_n)\circ \diff =(\tau*\tr_n)\circ \partial \circ E_{W^*(X)},$$
where $E_{W^*(X)}$ is the conditional expectation from $W^*(X)*M_n(\mathbb{C})$ to $E_{W^*(X)}$. As a consequence, we have $\diff^*(1\otimes 1)=\partial^*(1\otimes 1)$.

Now, let us define the non-commutative derivative $\diff_{ij}$ given by $\diff_{ij}(A)=\diff(A)\sharp(E_{i1}\otimes E_{1j})$. Using Lemma \ref{lemma:Mai}, we know that $\diff_{ij}^*(1\otimes 1)$ exists. They are defined on $W^*(X)*M_n(\mathbb{C})$ with domain $\mathbb{C}\langle X,X^*\rangle * M_n(\mathbb{C})$, but restricting them to the domain  $E_{11}\mathbb{C}\langle X,X^*\rangle *M_n(\mathbb{C})E_{11}$, we get the non-commutative derivatives $\partial_{ij}$ on $E_{11}W^*(X)*M_n(\mathbb{C})E_{11}$ such that $\partial_{ij}(X_{kl})=\delta_{i,k}\delta_{j,l}1\otimes 1$. We can compute
\begin{align*}
\langle \partial_{ij}(A),1\otimes 1\rangle_{L^2(\mathcal{M},\tau)\otimes L^2(\mathcal{M},\tau)}&=n^2\langle \diff_{ij}(A),1\otimes 1\rangle_{L^2(W^*(X)*M_n(\mathbb{C}),\varphi*\tr_n)\otimes L^2(W^*(X)*M_n(\mathbb{C}),\varphi*\tr_n)}\\
&=n^2\langle A,\diff_{ij}^*(1\otimes 1)\rangle_{L^2(W^*(X)*M_n(\mathbb{C}),\varphi*\tr_n)}\\
&=\langle A,n\diff_{ij}^*(1\otimes 1)\rangle_{L^2(\mathcal{M},\tau)}
\end{align*}
which shows that $\partial^*_{ij}(1\otimes 1)=n\diff_{ij}^*(1\otimes 1)$ exists.

Proposition \ref{prop:Deltamaximal} yields $\Delta(X_{ij}:i,j)=n^2$.
\end{proof}

The previous lemma and the previous theorem (for $d=1$) yield the following result.
\begin{corollary}\label{cor:free-compressions-Haar}
If $(U_{ij})_{i,j=1}^n$ are given by the free compressions of a unitary variable $U$ by matrix units, and if $U$ has a $L^3$-density with respect to the Lebesgue measure, then $\Delta(U_{ij}:i,j)=n^2$.
\end{corollary}
Note that free compressions of unitary variables correspond to quantum random variables on the unitary dual group, as studied in \cite{CU16}, and Haar unitaries correspond to the particular case of the Haar trace on the unitary dual group. Another particular case of interest is the unitary free Brownian motion $U(t)$, which has a bounded density whenever $t>0$ (see \cite[Proposition 1.6]{Voi99}): the free compressions $(U_{i,j}(t))_{i,j=1}^n$ of the unitary free Brownian motion have maximal $\Delta$  whenever $t>0$.

Our main use of the previous corollary will be to model free variables in matrix forms. For this we use Haar unitaries, written in matrix form $(U_{ij})_{i,j=1}^n$ as above, where the entries are given by free compressions of a Haar unitary $U$. Note that this means that the matrix $(U_{ij})_{i,j=1}^n$ is free from $M_n(\C)$, see \cite[Theorem 14.20]{NS06}. We will use such Haar unitaries to make other matrices free among themselves by conjugating with the Haar unitaries. For this it will be crucial that such $(U_{ij})_{i,j=1}^n$ are free from all other matrices where the entries are free from the $U_{ij}$. This fact is provided by the following proposition.
This can be derived as a direct consequence of Proposition 3.7 of \cite{NSS02}, but we prefer to give a selfcontained proof which is more in line with our setting. We will only need the special case where $\A$ is generated by a Haar unitary; for this case a direct proof can also be found in the recent preprint \cite[Theorem 17]{MPS2021}. 

\begin{prop}\label{prop:Rcyclic}
Let $(\M,\tau)$ be a tracial $W^*$-probability space and $\mathcal{A}$ and $\mathcal{B}$ be two $*$-free subalgebras of $\M$.

If $M_n( \mathcal{A})$ is $*$-free from $M_n(\C)$ in $(M_n(\M),\tr_n\otimes\tau)$, then $M_n( \mathcal{A})$ is also $*$-free from $M_n( \mathcal{B})$ in $(M_n(\M),\tr_n\otimes\tau)$.
\end{prop}
\begin{proof}
Without loss of generality, we can replace $\mathcal{B}$ by $u\mathcal{B}u^*$ with $u$ $*$-free from $\mathcal{A}$ and try to prove that $M_n( u\mathcal{B}u^*)=(I_n\otimes u)M_n( \mathcal{B})(I_n\otimes u)^*$ is $*$-free from $M_n( \mathcal{A})$.

We already know that $M_n( \mathcal{A})$ is $*$-free with amalgamation from $(I_n\otimes u)$ with respect to the conditional expectation $E_{M_n(\mathbb{C})}:=\id\otimes \tau.$ However, $E_{M_n(\mathbb{C})}$ is equal to $\tr_n\otimes \tau$ on $M_n( \mathcal{A})$ (because of the $*$-freeness of $M_n( \mathcal{A})$ from $M_n(\mathbb{C})$) and  $E_{M_n(\mathbb{C})}$ is equal to $\tr_n\otimes \tau$ on $W^*(I_n\otimes u)=I_n\otimes W^*(u)$. As a consequence, the $*$-freeness with amalgamation of $M_n( \mathcal{A})$ from $(I_n\otimes u)$ with respect to $E_{M_n(\mathbb{C})}$ implies the $*$-freeness with respect to $\tr_n\otimes \tau$.

Because we just proved that $(I_n\otimes u)$ is a Haar unitary which is $*$-free from $M_n( \mathcal{A})$, we get the $*$-freeness of $(I_n\otimes u)M_n( \mathcal{B})(I_n\otimes u)^*$ from $M_n( \mathcal{A})$.
\end{proof}

\section{Algebraic model for free variables with atoms}\label{sec:calculation}

In the first subsection, we will build an algebraic model for $*$-free normal random variables that have atoms of rational weights.
The main argument we will use is that we can compare the free situation with matrices over the free skew field, by the use of the compressions from the previous section.
This result is based on, but also largely extends, the main result in \cite{MSY19}, in the sense that now we can deal with variables with atoms of rational weights.
In particular, we will establish a rank equality between the analytic objects and the algebraic objects.
This allows us to describe the point spectrum distribution of rational expressions evaluated at normal variables with prescribed atoms.
Moreover, this rank equality also explains the strong Atiyah conjecture for the free products of finite 
groups as the analytic rank is equal to some normalized inner rank.

As our algebraic model contains no information about the non-atomic parts of the given random variables, a natural consequence is that for a rational expression $R$ and a tuple $X=(X_1,\dots,X_d)$ of $*$-free normal random variables such that $R(X)$ is well-defined, the point spectrum distribution of $R(X)$ does not depend on the non-atomic parts of $X_1,...,X_d$. 
In the second subsection we will show that this is even true if the weights of the atoms of $X_1,\dots,X_d$ are not necessarily rational; this will then give the full proof of Theorem \ref{thm:nonatomic}.
This phenomenon was known for the polynomial $P(x_1,x_2)=x_1+x_2$ according to \cite[Theorem 7.4]{BV98}.
Our theorem pushes such a result to general rational functions.
Moreover, the approximation to the general case actually allows us to calculate the point spectrum distribution.

\subsection{Algebraic model for free variables with atoms of rational weights}

Let $X=(X_1,\dots,X_d)$ be a tuple of $*$-free normal random variables in some $W^\ast$-probability space such that there exists $n\in \mathbb{N}$ with $\mu_{X_k}(\{\lambda\})\in \frac{1}{n}\mathbb{N}$ for each $1\leq k \leq d$ and each $\lambda\in \mathbb{C}$. For each $1\leq k \leq d$, the point spectrum distribution $\mu^p_{X_k}$ of $X_k$ can be written as
$$\sum_{\ell=1}^{\ell_k}\frac{1}{n}\delta_{\alpha_\ell^{(k)}},$$
where the multiset $\{\alpha_\ell^{(k)}\}_\ell\subset\mathbb{C}$ is the set of atoms of $X_k$.

Now, for each $1\leq k\leq d$, we consider indeterminate formal variables $u_{ij}^{(k)}$ ($1\leq i,j\leq n$) and $y^{(k)}_\ell$ $(1\leq \ell\leq n-\ell_k)$.
We denote by $\cP$ the unital complex algebra of noncommutative polynomials formed by these $dn^2+dn-\sum_k \ell_k$ variables and by $\cR$ its universal field of fractions.
For each $k=1,\dots,d$ we define the matrix
\begin{equation}\label{eq:formal unitary}
\bar{U}_k:=(u_{ij}^{(k)})_{i,j=1}^n\in M_n(\cP)
\end{equation}
and the diagonal matrix
\begin{equation}\label{eq:formal diagonal}
\bar{D}_k:=
\begin{pmatrix}\alpha_1^{(k)} &\cdots & 0&&&\\
\vdots &\ddots &\vdots&&\0&\\
0 &\cdots &\alpha_{\ell_k}^{(k)}&&&\\
&& &y^{(k)}_1 &\cdots & 0\\
&\0& &\vdots &\ddots &\vdots\\
&& &0 &\cdots &y_{n-\ell_k}^{(k)}
\end{pmatrix}\in M_n(\cP).
\end{equation}

\begin{theorem}\label{thm:main}
Let $X=(X_1,\dots,X_d)$ be a tuple of $*$-free normal random variables in some $W^\ast$-probability space $(\M,\tau)$
such that the weights of atoms are multiples of $\frac{1}{n}$ for a certain $n\in \mathbb{N}$.
Let $\bar{U}_k$ and $\bar{D}_k$, $k=1,\dots,d$ be defined as in \eqref{eq:formal unitary} and \eqref{eq:formal diagonal}.
We set $\bar{X}_k:=\bar{U}_k \bar{D}_k\bar{U}_k^{-1}\in M_n(\cR)$ and denote $\bar{X}:=(\bar{X}_1,\dots,\bar{X}_d)$.
Then, we have the following:
\begin{itemize}
\item For each integer $m$ and $A\in M_m(\C\left<x_1,\dots,x_d\right>)$, we have
\[
\rank(A(X))=\frac{1}{n}\rho_{\cR}(A(\bar{X}))\in [0,m]\cap\frac{1}{n}\mathbb{N}.
\]
\item For any rational expression $R$ in variables $x_1,\dots,x_d$, $R(X)$ exists in $L^0(\M)$ if and only if $R(\bar{X})$ exists in $M_n(\cR)$.
Moreover, provided well-definedness, we have for any $\lambda\in \mathbb{C}$,
$$\mu^p_{R(X)}(\{\lambda\})=1-\frac{1}{n}\rho_{\cR}(\lambda-R(\bar{X}))\in [0,1]\cap\frac{1}{n}\mathbb{N}.$$
\end{itemize}
\end{theorem}

\begin{proof}The second item is a consequence of the first one, thanks to Corollary~\ref{cor:isomorphic rational closure}. So let us prove the first item.

Let $U_1,\dots,U_d$ be $*$-free Haar unitary random variables in some $W^\ast$-probability space $(\M_1,\tau_1)$.
Let $U_{ij}^{(k)}\in\cN:=E_{11}\M_1\ast M_n(\C)E_{11}$, $1\leq i,j\leq n,1\leq k\leq d$
be their compressions by free matrix units as in Theorem \ref{thm:free-compression}.
According to Corollary \ref{cor:free-compressions-Haar}, $U_{ij}^{(k)}$, $1\leq i,j\leq n,1\leq k\leq d$ maximize their $\Delta$.

Moreover, the matrix
\[
\hat{U}_k:=(U_{ij}^{(k)})_{1\leq i,j\leq n}\in (M_n(\cN),\tr_n\circ\tau)
\]
is a Haar unitary random variable and thus $\rank(\hat{U}_k)=n$.
By Theorem \ref{th:Delta_maximal} we infer that $\rho(\bar{U}_k)=\rank(\hat{U}_k)=n$.
So we see that $\bar{U}_k$ is a full matrix and thus $\bar{U}_k^{-1}$ is matrix over $\cR$.
Therefore the evaluation of $\bar{U}_k^{-1}$ is exactly the operator $\hat{U}_k^\ast$ as $\hat{U}_k^\ast=(\hat{U}_k)^{-1}$.

Next, thanks to Lemma~\ref{Matrix_form:free_case}, we know that the distribution of $X_k$ is the one of
\[
D_k:=
\begin{pmatrix}\alpha_1^{(k)} &\cdots & 0&&&\\
\vdots &\ddots &\vdots&&\0&\\
0 &\cdots &\alpha_{\ell_k}^{(k)}&&&\\
&& &\gamma^{(k)}_1 &\cdots & 0\\
&\0& &\vdots &\ddots &\vdots\\
&& &0 &\cdots &\gamma_{n-\ell_k}^{(k)}
\end{pmatrix}
\]
where $\gamma^{(k)}_1,\dots,\gamma_{n-\ell_k}^{(k)}$ are normal random variables without atoms which are $*$-free in some tracial $W^\ast$-probability space.
Because $\gamma^{(k)}_1,\dots,\gamma_{n-\ell_k}^{(k)}$ are without atoms, we have $\Delta(\gamma^{k}_\ell)=1$, $\ell=1,\dots,n-\ell_k$.
Hence $\Delta(\gamma^{(k)}_1,\dots,\gamma_{n-\ell_k}^{(k)})=n-\ell_k$ by the additivity of $\Delta$ with respect to freeness (Proposition~\ref{prop:delta_free}).
Without loss of generality, we can assume that $(\gamma^{(k)}_1,\dots,\gamma_{n-\ell_k}^{(k)})_k$ are all living in the same tracial $W^\ast$-probability space $(\M_2,\tau_2)$ and are $*$-free, in such a way that
$$\Delta(\gamma_\ell^{(k)}:1\leq k \leq d, 1\leq \ell \leq n-\ell_k)=dn-\sum_{k=1}^d \ell_k.$$
Let us further consider the $W^\ast$-probability space $(\cN\ast\M_2,\tau\ast\tau_2)$, in which we have $$\Delta(U_{ij}^k,\gamma_\ell^{(k)}:1\leq k \leq d, 1\leq i,j\leq n, 1\leq \ell \leq n-\ell_k)=dn^2+dn-\sum_{k=1}^d \ell_k.$$
Thanks to Remark \ref{rem:Delta_maximal}, this maximality of $\Delta$ gives us an isomorphism $u^{(k)}_{ij}\mapsto U^{(k)}_{ij}$, $y^{(k)}_\ell\mapsto\gamma^{(k)}_\ell$ from the free field $\cR$ to its image such that the respective ranks are preserved.
Now let us define
\[
\hat{X}_k:=\hat{U}_kD_k\hat{U}_k^\ast\in(M_n(\cN\ast\M_2),\tr_n\circ(\tau\ast\tau_2))\quad \forall k=1,\dots,d,
\]
and denote $\hat{X}:=(\hat{X}_1,\dots,\hat{X}_d)$.
Then we have for any $A\in M_m(\Cx)$
\[
\rank(A(\hat{X}))=\rho_\cR(A(\bar{X})).
\]
Hence the first part of the desired result follows if $\rank(A(X))=\frac{1}{n}\rank(A(\hat{X}))$.
Such a rank equality naturally holds if $\hat{X}$ has the same distribution as the tuple $X$.

Let us verify that $\hat{X}$ has the same distribution as the original tuple $X$.
Note that for each $k=1,\dots,d$, $\hat{X}_k$ has the same distribution as $X_k$ and $\hat{U}_1,\dots,\hat{U}_d$ are freely independent Haar unitary random variables.
So it suffices to show that $\hat{U}_1,\dots,\hat{U}_d$ are $\ast$-free from $D_1,\dots,D_d$.
Since the family $\{\hat{U}_1\dots,\hat{U}_d\}$ is $\ast$-free from $M_n(\C)$ (which can be seen with the help of the identification $(\M_1,\tau_1)\ast(M_n(\C),\tr_n)\simeq (M_n(\cN),\tr_n\circ\tau)$) and the entries of $\hat{U}_1,\dots,\hat{U}_d$ are $\ast$-free from the entries of $D_1,\dots,D_d$, the desired result is a consequence of the Proposition \ref{prop:Rcyclic}.
\end{proof}

\begin{remark}\label{rem:normality}
Note that the requirement on the normality in the above theorem comes from the requirement on the normality for modelling a tuple of $\ast$-free random variables by their matrix model (see Lemma~\ref{Matrix_form:free_case}).
So actually we can also build an algebraic model mutatis mutandis if
\begin{enumerate}
    \item we put $Y_1,\dots,Y_d\in M_n(\M_2)$ for some integer $n$ and $W^\ast$-probability space $(\M_2,\tau_2)$ such the entries of $Y_1,\dots,Y_d$ are chosen from an algebra $\C\left<\gamma_1,\dots,\gamma_\ell\right>$ where $\Delta(\gamma_1,\dots,\gamma_\ell)=\ell$ (but $\gamma_1,\dots,\gamma_\ell$ and $Y_1,\dots,Y_d$ may not necessarily be normal);
    \item for $k=1,\dots,d$, we set $\hat{X}_k:=\hat{U}_kY_k\hat{U}_k^\ast$, where $\hat{U}_k$ is defined as in Theorem \ref{thm:main}
\end{enumerate}

In particular, we can also set $\ell=0$, i.e.,
\begin{enumerate}
    \item $Y_1,\dots,Y_d\in M_n(\C)$;
    \item for $k=1,\dots,d$, $\hat{X}_k:=\hat{U}_kY_k\hat{U}_k^\ast$, where $\hat{U}_k$ is defined as in Theorem \ref{thm:main}.
\end{enumerate}
Then an algebraic model can be built mutatis mutandis.
\end{remark}

Furthermore, we can actually model the free product of algebras rather than individual variables if these algebras are matrix algebras over some variables that have maximal $\Delta$.
We record this observation as the following proposition.

\begin{prop}\label{prop: free product of matrix algebras}
Let $\A_1,\dots,\A_d$ be subalgebras of $M_n(\C\left<\gamma_1,\dots,\gamma_\ell\right>)$, where $\ell$ is an integer and $\gamma_1,\dots,\gamma_\ell$ are in some $W^\ast$-probability space $(\M_2,\tau_2)$ such that $\Delta(\gamma_1,\dots,\gamma_\ell)=\ell$.
We denote by $\Phi:\C\left<y_1,\dots,y_\ell\right>\rightarrow\C\left<\gamma_1,\dots,\gamma_\ell\right>$ the isomorphism determined by $\Phi(y_i)=\gamma_i$, $i=1,\dots,\ell$, where $y_i$'s are formal variables.
For $k=1,\dots,d$, we set $\bar{\A}_k:=\bar{U}_k\Phi^{-1}(\A_k)\bar{U}_k^{-1}$ which are subalgebras of $M_n(\cR)$, where $\cR$ and $\bar{U}_k$ are defined as in Theorem \ref{thm:main}.
Then $\ast_{k=1}^d\A_k$ is isomorphic to the subalgebra generated by $\bar{\A}_1,\dots,\bar{\A}_d$ with rank preserved.
\end{prop}

\begin{proof}
For $k=1,\dots,d$, we set $\hat{\A}_k:=\hat{U}_k\A_k\hat{U}_k^\ast$ which are subalgebras of $M_n(\cN\ast\M_2)$, where $\cN$ and $\hat{U}_k$ are defined as in Theorem \ref{thm:main}.
For each $k=1,\dots,d$, let $\Psi_k:\A_k\rightarrow M_n(\cN\ast\M_2)$ be the homomorphism determined by
\[
\Psi_k(A):=\hat{U}_kA\hat{U}_k^\ast,\quad,\forall A\in\A_k.
\]
Clearly the images $\Psi_1(\A_1),\dots,\Psi_d(\A_d)$ are $\ast$-free in $(M_n(\cN\ast\M_2),\tr_n\circ(\tau\ast\tau_2))$.
Thus $\ast_{k=1}^d\A_k$ is isomorphic to the subalgebra generated by $\Psi_1(\A_1),\dots,\Psi_d(\A_d)$ with preserved ranks.
Finally, the isomorphism with preserved ranks between the subalgebra generated by $\Psi_1(\A_1),\dots,\Psi_d(\A_d)$ and the subalgebra generated by $\bar{\A}_1,\dots,\bar{\A}_d$ can be seen by a similar argument as in Theorem \ref{thm:main}.
\end{proof}

\subsection{Atiyah conjecture and division closure}\label{subsect:Atiyah}

In this subsection, let us briefly discuss the link between our results obtained so far and the Atiyah conjecture.

First, let us explain the consequence of the our results for the Atiyah conjecture of discrete groups.
For a group $G$, we denote by $\lcm(G)$ the least common multiple of the orders of finite subgroups of $G$, if it exists (with a convention that $\lcm(G)=1$ if $G$ is torsion-free).
For a group $G$ with $\lcm(G)<\infty$, we say that $G$ satisfies the \emph{strong Atiyah conjecture} (see, for example, \cite[Conjecture 10.2]{Luc02}) if for every $m\in\bN$ and every matrix $A\in M_m(\C G)$ we have
\[
\rank A\in\frac{1}{\lcm(G)}\Z.
\]

Let us consider finite groups $G_1,\dots,G_d$ with orders $n_1,\dots,n_d$ respectively.
For each $k=1,\dots,d$, we regard the group algebra $\C G_k$ as the subalgebra of $M_{n_k}(\C)$ via the left-regular representation.
In order to apply Proposition \ref{prop: free product of matrix algebras}, we need to put $\C G_1,\dots,\C G_d$ (considered as matrix subalgebras) into a matrix algebra $M_n(\C)$ for some integer $n$.
Naturally, the smallest convenient choice for this $n$ is $\lcm(n_1,\dots,n_d)$.
Hence we can have the following corollary, thanks to Proposition \ref{prop: free product of matrix algebras}.

\begin{corollary}
Let $G_1,\dots,G_d$ be finite groups with orders $n_1,\dots,n_d$. Then their free product $G:=G_1\ast G_2\ast\dots\ast G_d$ satisfies the strong Atiyah conjecture.
Moreover, for any matrix $A\in M_m(\C G)$, $\rank A$ is equal to the normalized inner rank of some $nm\times nm$ matrix over some free field $\cR$ (as constructed in Proposition \ref{prop: free product of matrix algebras}).
\end{corollary}

Actually, it is known that the strong Atiyah conjecture is stable with respect to finite free products (see, for example, \cite[Theorem 7.7]{Rei98}).
Therefore our corollary can be understood as a free probabilistic proof or explanation of the strong Atiyah conjecture for free products of finite groups.

Let us make a further remark on the rational/division closure of $\C G$.
We actually have the isomorphism between the rational/division closure of $\C G$ in the algebra of affiliated operators with respect to the group von Neumann algbera $\cL(G)$ and the rational/division closure of the subalgebra generated by $\bar{\C G_k}:=\bar{U}_k\C G_k\bar{U}^{-1}_k$ (as defined in Proposition \ref{prop: free product of matrix algebras}) due to the rank equality.
This can be seen similarly with the help of proof of Corollary \ref{cor:isomorphic rational closure}.
So we see that the division closure of $\C G$ is a subalgebra of a matrix algebra over some free field.
This observation seems to be related to the known result that if a group $G$ satisfies the strong Atiyah conjecture, then its division closure is some direct sum of finite matrix algebras over skew fields (see, for example, \cite{KLS17}).

Finally, let us discuss a bit the general case that may not come from groups.
In \cite{SS15}, a similar notion was introduced under the name of Atiyah property in order to include random variables that may not come from groups.
Let us introduce this property with some slight modification.

\begin{definition}
Let $\A$ be a $\ast$-subalgebra of a tracial $W^\ast$-probability space $(\M,\tau)$ and let $\Gamma$ be an additive subgroup of $\R$ containing $\Z$.
We say that $\A$ satisfies the \emph{Atiyah property with $\Gamma$} if for every $m\in\bN$ and every matrix $A\in M_m(\A)$ we have that
$
\rank A\in\Gamma$.
\end{definition}

Then we have the following as a corollary of Theorem \ref{thm:main}.
\begin{corollary}\label{cor:Atiyah}
Let $X=(X_1, \ldots,X_d)$ be a tuple of $*$-free normal random variables with compactly supported probability measures $\mu_k$ as distribution, respectively.
Suppose that for each $k=1,\dots,d$, there exists $n_k\in \mathbb{N}$ such that all the atoms of $\mu_k$ have weights contained in $\frac{1}{n_k}\mathbb{N}$.
Then $\C\langle X\rangle$ has the Atiyah property with group $\frac{1}{n}\bN$, where $n:=\lcm(n_1,\dots,n_d)$.
\end{corollary}

Recall that the result in \cite[Theorem 3.4]{SS15} says that $\C\langle X\rangle$ has the Atiyah property with group $\frac{1}{n}\bN$, where $n:=n_1\cdots n_d$.
We improve the denominator for the weight of the atoms from the product $n_1\cdots n_d$ to their least common multiple $\lcm(n_1,\dots,n_d)$.
Furthermore, we also improve their result in the sense that we can extract precise information on the positions and the weights of atoms via the calculation of the inner rank of matrices in formal variables.

For a pair of self-adjoint free variables $X=(X_1,X_2)$, Corollary~\ref{cor:Atiyah} has already been established in \cite{BBL2021}, as the following result shows. Interestingly, this result is also valid for free variables with atoms of irrational weights. In this respect, it is much more general than \cite[Theorem 3.4]{SS15} or Corollary~\ref{cor:Atiyah}. We are not able to attain this level of generality with our approach because we did not succeed to obtain the Atiyah property via an approximation procedure. 

\begin{theorem}[Corollary 3.3 of \cite{BBL2021}]
Let $X=(X_1, X_2)$ be a tuple of free self-adjoint random variables with compactly supported probability measures $\mu_1,\mu_2$ as distribution, respectively.
Suppose there exists an additive subgroup $\Gamma$ of $\mathbb{R}$ containing $\mathbb{Z}$ such that all the atoms of $\mu_1$ and $\mu_2$ have weights contained in $\Gamma$.
Then $\C\langle X\rangle$ has the Atiyah property with group $\Gamma$.
\end{theorem}

\begin{proof}Because  \cite[Corollary 3.3]{BBL2021} is not stated in this terms, let us give some details. Let $A\in M_m(\C\langle X\rangle)$. The matrix $A$ can be written as a polynomial expression in $I_m\otimes X_1$, $I_m\otimes X_2$ and some matrices $M\subset M_m(\C)$. That is, let $R$ be a rational expression of formal variables $x_1,x_2$ as well as some variables corresponding to the finite set $M$ such that $R(I_m\otimes X_1,I_m\otimes X_2,M)=A$.
Now let $(u,Q,v)$ be a formal linear representation of $R$, of dimension $n$ and with display $L$ in the sense of Definition \ref{def:formal linear representation}. Thanks to Lemma \ref{lem:inner rank of display}, we have
$$\rank(A)=\rank(R(I_m\otimes X_1,I_m\otimes X_2,M))=\rank(L(I_m\otimes X_1,I_m\otimes X_2,M))-mn.
$$
The matrix $L(I_m\otimes X_1,I_m\otimes X_2,M))\in M_{mn}(\C\langle X\rangle)$ being linear, \cite[Corollary 3.3]{BBL2021} allows to say that its rank is an affine combination (with integer coefficients) of the weights of the atoms of $\mu_1$ and $\mu_2$. As a consequence, the rank of $L(I_m\otimes X_1,I_m\otimes X_2,M))\in M_{mn}(\C\langle X\rangle)$ belongs to $\Gamma$, and so does the rank of $A$.
\end{proof}

\subsection{Approximation for general free normal variables}\label{subsect:approximation of free variables}
From Theorem \ref{thm:main}, we see that the position and the size of an atom of $R(X_1,\dots,X_d)$ do not depend on the non-atomic parts of the marginal distributions, if we assume that the variables $X_i$'s have rational weights for their atoms. 
Let us remove in the following this hypothesis on rational weights.
Theorem \ref{thm:nonatomic} is the second item in the following theorem.

\begin{theorem}
Let $X=(X_1,\dots,X_d)$ and $Y=(Y_1,\dots,Y_d)$ be two tuples of $\ast$-free normal variables in tracial $W^\ast$-probability spaces such that, for each $1\leq i\leq d$, we have $\mu^p_{X_i}=\mu^p_{Y_i}$. 
Then we have the following statements.
\begin{itemize}
\item For each integer $m$ and each $A\in M_m(\Cx)$, we have
\[\rank(A(X))=\rank(A(Y)).\]
\item For each rational expression $R$, $R(X)$ is well-defined if and only if $R(Y)$ is well-defined, and in this case,
\[
\mu^p_{R(X)}=\mu^p_{R(Y)}.
\]
\end{itemize}
\end{theorem}

\begin{proof}The second item is a consequence of the first one, thanks to Corollary~\ref{cor:isomorphic rational closure}. So let us prove the first item.

Thanks to Lemma~\ref{lemma:approximation}, we can consider, for each $N$, two tuples of $*$-free normal random variables $X^{(N)}=(X_1^{(N)},...,X_d^{(N)})$ and $(Y_1^{(N)},...,Y_d^{(N)})$  such that:
\begin{itemize}
\item 
$\rank(X_i^{(N)}-X_i)\leq C_X/N$ and $\rank(Y_i^{(N)}-Y_i)\leq C_Y/N$, for all $i$;
\item  for all $\lambda\in \mathbb{C}$, we have  $\mu_{X_k^{(N)}}(\{\lambda\})=\mu_{Y_k^{(N)}}(\{\lambda\})$;
\item the weights  of the atoms of $X^{(N)}_k$ and $Y^{(N)}_k$ are  rational.
\end{itemize}
Lemma~\ref{lem:comparison of Sylvester rank} says that there exists some constant $c>0$ such that
\[
|\rank(A(X)-A(X^{(N)}))|\leq c/N\quad \text{and} \quad|\rank(A(Y)-A(Y^{(N)}))|\leq c/N.
\]
Due to the subadditivity of $\rank$, we have
\[
|\rank(A(X))-\rank(A(X^{(N)}))|\leq c/N\quad\text{and}\quad|\rank(A(Y))-\rank(A(Y^{(N)}))|\leq c/N.
\]
Moreover, according to Theorem \ref{thm:main}, we have
\[
\rank(A(X^{(N)}))=\rank(A(Y^{(N)})).
\]
Hence we conclude that
\[
|\rank(A(X))-\rank(A(Y))|\leq 2c/N.
\]
Letting $N$ tend to $\infty$, we get that
\[
\rank(A(X))=\rank(A(Y)).
\]
\end{proof}

Recall that Remark \ref{rem:rational closure by dist} says that the rational closure of a tuple $X=(X_1,\dots,X_d)$ of random variables only depends on the $\ast$-distribution of $X$.
In the case that $X$ is a tuple of free random variables, then this can be further improved in the sense that the rational closure of $X$ only depends on the point spectrum distribution of each variable $X_i$.
That is, with the help of Corollary~\ref{cor:isomorphic rational closure}, we get the following corollary.

\begin{corollary}
Let $X=(X_1,\dots,X_d)$ and $Y=(Y_1,\dots,Y_d)$  be two $d$-tuples of $*$-free normal variables in tracial $W^*$-probability spaces such that, for each $1\leq i \leq d$, we have
$\mu^p_{X_i}=\mu^p_{Y_i}.$
Then the rational closure of $X$ is isomorphic to the rational closure of $Y$.
\end{corollary}

We know that for the polynomial $P(x_1,x_2)=x_1+x_2$, \cite[Theorem 7.4]{BV98} gave a complete description of the point spectrum distribution of $P(X_1,X_2)$ for any freely independent self-adjoint random variables with prescribed atoms. This result was recently promoted to a matrix-version in \cite{BBL2021} -- this gives, via linearization, a possible analytic route for dealing with the the point spectrum distribution for arbitrary polynomials in free variables.

In contrast to this analytic approach, we are able to provide here an algebraic method for the calculation of the point spectrum distribution.
Of course, since we need to start with random variables with atoms of rational weights, we need to control the approximation process.
First, notice that in our approximation scheme (Lemma \ref{lemma:approximation}) we have $\mu^p_{X_i^{(N)}}\leq\mu^p_{X_i}$, $i=1,\dots,d$. This means that for a rational expression $R$ for which $R(X)$ is well-defined, all the $R(X^{(N)})$ are actually also well-defined, by Theorem \ref{thm:comparison-rational functions}.
Furthermore, we have seen that  $\rank(R(X^{N}))$ converges to $R(X)$.
Let us record this observation as a remark.

\begin{remark}\label{rem:approximation}
Let $X=(X_1,\dots,X_d)$ be a tuple of $\ast$-free normal variables in some tracial $W^\ast$-probability space.
Then we can construct a sequence $X^{(N)}=(X^{N}_1,\dots,X^{(N)}_d)$ of tuples of $\ast$-free normal variables such that:
\begin{itemize}
\item each atom of $\mu^p_{X^{(N)}_k}$ has a rational weight and $\mu^p_{X^{N}_i}\leq\mu^p_{X_i}$;
\item for each matrix $A$ over $\Cx$ we have $$\rank(A(X))=\lim_{N\rightarrow\infty}\rank(A(X^{(N)})),$$ where $\rank(A(X^{(N)}))$ can be calculated algebraically as in Theorem \ref{thm:main};
\item for each rational expression $R$ such that $R(X)$ is well-defined, we have
$$\rank(R(X))=\lim_{N\rightarrow\infty}\rank(R(X^{(N)})),$$ where $R(X^{(N)})$ is well-defined and $\rank(R(X^{(N)}))$ can be calculated algebraically as in Theorem \ref{thm:main}.
\end{itemize}
\end{remark}

As we will show in Section \ref{sect:application} and \ref{sect:commutator}, this allows us to extend our results to arbitrary weights. Thus, in principle, we can calculate atoms of any polynomials (even rational functions) by doing algebraic inner rank calculations in the free field, according to
Theorem \ref{thm:main}.
This calculation of the inner rank for a given polynomial can, for example, be done by a noncommutative Gauss elimination (that is, a diagonalization by row and column operators).
Alternatively, we can also relate the calculation of the inner rank of a given polynomial to the calculation of the inner rank of its display (thanks to Lemma \ref{lem:inner rank of display}).
The calculation of this inner rank might then be carried out with the help of various other means, see, for example, the notions and tools in \cite{GGOW16}.
However, whereas all those algebraic tools for the calculation of inner ranks as well as the analytic approach from \cite{BBL2021} have great conceptual power and can be used to develop algorithmic or numerical approaches for doing calculations in concrete examples, it is very hard to get general explicit results from there for the size of atoms for bigger classes of polynomials. 

In our approach, on the other side, we can make this additional step, since our representation of freely independent random variables via elements in the free field  allows us the compare the free situation with other situations (see Section \ref{sec:comparison}), resulting in, often quite sharp, 
upper and lower bounds for the weights of atoms for general classes of polynomials.

\section{Comparison of atoms for rational functions: free variables versus general variables}\label{sec:comparison}

The main goal of this section is to prove the following theorem.

\begin{theorem}\label{thm:comparison-rational functions}
Let $Y=(Y_1,\dots,Y_d)$ be a tuple of normal random variables in some tracial $W^\ast$-probability space.
Let $X=(X_1,\dots,X_d)$ be a tuple of $*$-free normal random variables in some tracial $W^\ast$-probability space such that, for each $k=1,\dots,d$, we have  $$\mu_{X_k}^p\leq\mu_{Y_k}^p.$$
Then we have the following.
\begin{enumerate}
\item For any $A\in M_n(\C\left<x_1,\dots,x_d\right>)$, we have
\[\rank(A(Y))\leq\rank(A(X)).\]
\item
For any rational expression $R$ in variables $x_1,\dots,x_d$ such that $R(Y)$ is well-defined as an unbounded operator, $R(X)$ is also well-defined and
$$\mu^p_{R(X)}\leq\mu^p_{R(Y)}.$$
\end{enumerate}
\end{theorem}

In Subsection \ref{subsect:universality} we apply this theorem to deduce the universality property of the rational closures given by $*$-free normal random variables and in Section \ref{subsect:matrix approximation} we consider matrix approximations.

Theorem~\ref{thm:comparison-rational functions} shows that $*$-free variables realize an extreme situation for tuples of random variables with prescribed marginals: atoms of polynomials can not be smaller than the unavoidable sizes at  the unavoidable  locations of the $*$-free case. It is tempting to wonder how generic is this situation: in other words, how often do the inequalities in Theorem~\ref{thm:comparison-rational functions} hold to be equalities?
Following the work of \cite{BCDLT2010}, where this genericity aspects are under consideration, one could attack this question by rotating each variable by unitary conjugation, and asks about the property of the set of unitaries realizing equalities in Theorem~\ref{thm:comparison-rational functions}: is it norm-dense, is it a $G_\delta$-set? We leave those questions for further investigations.
\subsection{Proof of Theorem~\ref{thm:comparison-rational functions}} 
First we prove an intermediate step for Theorem \ref{thm:comparison-rational functions}.
We start with a slightly different setting, where the analytic distribution of each random variable has finitely many atoms of rational weights.

\begin{lemma}\label{lem:comparison-rational functions}
Let $Y=(Y_1,\dots,Y_d)$ be a tuple of normal random variables in some $W^\ast$-probability space $(\M,\tau)$.
Let $X=(X_1,\dots,X_d)$ be a tuple of $*$-free normal random variables satisfying the following conditions.
\begin{itemize}
\item For each $k=1,\dots,d$, $X_k$ has finitely many atoms and each atom is of rational weight.
\item For each $k=1,\dots,d$, we have $\mu_{X_k}^p\leq\mu_{Y_k}^p$.
\end{itemize}
Then we have the following.
\begin{enumerate}
\item For any $A\in M_n(\C\left<x_1,\dots,x_d\right>)$, we have
\[\rank(A(Y))\leq\rank(A(X)).\]
\item
For any rational expression $R$ in variables $x_1,\dots,x_d$ such that $R(Y)$ is well-defined as an unbounded operator, $R(X)$ is also well-defined and
$$\mu^p_{R(X)}\leq\mu^p_{R(Y)}.$$
\end{enumerate}
\end{lemma}

\begin{proof}The second item is a consequence of the first one, thanks to Proposition
\ref{prop:comparison rational closure}. Let us prove the first item.

First, let us replace $X$ in this equality by $\bar{X}$, which is an algebraic model for $X$.
According to the assumption there exists $r\in \mathbb{N}$ such that $\mu_{X_k}(\{\lambda\})\in \frac{1}{r}\mathbb{N}$ for each $1\leq k \leq d$ and $\lambda\in \mathbb{C}$.
That is, for each $1\leq k \leq d$, the point spectrum distribution $\mu^p_{X_k}$ of $X_k$ can be written as
$$\mu^p_{X_k}=\sum_{\ell=1}^{\ell_k}\frac{1}{r}\delta_{\alpha_\ell^{(k)}},$$
where the multiset $\{\alpha_\ell^{(k)}\}_\ell\subset\mathbb{C}$ is the set of atoms of $X_k$.
For each $1\leq k\leq d$, we consider variables $u_{ij}^{(k)}$ ($1\leq i,j\leq r$) and $y^{(k)}_\ell$ $(1\leq \ell\leq r-\ell_k)$.
We denote by $\cP$ the polynomial ring formed by these $dr^2+dr-\sum_k \ell_k$ variables and by $\cR$ its universal field of fractions.
For each $k=1,\dots,d$ we define the matrix
\[
\bar{U}_k:=(u_{ij}^{(k)})_{i,j=1}^r\quad\text{and}\quad\bar{D}_k:=\diag(\alpha^{(k)}_1,\dots,\alpha^{(k)}_{\ell_k},y^{(k)}_1,\dots,y^{(k)}_{r-\ell_k})\quad\in M_r(\cP).
\]
as in \eqref{eq:formal unitary} and \eqref{eq:formal diagonal}, and $\bar{X}_k:=\bar{U}_k\bar{D}_K\bar{U}_k^{-1}$as in Theorem~\ref{thm:main}.
For simplicity, let us denote $\bar{U}=(\bar{U}_1,\dots,\bar{U}_d)$, $\bar{D}=(\bar{D}_1,\dots,\bar{D}_d)$ and $\bar{X}=(\bar{X}_1,\dots,\bar{X}_d)$.

Thanks to Theorem \ref{thm:main}, we have
$$\rank(A(X))=\frac{1}{r}\rho_{\mathcal{R}}(A(\bar{X})).$$
Next, we regard $A(\bar{X})$ as a rational expression in $I_n\otimes \bar{U_1},\dots,I_n\otimes \bar{U}_d$ and $I_n\otimes \bar{D}_1,\dots,I_n\otimes \bar{D}_d$ and some matrices $M\subset M_n(\mathbb{C})\otimes I_r$ (since $A$ is a polynomial in $I_n\otimes x_1,\dots,I_n\otimes x_d$ and some matrices in $M_n(\C)\otimes I_r$).
That is, let $R$ be a rational expression of formal variables $u_1,\dots,u_d$ and $z_1,\dots,z_d$ as well as some variables corresponding to the finite set $M$ such that $R(I_n\otimes \bar{U},I_n\otimes \bar{D},M)=A(\bar{X})$.
Now let $(u,Q,v)$ be a formal linear representation of $R$, of dimension $m$ and with display $L$ in the sense of Definition \ref{def:formal linear representation}. Thanks to Lemma \ref{lem:inner rank of display}, we have $(I_n\otimes \bar{U},I_n\otimes \bar{D},M)\in\dom_{M_{nr}(\cR)}(R)$ as $\bar{X}\in\dom_{M_{nr}(\cR)}(A)$ (or equivalently, $X\in\dom_{L^0(\M)}(A)$), and we have
$$\rho_{\mathcal{R}}(A(\bar{X}))=\rho_{\mathcal{R}}(R(I_n\otimes \bar{U},I_n\otimes \bar{D},M))=\rho_{\mathcal{R}}(L(I_n\otimes \bar{U},I_n\otimes \bar{D},M))-mnr.
$$

Now, thanks to Proposition~\ref{Matrix_form:general_case}, for each $1\leq k \leq d$, we can write $I_r\otimes Y_k \in M_r(\mathcal{M})$ as $I_r\otimes Y_k=U_kD_kU_k^*$ where $U_k\in M_r(\mathcal{M})$ is a unitary matrix and 
$$D_k=\diag(\alpha^{(k)}_1,\dots,\alpha^{(k)}_{\ell_k},\gamma^{(k)}_1,\dots,\gamma^{(k)}_{r-\ell_k})\in M_r(\M).$$
We have
$$\rank(A(Y))=\frac{1}{r}\rank(A(I_r\otimes Y)).$$
We denote $U=(U_1,\dots,U_d)$ and $D=(D_1,\dots,D_d)$.
Similar, with the help of Lemma \ref{lem:inner rank of display}, if $A(I_r\otimes Y)=R(U,D,M)$ is well-defined, then 
\[
\rank(A(I_r\otimes Y))=\rank(R(I_n\otimes U,I_n\otimes D,M))=\rank(L(I_n\otimes U,I_n\otimes D,M))-mnr.
\]
and
so the desired claim follows if
$$\rank(L(I_n\otimes U,I_n\otimes D,M))\leq \rho_{\mathcal{R}}(L(I_n\otimes \bar{U},I_n\otimes \bar{D},M)).$$
In order to see this let us denote by $\Phi:\cP\rightarrow\M$ the unital ring homomorphism given by extending $\Phi(u_{ij}^{(k)})=U_{ij}^{(k)}$ and $\Phi(y_\ell^{(k)})=\gamma_\ell^{(k)}$.
So $\Phi(\bar{U}_k)=U_k$, $\Phi(\bar{D}_k)=D_k$ and thus it boils down to prove that
$$\rank(\Phi(L(I_n\otimes \bar{U},I_n\otimes \bar{D},M)))\leq \rho_{\mathcal{R}}(L(I_n\otimes \bar{U},I_n\otimes \bar{D},M)).$$
Note that $\Phi$ extends to a homomorphism from $M_{nmr}(\cP)$ to $M_{nmr}(\M)$ and thus for any matrix $A$ over $\cP$,
\[
\rho_\M(\Phi(A)))\leq\rho_{\mathcal{R}}(A)).
\]
Therefore, finally, the proof is completed by the fact
\[
\rank(\Phi(A))\leq\rho_{\M}(\Phi(A))
\]
(for a proof of this fact, see the last paragraph of the proof of \cite[Theorem 5.6]{MSY19}).
\end{proof}

\begin{proof}[Proof of Theorem \ref{thm:comparison-rational functions}]The second item is a consequence of the first one, thanks to Proposition
\ref{prop:comparison rational closure}.

Using Lemma~\ref{lemma:approximation}, we can consider, for each integer $N$, a tuple of $*$-free normal random variables $X^{(N)}=(X_1^{(N)},\dots,X_d^{(N)})$  such that for $k=1,\dots,d$
\begin{itemize}
\item 
$\rank(X_k^{(N)}-X_k)\leq C_X/N$;
\item  for all $\lambda\in \mathbb{C}$, we have  $\mu_{X_k^{(N)}}(\{\lambda\})\leq\mu_{X_k}(\{\lambda\})\leq\mu_{Y_k}(\{\lambda\})$;
\item the atoms of $X^{(N)}_k$ are finite and have rational weights.
\end{itemize}
For any matrix $A$ over polynomials $\C\left<x_1,\dots,x_d\right>$, Lemma~\ref{lem:comparison-rational functions} yields
\[\rank(A(Y))\leq\rank(A(X^{(N)})),\]
and Lemma~\ref{lem:comparison of Sylvester rank} says that
\[|\rank(A(X))-\rank(A(X^{(N)}))|\leq c/N,\]
for some constant $c>0$.
This means that
\[\rank(A(Y))\leq\rank(A(X))+c/N.\]
Letting $N$ tend to $\infty$, we get that
$\rank(A(Y))\leq\rank(A(X))$.
\end{proof}

\subsection{Universality of the free case}\label{subsect:universality}

Here, we want to show that the rational closure of a free tuple of random variables with prescribed point spectrum distributions is the most general one in the sense that it satisfies a universal property.

In Subsection \ref{sec:category of Cx-algebras}, we defined the category of $\Cx$-algebras whose arrows are given by specializations. Let $m_1,\dots,m_d$ be $d$ sub-probability measures.
We consider the sub-category $\mathcal{C}_m$ of $\Cx$-algebras given by rational closures of $d$-tuples $X=(X_1,\ldots,X_d)$ of normal random variables in tracial $W^*$-probability spaces, such that, for all $1\leq k\leq d$, we have
$$m_k\leq \mu_{X_k}^p.$$
The rational closure of a free-tuple is universal in the sense that it is an initial object in the category $\mathcal{C}_m$, as shown in the following corollary.

\begin{corollary}
Let $X=(X_1,\dots,X_d)$ be a tuple of $*$-free normal random variables such that, for each $k=1,\dots,d$, $m_k= \mu_{X_k}^p$. Let $Y=(Y_1,\dots,Y_d)$ be a tuple of normal random variables in some $W^\ast$-probability space $(\M,\tau)$ such that, for all $1\leq k\leq d$, we have
$$m_k\leq \mu_{Y_k}^p.$$
Then, there exists a specialization from the rational closure of $X$ to the rational closure of $Y$, and in particular, the rational closure of $Y$ is a quotient of a subalgebra of the rational closure of $X$.
\end{corollary}
\begin{proof}
It is a direct consequence of Proposition \ref{prop:comparison rational closure} and Theorem \ref{thm:comparison-rational functions}.
\end{proof}

As pointed out in Remark \ref{rem:normality}, the normality required in the above Corollary is due to the way how we model the random variables $X_1,\dots,X_d$.
So in the simple case that $X_1,\dots,X_d$ are free matrices in $\ast_{i=1}^d M_n(\C)$, we can remove the normality requirement on $X_1,\dots,X_d$ (and thus the normality of $Y_1,\dots,Y_d$ is also not necessary).

\subsection{Approximation with matrices}\label{subsect:matrix approximation}

Let $X=(X_1,\dots,X_d)$ be a tuple of $*$-free normal random variables.
(Again, as pointed out in Remark \ref{rem:normality}, normality can be removed if $X_1,\dots,X_d$ are already in the matrix form, for example, if they are free matrices in $\ast_{i=1}^d M_n(\C)$.)
We define for each $\varepsilon>0$
\[
\mathcal{X}_{m,\varepsilon}:=\{Y=(Y_1,\dots,Y_d)\in M_m(\C)^d:\mu_{X_k}(\lambda)-\varepsilon\leq\mu_{Y_k}(\lambda),\ \forall\lambda\in\C\}
\]
and $\mathcal{X}_{\varepsilon}:=\coprod_{m=1}^\infty\mathcal{X}_{m,\varepsilon}$.
Then, thanks to Theorem \ref{thm:comparison-rational functions}, for each polynomial $P\in\Cx$
\[
\mu^p_{P(X_\varepsilon)}\leq\inf_{Y\in\mathcal{X}_\varepsilon}\mu^p_{P(Y)},
\]
where $X_\varepsilon$ is a copy of $X$ with moving a $\varepsilon$ weight from each atom of $X_k$, $k=1,\dots,d$ (as in Lemma \ref{lemma:approximation}).
Then by an application of Theorem \ref{BVrational} we can see that
\[
\mu^p_{P(X)}\leq\lim_{\varepsilon\rightarrow 0}\inf_{Y\in\mathcal{X}_\varepsilon}\mu^p_{P(Y)},
\]

Now, we prove that this inequality actually is an equality.

\begin{prop}\label{bound from matrices}
Let $X=(X_1,\dots,X_d)$ be a tuple of $*$-free normal random variables.
Then for any $P\in\Cx$
\[
\mu^p_{P(X)}=\lim_{\varepsilon\rightarrow 0}\inf_{Y\in\mathcal{X}_\varepsilon}\mu^p_{P(Y)}
\]
\end{prop}

\begin{proof}
We only need to prove that for every $\varepsilon>0$, $\inf_{Y\in\mathcal{X}_\varepsilon}\mu^p_{P(Y)}\leq \mu^p_{P(X)}$.
For any operator $Z$ in a tracial $W^\ast$-probability space, we know that
\begin{align*}
1-\mu^p_Z(\{\lambda\})&=\rank(\lambda-Z)
=\frac{1}{2}\rank\begin{pmatrix}\lambda-Z & 0\\0 &\overline{\lambda}-Z^\ast\end{pmatrix}=\frac{1}{2}\rank\begin{pmatrix}0 & \lambda-Z\\\overline{\lambda}-Z^\ast &0 \end{pmatrix}.
\end{align*}
Let us denote
\[
A_\lambda:=\begin{pmatrix}0 & \lambda-P\\\overline{\lambda}-P^\ast & 0\end{pmatrix}\in M_2(\C\langle x_1,\dots,x_d,x^\ast_1,\dots,x^\ast_d\rangle).
\]
Then it suffices to prove that $\inf_{Y\in\X_\varepsilon}\mu_{A_\lambda(Y)}(\{0\})\leq\mu_{A_\lambda(X)}(\{0\})$.

For that purpose, first for each $k=1,\dots,d$ we can take a sequence $D^{(m)}_k$ of diagonal matrices such that $D^{(m)}=(D_1^{(m)},\dots,D_d^{(m)})\in\X_{m,\varepsilon}$ and $\mu^p_{D^{(m)}_k}$ converge weakly towards $\mu_{X_k}$.
Then we define $Y^{(m)}:=(Y^{(m)}_1,\dots,Y^{(m)}_d)\in\X_{m,\varepsilon}$ with $Y^{(m)}:=U^{(m)}_kD^{(m)}_kU^{(m)*}_k$, where $U^{(m)}_1,\dots,U^{(m)}_d$ are a sample of independent Haar unitary random matrices of dimension $m$.
Clearly we have that $Y^{(m)}$ converges in $\ast$-distribution towards $X$ and thus
$\mu_{A_\lambda(Y^{(m)})}$ converges weakly towards $\mu_{A_\lambda(X)}$.
Hence
\begin{align*}
\liminf_m \mu_{A_\lambda(Y^{(m)})}(\{0\}) &\leq\limsup_m\mu_{A_\lambda(Y^{(m)})}(\{0\})\\
&\leq\mu_{A_\lambda(X)}(\{0\}),
\end{align*}
where we used Portmanteau Theorem in the last inequality.
Hence the desired inequality follows as 
$\inf_{Y\in\X_\varepsilon}\mu_{A_\lambda(Y)}(\{0\})\leq\mu_{A_\lambda(Y^{(m)})}(\{0\})$
\end{proof}

The main spirit of Proposition \ref{bound from matrices} is that if we prove a bound from below for the size of an atom for $P(Y_1,.....,Y_d)$ when considering  \textbf{all} tuples of matrices in $Y_1,.....,Y_d  \in \mathcal{X}$, this bound will hold for $P(X_1,...,X_d)$. On the contrary if we have \textbf{one} example of $Y_1,.....,Y_d \in \mathcal{X}$ for which we can calculate the size of the atom, then this size is a bound from above for the size of the atom for $P(X_1,...,X_d)$. In the next section we will use this observation repeatedly.

\section{Applications to Polynomials}\label{sect:application}

The purpose of this section is to apply the our general results in the particular but important case of polynomials. In this and the following section all our random variables and polynomials will be selfadjoint\footnote{A polynomial $P\in\mathbb{C}\left<x_1,\dots,x_n\right>$ is selfadjoint if $P=P^*$. In our case this means that for any monomial $\lambda x_{i_1}x_{i_2}...x_{i_n}.$ appearing in $P$, the monomial $\overline{\lambda} x_{i_n}x_{i_{n-1}}...x_{i_1}$ also appears in $P$}  thus all analytic distributions will be probability measures on $\R$.  Also, we will exclude trivial cases and then $P(x_1,\ldots,x_n)$ will contain all the variables $x_1,\ldots,x_n$.

We use Theorem \ref{thm:atomic} to bound from above the size of atoms for some polynomials. The proofs of our estimates will use operators with atomic distributions with rational sizes. However, all the bounds in this section remain true for general unbounded operators by suitable approximation.

More explicitly, we will consider $X=(X_1,\dots,X_d)$, a tuple of $*$-free normal random variables, such that for $k=1,\dots,d$,
\[
\mu_{X_k}=\frac{1}{n}\sum_{i=1}^n\delta_{\lambda^{(k)}_i},
\]
where $\lambda_i^{(k)}\in\R$, $i=1,\dots,n$.

Then we will choose a tuple $Y=(Y_1,\dots,Y_d)$ of matrices in $M_n(\C)$ such that for each $k=1,\dots,d$, $X_k$ has the same distribution as $Y_k$.

Again, thanks to Theorem \ref{thm:comparison-rational functions}, for each selfadjoint  polynomial $P\in\Cx$ and $\lambda\in\mathbb{R}$
\begin{equation}\label{t}
\mu_{P(X)}(\{\lambda\})\leq\mu_{P(Y)}(\{\lambda\}).
\end{equation}
 
Interestingly, in many cases, we will be able to calculate these weights exactly by comparing the bounds from above with general bounds from below that we derive in the next Subsection \ref{subsect:unavoidable atoms}.

\subsection{Some unavoidable atoms} \label{subsect:unavoidable atoms}

Let us derive some lower bounds for the size of atoms for polynomials in any operators. The following propositions are basic lower bounds for the size of atoms. Notice that the proofs do not use any freeness assumptions for the involved operators. These are probably well known to the experts and easy to derive. We prove them for the convenience of the reader.

\begin{prop}\label{unatoms1}
Let $(\mathcal{M},\tau)$ be a tracial $W^*$-probability space and consider a tuple $X=(X_1,...,X_d)$ of selfadjoint operators in $\mathcal{M}$.
Suppose that for each $i$, $\mu_i(\lambda_i)\geq t_i$ for some $\lambda_i\in\R$ and $t_i\in[0,1]$.
If $t_1+\cdots+t_d> d-1$ then for any selfadjoint polynomial $P\in \Cx$, the measure $\mu_{P(X)}$ has an atom at $P(\lambda)$ of size at least $t_1+\cdots+t_d-(d-1)$, where $\lambda:=(\lambda_1,\dots,\lambda_d)$.
\end{prop}

\begin{proof}
For each $i\in\{1,\dots,d\}$, let $p_i$ be the orthogonal projection onto the eigenspace $\ker(\lambda_i-X_i)$.
That is, $X_i p_i=\lambda_ip_i$ for all $i\in\{1,\dots,d\}$ and $\tau(p_i)=t_i$.
Then for $p:=\min(p_1,\dots,p_d)$, we have $P(X)p=P(\lambda)p$.
Since $\tau(p)\geq \tau(p_1)+ \cdots+\tau(p_d)-(d-1)=t_1+ \cdots+t_d+d-1$, we get the conclusion.
\end{proof}

\begin{prop}\label{unatoms3}
Let $X,Y_1,...,Y_d$ be selfadjoint operators in some tracial $W^\ast$-probability space $(\mathcal{M},\tau)$.
Suppose that $\mu_X(\{0\})\geq t$.
    Then for any selfadjoint polynomial $P$ of the form $$P(x,y_1,...,y_d)=\sum^k_{i=1}Q_{i,1}(x,y_1,...,y_d)xQ_{i,2}(x,y_1,...,y_d),$$ the analytic distribution of $P(X,Y_1,...,Y_d)$ has an atom at $0$ whose size is at least\newline $\max(kt-(k-1),0)$.
\end{prop}

\begin{proof}
We bound the rank,
\begin{align*}
\rank(P(X,Y_1,\dots,Y_d))&\leq\sum^k_{i=1}\rank(Q_{i,1}(X,Y_1,...,Y_d)XQ_{i,2}(X,Y_1,...,Y_d)) \\
&\leq k\rank(X)\\&=k(1-t),
\end{align*}
from where we obtain
$$\mu_{P(X,Y_1,\dots,Y_d)}(\{0\})=1-\rank(P(X,Y_1,\dots,Y_d))\geq1-k(1-t)=kt-(k-1).$$
\end{proof}

Let us isolate two cases of polynomials of a special form which we will use later.

\begin{corollary} \label{unatoms4}

Let $X,Y_1,...Y_d$ be selfadjoint operators in some tracial $W^\ast$-probability space $(\mathcal{M},\tau)$.
\begin{enumerate}\item Suppose that $\mu_X(\{0\})\geq t>1/2$.
Then for any polynomial  $P$ of the form $P(x,y_1,...,y_d)=xQ(x,y_1,...,y_d)+ Q^*(x,y_1,...,y_d)x$ the analytic distribution of $P(X,Y_1,...,Y_d)$ has an atom at $0$ of size at least $2t-1$. 
\item For any polynomial $P$ of the form $P(x,y_1,...,y_n)=Q_1(x,y_1,...,y_n)xQ_2(x,y_1,...,y_n)$,  we have  $\mu_P(X,Y_1,...,Y_n)(\{0\})\geq\mu_X(\{0\})$  
\end{enumerate}
\end{corollary}

\begin{remark}
All the bounds above may also be derived, by proving the analog theorems for the specific case of matrices.

Indeed, if we prove such bound for free random variables $X,Y_1,\ldots,Y_d$, then the conclusion will also hold for operators in a finite von Neumann algebra, by our Theorem \ref{thm:comparison-rational functions}. 
Now, by Proposition \ref{bound from matrices} and Theorem \ref{BVrational}, we only need to prove this for $X,Y_1,...Y_d\in M_n(\mathbb{C})$.
\end{remark}

\subsection{Using diagonal matrices}

Let $X_1,X_2$ be two freely independent selfadjoint random variables with
\[
\mu_{X_1}=\frac{1}{n}\sum_{i=1}^n\delta_{\lambda_i}\quad\text{and}\quad\mu_{X_2}=\frac{1}{n}\sum_{i=1}^n\delta_{\mu_i}
\]
where $\lambda_i,\mu_i\in\R$, $i=1,\dots,n$.
Let $Y_1$ and $Y_2$ be diagonal matrices with $\lambda_1,\dots,\lambda_n$ and $\mu_1,\dots,\mu_n$ on their diagonals respectively.
Thanks to Theorem \ref{thm:comparison-rational functions}, for any $P\in\C\left<x_1,x_2\right>$ we have
\[
\sigma_p(P(X_1,X_2))\subset\sigma_p(P(Y_1,Y_2)=\{P(\lambda_i,\mu_i)\mid i=1,\dots,n\}
\]
as
\[
P(Y_1,Y_2)=\begin{pmatrix}P(\lambda_1,\mu_1) & 0 & \cdots & 0)\\0 & P(\lambda_2,\mu_2) & \cdots & 0\\\vdots & \vdots & & \vdots\\0 & 0  &\cdots & P(\lambda_n,\mu_n)\end{pmatrix}.
\]
Moreover, since we can reorder $\lambda_i$'s and $\mu_j$'s while $\mu_{X_1}$ and $\mu_{X_2}$ stay invariant, we see that
\[
\sigma_p(P(X_1,X_2))\subset\bigcap_{\sigma\in S_n}\{P(\lambda_i,\mu_{\sigma(i)})\mid i=1,\dots,n\}.
\]

Let us define $$\Gamma_\sigma(\lambda):=\{i\in\{1,...n\}\mid p(\lambda_i,\mu_{\sigma(i)})=\lambda\},\quad \forall\lambda\in\mathbb{C}, \forall\sigma\in S_n$$
and $$k(\lambda):=\min_{\sigma\in S_n}|\Gamma_\sigma(\lambda)|.$$
Then we have $\mu_{P(X_1,X_2)}(\{\lambda\})\leq \frac{1}{n}k(\lambda).$
This can be done similarly for any tuple of operators $X_1,\dots,X_d$. 

\begin{theorem}  \label{diagonalbound1}
Let $X_1\dots,X_d$ be freely independent selfadjoint random variables and 
let $P\in\mathbb{C}\left< x_1,\dots,x_d \right>$ be a selfadjoint polynomial.
The possible atoms of $P(X_1,...,X_d)$ are contained in the set $$\{P(\rho_1,...,\rho_d)\mid\rho_i\text{ is atom of }X_i, ~\forall i=1,\dots, d\}.$$
Furthermore, if $X_1,\dots,X_d\in M_n(\C)$ and $\lambda\in\R$, then  $\mu_{P(X_1,...,X_d)}(\{\lambda\})$ is smaller than $k(\lambda)/n$, where
$$k(\lambda):=\min_{\sigma_1,...,\sigma_d\in S_n} |\{j\in\{1,...n\}:P(\lambda^{(1)}_{\sigma_1(j)},...,\lambda^{(d)}_{\sigma_d(j)})=\lambda \}|,$$
where, for $i=1,\ldots,d$, $\{\lambda^{(i)}_{j}\}^n_{j=1}$ denotes the eigenvalues of $X_i$.
\end{theorem}

One natural question is whether $k(\lambda)$ gives the precise size of the atom at $\lambda$. Certainly if $k(\lambda)=0$ this is the case. A positive result is well-known for the special case  $P(x_1,x_2)=x_1+x_2$.
Actually, Theorem 7.4 in \cite{BV98} implies that $\alpha=\lambda+\mu$ is an atom of $X_1+X_2$ if and only if the sum of the multiplicity of $\lambda$ and the multiplicity of $\mu$ is larger than $n$. Such a condition is sufficient for arbitrary polynomials $P$ in $\C\left<x_1,x_2\right>$. That is, $\alpha=P(\lambda,\mu)$ is an atom of $P(X_1,X_2)$ if the sum of the multiplicity of $\lambda$ and the multiplicity of $\mu$ is larger than $n$. Proposition \ref{unatoms1} actually corresponds to this situation for any number of variables.

However, this description is not sufficient to give all atoms of $P(X_1,X_2)$. For example, if
$P(x_1,x_2) = x_1x_2-x_2x_1$, then $\{\lambda\in \mathbb{C}\mid k(\lambda)>0\}=\{0\}$ and  $k(0)=n$.  But in Section 7 we will see in general  for free random variables that $X_1X_2- X_2X_1$ may not have atoms even when $X_1$ and $X_2$ are purely atomic. Nevertheless, the Theorem \ref{diagonalbound2} shows that this condition could still be necessary under the assumption of certain injectivity properties of the polynomial.

Before proving this result, let us give some examples. First, Theorem~\ref{diagonalbound2} allows to recover the well-known results about atoms of the free additive \cite{BV98} and to generalize to the result in \cite{Bel03} for multiplicative convolutions: indeed, we readily see that $P(X,Y)=X+Y$ satisfies the above condition for all $a$, and so does $P(X,Y)=XYX$ for $X>0$ and $a\neq0$.   Another new example is $P(X,Y)=X^2YX+XYX^2$ for $a\neq0$, and it is easy to exhibit other examples as Example \ref{examples} shows.

Let us state this
\begin{corollary}
Let $\mu,\nu\in\mathcal{P}$ be probability measures on $\mathbb{R}$.
\begin{enumerate}
    \item $\mu\boxplus\nu$ has an atom at $a\in\mathbb{R}$ if and only if there exist $\lambda,\rho\in \mathbb{R}$ such that $\rho+\lambda=a$ and $\mu(\lambda)+\nu(\rho)>1$. Moreover, if $\mu(\lambda)+\nu(\rho)>1$, we have $\mu\boxplus\nu({a})=\mu(\lambda)+\nu(\rho)-1$.

\item If $\mu\in \mathcal{P}(\mathbb{R}^+)$ then $\mu\boxtimes\nu$ has an atom at $a\in\mathbb{R}\setminus\{0\}$ if and only if there exist $\lambda,\rho\in \mathbb{R}$ such that $\rho\lambda=a$ and  $\mu(\lambda)+\nu(\rho)>1$. Moreover, if $\mu(\lambda)+\nu(\rho)>1$, we have $\mu\boxtimes\nu({a})=\mu(\lambda)+\nu(\rho)-1.$ The atom at $0$ is given by $\mu\boxtimes\nu\{0\}=\max(\mu\{0\},\nu\{0\})$.
\item If $P(X,Y)=X^2YX+XYX^2$, then $P^{\Box}(\mu,\nu)$ has an atom at $a\in\mathbb{R}\setminus\{0\}$ if and only if there exist $\lambda,\rho\in \mathbb{R}$ such that $P(\rho,\lambda)=a$ and  $\mu(\lambda)+\nu(\rho)>1$. Moreover, if $\mu(\lambda)+\nu(\rho)>1$, we have $P^{\Box}(\mu,\nu)({a})=\mu(\lambda)+\nu(\rho)-1$.

\end{enumerate}
\end{corollary}
\begin{proof} 

Given Theorem \ref{diagonalbound2}, the only thing to analyze is the atom at $0$ in part (2).
Corollary \ref{unatoms4} already shows that $\mu\boxtimes\nu\{0\}\geq\max(\mu\{0\},\nu\{0\})$. To show that $\mu\boxtimes\nu\{0\}\leq\max(\mu\{0\},\nu\{0\})$, we may assume that $\mu\sim X_n^2$ and $\nu\sim Y_n$. 

Consider the specific realization of $X_n$ and $Y_n$,
\begin{equation*}
X_n=\left(
    \begin{matrix}
      \lambda_1&      &       &       &       &               \\
      &     \ddots  &     &       &       &               \\
       &      &  \lambda_m     &     &       &                \\
       &       & & 0 & &               \\
       &       &       &     &  \ddots &               \\
      &       &       &       &      &   0        \\
    \end{matrix}
    \right)
    \qquad
\text{and}
\qquad
Y_n=\left(
    \begin{matrix}
      \rho_1&      &       &       &       &               \\
      &     \ddots  &     &       &       &               \\
       &      &  \rho_l  &     &       &                \\
       &       & & 0& &               \\
       &       &       &     &  \ddots &               \\
      &       &       &       &      &   0      \\
    \end{matrix}
    \right)
  \end{equation*}
with $\lambda_i \neq 0$  for $0\leq i\leq m$ and $\rho_j\neq0$, for $0\leq j\leq l$.  Then, if $r=\min(l,m)$, we have
$$X_nY_nX_n=\left(
    \begin{matrix}
      \lambda_1\rho_1 \lambda_1&      &       &       &       &               \\
      &     \ddots  &     &       &       &               \\
       &      &  \lambda_{r}\rho_r \lambda_{r} &     &       &                \\
       &       & & 0& &               \\
       &       &       &     &  \ddots &               \\
      &       &       &       &      &   0      \\
    \end{matrix}
    \right).$$
The result follows since $\Null(X_nY_nX_n)=max\{\Null(X_n),\Null(X_n)\}$.
\end{proof}

Now we prove Theorem~\ref{diagonalbound2}, which is one of the main results of this section.

\begin{proof}[Proof of Theorem~\ref{diagonalbound2}]

The atom is at least $\max(t+s-1, 0)$ at $a$, by Proposition \ref{unatoms1}. We want to prove that the atom is at most $\max(t+s-1, 0)$ at $a$. 

\textbf{Case 1. $s+t-1>0$.} Consider the diagonal matrices, 
\begin{equation*}
X_n=\left(
    \begin{matrix}
      \lambda&      &       &       &       &               \\
      &     \ddots  &     &       &       &               \\
       &      &  \lambda     &     &       &                \\
       &       & & \lambda_{i+1} & &               \\
       &       &       &     &  \ddots &               \\
      &       &       &       &      &   \lambda_n        \\
    \end{matrix}
    \right)
\qquad\text{and}\qquad
Y_n=\left(
    \begin{matrix}
     \rho_1&      &       &       &       &               \\
      &     \ddots  &     &       &       &               \\
       &      &  \rho_j  &     &       &                \\
       &       & & \rho& &               \\
       &       &       &     &  \ddots &               \\
      &       &       &       &      &   \rho       \\
    \end{matrix}
    \right)
  \end{equation*}
where $i=sn\geq j=n-tn$.

Then $$P(X_n,Y_n)=\left(
    \begin{matrix}
      A_1&    0  &  0                 \\
     0 &    A_2&   0                \\
      0 &  0    &  A_3  
    \end{matrix}
    \right)$$
where
 \begin{equation*}
A_1=\left(
    \begin{matrix}
      P(\lambda,\rho_1)&      &                   \\
      &     \ddots  &                    \\
       &      &  P(\lambda ,\rho_i)  
    \end{matrix}
    \right),
    \qquad
A_2=\left(
    \begin{matrix}
     P(\lambda,\rho)&      &                    \\
      &     \ddots  &          \\
       &      &  P(\lambda,\rho)                    \end{matrix}
    \right),  \end{equation*}
and
\begin{equation*}A_3=\left(
    \begin{matrix}
       P(\lambda_{j+1},\rho)& &               \\
       &  \ddots &               \\
       &      &   P(\lambda_n,\rho)       
    \end{matrix}
    \right) . 
  \end{equation*}
By hypothesis, $P(\lambda,\rho_{l})\neq a$ for $l\leq i$ and  $P(\lambda_{l},\rho)\neq a$ for $l>j$. We see that the size of the eigenspace associated to $a$ is $\dim(A_2)=i-j=(s+t-1)n$.

\textbf{Case 2. $s+t-1\leq 0$.} We claim that we can reorder the eigenvalues in such a way that $P(\lambda_i,\rho_i)\neq a$. Then taking the matrices $X_n=\diag(\lambda_1,\dots,\lambda_n)$ and $Y_n=(\rho_1,\dots,\rho_n)$, we see that $P(X_n,Y_n)=\diag(P(\lambda_1,\rho_1),\dots,P(\lambda_n,\rho_n)$ which has no eigenvalue equal to $a$, as desired. 

\textbf{Claim.} There is an ordering of the eigenvalues of $X_n$ and $Y_n$ such that $P(\lambda_i,\rho_i)\neq a$ for all $i=1,\cdots,n$.

We prove this by induction on $n$. For $n=1$ and $n=2$, it is clear.
Consider a reordering of $\{\lambda_i,\rho_i\}$ such $P(\lambda_i,\rho_i)\neq a$ for $i\leq n-1$ which is possible  by induction. Now consider $\lambda_n,\rho_n$. If $P(\lambda_n,\rho_n)\neq a$ we are done. Otherwise, if $P(\lambda_n,\rho_n)=a$, then let 
$$S=\{j\in[n-1]\mid \text{such that $\lambda_j=\lambda_n$ or  $\rho_j=\rho_n$}\}.$$

If $|S|=n-1$, by choosing for each $j$, $\rho_j$ or $\lambda_j$,  together with $\lambda_n,\rho_n$ we have that $sn+tn\geq n+1$, which yields a contradiction. Finally, if $|S|\leq n-2$, there exists $j$ such that $\lambda_j\neq\lambda_n$ and $,\rho_j\neq\rho_n$. By the injectivity of $P$ for $a$,  $P(\lambda_j,\rho_n)\neq a$ and $P(\lambda_n,\rho_j)\neq a$, from where we get the desired reordering.
\end{proof}

\begin{remark}Note that the conditions in the previous theorem only depend on the evaluation of the polynomial in the atoms of the marginals and thus,  under the above hypothesis, two injective non-commutative polynomials in two variables with the same values when evaluated in the atomic parts will have the same atoms when evaluated in free random variables. For instance, $X+Y$ and $X+Y+i(XY-YX)$ have exactly the same atoms, for any $X$ and $Y$ which are free.
\end{remark}

One of the implications of Theorem \ref{diagonalbound2} is available for more variables, namely the analog of the bound from Proposition \ref{unatoms1} is also achieved.
For that purpose, first we generalize that condition on polynomials, as follows.
\begin{definition}
Let $P\in\Cx$ be given.
For $a\in\mathbb{R}$, if for all $\rho_1,\ldots, \rho_i,\ldots,\rho_d\in\mathbb{C}$, such that $P(\rho_1,\ldots, \rho_i,\ldots,\rho_d)=a$ we have that $$P(\rho_1,\ldots, \rho_i,\ldots,\rho_d)\neq P(\rho_1,\ldots,\tilde\rho_i, \ldots,\rho_d)$$
 whenever $\rho_i\neq\tilde\rho_i$ for any given scalar values $\rho_1,\dots,\rho_{i-1},\rho_{i+1},\dots,\rho_d$, then we say $P$ is \emph{injective} for $a$. If $P$ is injective for all $a\in\mathbb{R}$, then we simply say that $P$ is injective.
\end{definition}

\begin{eg}\label{examples}
For random variables supported on $\mathbb{R}^+$, a polynomial $P$ with positive coefficients such that any monomial has odd order in any variable is injective. Some examples in three variables are:
\begin{enumerate}
\item $X+Y+Z$;
\item $YXZ+ZXY$;
\item $X^3+XZX^4+X^4ZX+Y. $
\end{enumerate}
\end{eg}

\begin{theorem} \label{thm:diagonal} 
Let $P\in\Cx$ be injective for $a\in\R$ and let $a=P(\lambda_1,\ldots,\lambda_d)$.
Suppose that $X_1,\dots,X_d$ are free selfadjoint random variables and $X_i\sim\mu_i$  with $\mu_i(\{\lambda_i\})=t_i$.
If $t_1+\ldots+t_d\geq d-1$ then the distribution of $P(X_1,\dots,X_d)$ has an atom at $a$ of size \emph{exactly} $t_1+\ldots+t_d-(d-1)$.

\end{theorem}

\begin{proof}
Again, the lower bound is obtained by Proposition \ref{unatoms1}.
We proceed by induction on the number of variables to get the upper bound.
We assume that we can make a construction as in the proof of Theorem \ref{diagonalbound2} for $d$ variables and show that we can make the same construction for $d+1$. Suppose that $X_1,\dots,X_{d+1}$ are free selfadjoint random variables, $X_i\sim\mu_i$  with $\mu_i(\{\lambda_i\})=t_i$ and  $t_1+\ldots+t_{d+1}\geq d$. First note that, since $t_i\leq1$, the condition $t_1+\ldots+t_{d+1}\geq d$ implies that $t_1+\ldots+t_{d}\geq(d-1)$.

Using the induction hypothesis, we assume that for the first $d$ variables we have diagonal matrices $Y_1,\dots,Y_d$ such that:
\begin{enumerate}
\item There are $n(d-(t_1+\ldots+t_d))=l_d$ indices $m$ such that the following holds: There exists $i\in \{1,\dots,d\}$ such that $(Y_{j})_{{mm}}=\lambda_j$ for $j\neq i$ and $(Y_{i})_{{mm}}\neq\lambda_i$. 
\item The remaining $n-l_d$ indices, $m$, satisfy $(Y_j)_{mm}=\lambda_j$ for all $j$.
\end{enumerate}
The base case $d=2$ is done in Theorem \ref{diagonalbound2}.  Without  loss  of  generality,  by  a  simple  reordering,  we  may  assume  that  case  (1)  is achieved in the indices $1,....,l_d$.

We claim that if (1) and (2) are satisfied then the eigenspace of $a$ has dimension 
$$ n-l_d=n-n(d-t_1+\ldots+t_{d})=n(t_1+\ldots+t_{d}-d-1)$$ as desired. Indeed, for $m<l_{d}$, since  $(Y_j)_{{mm}}=\lambda_j$ for $j\neq i$ and $Y_{i}\neq{\lambda_{d}}$, then $P((Y_1)_{mm},\ldots,(Y_{d+1})_{mm})\neq a$. Conversely, for $m\geq l_d$ we have $(Y_j)_{{mm}}=\lambda_j$ and thus  $P((Y_1)_{mm},\ldots,(Y_{d+1})_{mm})=a$ in this range.

To prove (1) and (2) for $d+1$, we consider the diagonal matrix $(Y_{d+1})_{mm}=\lambda_{d+1}$ for $m\leq n(t_{d+1})$ and $(Y_{d+1})_{mm}\neq\lambda_{d+1}$ for $m\leq n(t_d)$.

Firstly, since $t_1+\ldots+t_{d+1}>d$ then $n(t_{d+1})>n(d-t_1+\ldots +t_d)=l_d$. Thus, the indices $m\in\{1,...,l_d\}$ satisfy the property in (1) since by induction  $(Y_{j})_{mm}=\lambda_{j}$ for $j\neq i$ and $1\leq j \leq d$, $(Y_{i})_{mm}=\lambda_{j}$ and we have $(Y_{d+1})_{mm}=\lambda_{d+1}$.

Secondly, for $m>n(t_{d+1})$, the indices $m$, also satisfy property (2),  since $m>l_d$, by induction that  $(Y_j)_{{mm}}=\lambda_j$ for $j\leq d$ and $Y_{d+1}\neq{\lambda_{d+1}}$.

Finally, the indices in  $m\in\{l_{d},...,n(t_{d+1})\}$ satisfy the property in (2), since again by induction  $(Y_{j})_{mm}=\lambda_{j}$ for all $1\leq j \leq d$,  and in this interval we have $(Y_{d+1})_{mm}=\lambda_{d+1}$.  This proves  that the number of indices satisfying (2) is $n(t_{d+1})-l_d=n-l_{d+1}$, which finishes the induction.

\end{proof}

On the other hand, similar as for the commutator, for some polynomials we may exclude atoms outside of  $0$ (or any $\lambda\in\mathbb{R}$, by a simple shifting).  A selfadjoint polynomial $P=P(x,y)$ in the variables $x$ and $y$, is called \emph{mixed} if it consists only of monomials containing both $x$ \emph{and} $y$ (this is equivalent to $P(0,y)=0$ and $P(x,0)=0$). 

\begin{prop}\label{diagonalatoms1}
Let $X$ and $Y$ be free random variables and suppose that $X\sim\mu$  and $Y\sim\nu$ with $\mu=t\delta_0+(1-t)\tilde\mu$ and  $\nu=s\delta_0+(1-s)\tilde\nu$. If $t+s\geq1$ then for all mixed polynomials $P$,  the distribution of $P(X,Y)$ has no atoms outside of $0$.
\end{prop}

\begin{proof}
After suitable approximation, we may assume $$\mu=\frac{1}{n}\sum^m_{i=1} \delta_{\lambda_i}+\frac{n-m}{n}\delta_0\qquad\text{and}\qquad
\nu=\frac{m}{n}\delta_0+\frac{1}{n}\sum^{n-m}_{i=1} \delta_{\mu_i}$$
for some $\lambda_i,\rho_j\in\mathbb{R}$ .

By Proposition \ref{bound from matrices} it is enough to find matrices $X_n,Y_n$ with distributions $\mu$ and $\nu$ such that $P(X_n,Y_n)=0$. Consider the diagonal matrices, 

\begin{equation*}
X_n=\left(
    \begin{matrix}
      \lambda_1&      &       &       &       &               \\
      &     \ddots  &     &       &       &               \\
       &      &  \lambda_m     &     &       &                \\
       &       & & 0 & &               \\
       &       &       &     &  \ddots &               \\
      &       &       &       &      &   0        \\
    \end{matrix}
    \right)
    \qquad
\text{and}
\qquad
Y_n=\left(
    \begin{matrix}
     0&      &       &       &       &               \\
      &     \ddots  &     &       &       &               \\
       &      &  0  &     &       &                \\
       &       & & \mu_1& &               \\
       &       &       &     &  \ddots &               \\
      &       &       &       &      &   \mu_{n-m}       \\
    \end{matrix}
    \right)
  \end{equation*}
Notice that $X_nY_n=X_nY_n=0$ and thus for any mixed monomial $m$ we have that  $m(X_n,Y_n)=0$. By linearity, we have for any mixed polynomial that $P(X_n,Y_n)=0$.
\end{proof}

\subsection{The algebra of two non-commuting projections}

The idea of this section is to improve the bounds from the previous one by considering a more general situation where we allow blocks of size $2\times2$ on the diagonal.

So let  $X_n,Y_n\in M_{2n}(\C)$ be matrices consisting of diagonal blocks of size $2\times2$, $A_1,\dots,A_n$ and $B_1,\dots,B_n$ on their diagonals, respectively. Similar as before
\[
p(X_n,Y_n)=\begin{pmatrix}p(A_1,B_1) & 0 & \cdots & 0\\0 & p(A_2,B_2) & \cdots & 0\\\vdots & \vdots & & \vdots\\0 & 0  &\cdots & p(A_n,B_n)\end{pmatrix}.
\]

We are interested in the eigenvalues of $p(A_i,B_i)$ with $A_i,B_i\in M_{2}(\C)$. We may now write $A_i$ and $B_i$ in terms of the one rank one projections $$R=\left(
    \begin{matrix}
    1& 0   \\
      0&   0  
    \end{matrix}    \right)
  \qquad\text{and}\qquad
  T= \left(
       \begin{matrix}
    t& \sqrt{(1-t)t}   \\
      \sqrt{(1-t)t}&    1-t 
    \end{matrix}    \right).$$
That is, we may assume that
 $$A_i=\begin{pmatrix}\lambda_1 & 0 \\
0&\lambda_2\end{pmatrix}=(\lambda_1-\lambda_2)\begin{pmatrix}1 & 0 \\
0&0\end{pmatrix}+\lambda_2\begin{pmatrix}1 & 0 \\
0&1\end{pmatrix}$$
and  
\begin{align*}
B_i&=
       \begin{pmatrix}
    t\mu_1+\mu_2(1-t)& (\mu_1-\mu_2) \sqrt{(1-t)t}  \\
   (\mu_1-\mu_2)   \sqrt{(1-t)t}&    \mu_1(1-t)+t\mu_2
    \end{pmatrix}\\[1em]
    &=
(\mu_1-\mu_2)\left(
       \begin{matrix}
    t& \sqrt{(1-t)t}   \\
      \sqrt{(1-t)t}&    1-t 
    \end{matrix}    \right) 
+\mu_2\begin{pmatrix}1 & 0 \\
0&1\end{pmatrix}
\end{align*}

In this way, $A=(\lambda_1-\lambda_2)R+\lambda_2I$ and  $B=(\mu_1-\mu_2)T+\mu_2I$. So $P(A_i,B_i)=Q_i(R,T)$ for some polynomial $Q_i$.

In this way the invertibility of $P(A_i,B_i)$ is translated into the invertibility of $Q_i(R,T)$. The set of polynomials in $R$ and $T$ is well understood and rather easy  to handle.  Moreover, when considering $T$ as a function of $t$, for all $t\in(0,1)$, the algebra generated by $R$ and $T$ is known to be isomorphic to the $C^*$-algebra generated by two non-commuting projections, see \cite{RS89}.

To understand any polynomial in $R,T$ it is enough to consider certain quadratic polynomials. Indeed, since we have the relations $RTR=tR$, $TRT=tT$, $R^2=R$ and $T^2=T$, any polynomial in $R$ and $T$ may be written as a linear combination 
$$aR+bT+cRT+dTR+eI,$$
which actually has the matrix form 
$$\begin{pmatrix} a + e + (b + c  + d )t & (b+c)\sqrt{(1 - t) t} + c \\ (b+d)\sqrt{(1 - t) t} & e + b (1 - t)\end{pmatrix}.$$
The determinant of this matrix is given by $ a (b (-t) + b + e) + b e + c d t^2 - c d t + c e t + d e t + e^2$ which for $e=0$ (that we may always assume), reduces to $(t - 1) (c d t - a b)$.

We aim to turn this calculation into a criterion to avoid situations as in Corollary  \ref{unatoms3}, or even to characterize all the sizes of the atoms of some polynomials. To do this we introduce some notation.

For $\lambda,\rho\in\mathbb{C}$, let us define define $\Subs_{\lambda,\rho}:{C}\left<x,y\right>\to\mathbb{C}[t]$ which sends $P\mapsto \Subs_{\lambda,\rho}[P]$ as follows: For a monomial $m=m(x,y)$, $\Subs_{\lambda,\rho}
[m]\in\mathbb{C}[t]$, is given by the formula
$$\Subs_{\lambda,\rho} [m] (t)=m(\lambda,\rho)t^{k-1},$$ whenever $m(x,y)$ is of any of the following forms: $m(x,y)=x^{i_1}y^{j_1}\cdots x^{i_{k}}, m(x,y)=y^{i_1}x^{j_1}\cdots y^{i_{k}}, m(x,y)=x^{i_1}y^{j_1}\cdots y^{i_k}, m(x,y)=y^{i_1}x^{j_1}\cdots x^{i_{k}}.$

We extend $Subs_{\lambda,\rho}$ by linearity to any polynomial $P\in\mathbb{C}\left<x,y\right>$ with no linear term.

\begin{definition}Let $P\in\mathbb{C}\left<x,y\right>$, be a polynomial with no constant term and write $P=P_1+P_2+P_3+P_4$.
where
$$P_1=xQ_1y, \quad P_2=yQ_2x, \quad P_3=xQ_3x+ax\text{ and}\quad P_4=yQ_4y+by$$ for some
polynomials $Q_i\in\mathbb{C}\left<x,y\right>$ and $a,b\in\mathbb{C}$.
\begin{enumerate} 

\item  We say that $P$ satisfies the determinant condition on $(\lambda,\mu)$ if  \emph{for some} $ t\in[0,1],$
\begin{equation}\label{eq:det-condition}
t{\Subs}_{\lambda,\rho}[P_1](t) \Subs_{\lambda,\rho} [P_2](t) \neq  \Subs_{\lambda,\rho} [P_3](t) \Subs_{\lambda,\rho} [P_4](t).
\end{equation}

\item We say that $P$ satisfies the \emph{strong determinant condition} if it satisfies the determinant condition for all pairs $(\lambda,\rho)$ with  $\lambda,\rho\neq0$.
\end{enumerate} 

\end{definition}

\begin{remark}
While condition \eqref{eq:det-condition} seems to be hard to check for specific $\lambda$ and $\mu$, there are actually many polynomials such that the strong determinant condition is satisfied and easy to check. Here we list some useful conditions which imply the strong determinant condition.

1) $P_1=0$ or $P_2=0$ and $P_3,P_4\neq0$. 

2) $P_3=0$ or $P_4=0$ and $P_1,P_2\neq0$.

3) $deg(P_1P_2)\neq deg(P_3P_4)+1.$

4) $P_3$ and $P_4$ have constant term.

\end{remark}

\begin{prop}\label{criteria 1}
Let $X$ and $Y$ be free random variables and let $t\geq s$ be such that $t+s\geq1$. Suppose that $X\sim\mu$  and $Y\sim\nu$ with $\mu=t\delta_0+(1-t)\tilde\mu$ and  $\nu=s\delta_0+(1-s)\tilde\nu$, with $\tilde\mu=\sum_{i=1}^{l} \frac{1}{n}\delta_{\lambda_i}$ and $\tilde\nu=\frac{1}{n}\sum_{i=1}^{m} \delta_{\rho_i}$ with $\lambda_i\neq0$ and $\rho_i\neq0$. If there exist an injective function $\sigma:[l]\to [m]$ such that  $\lambda_i,\rho_{\sigma(i)}$ satisfies the determinant condition \eqref{eq:det-condition} for $P(X,Y)$ then we have the following.

\begin{enumerate}
\item The atom of $P(X,Y)$ at $0$ has size at most $2t-1$.
\item If for all $\rho$, $P(0,\rho)\neq0$  then the atom at $0$ has exactly size $s+t-1$. 
\item If P(x,y) is \text{mixed}, then $P(X,Y)$ has no more atoms.
\end{enumerate}
\end{prop}

\begin{eg}\begin{enumerate}
 
 \item Let $X_1,X_2$ have distributions of the form $1/2\delta_0+1/2 \mu_1$ and $1/2\delta_0+ 1/2\mu_2$  with $\mu_i$ having no atoms at $0$. Then any polynomial of the form $P(X,Y)=XQ_1Y+YQ_2X$ has no atoms. In particular, this is the case for $XY+XY$ and $i(XY-YX)$.  
\item Consider the polynomial $P(x,y)=yxyxy+xyx+xyxy+xyxy$, which is mixed and does not satisfy the determinant condition. One sees that $$P(x,y)=(y+1)(xyx)(y+1)$$ and thus $P(X,Y)$ has an atom at 0 if $Y$ has an atom at 0 too (this is just free multiplicative convolution). Thus the determinant condition is not superfluous. Simpler examples like $Q(x,y)=xyx$, $R(x,y)=xyx+xyxyx$ also work.  \end{enumerate}
\end{eg}

\begin{proof}[Proof of Proposition~\ref{criteria 1}]

 Without loss of generality we may assume that $\sigma$ in the hypothesis is the inclusion, i.e. $\sigma(i)=i$.

Now, consider the block-diagonal matrices in $M_{2n}$ given by 
\begin{equation*}
X_n=\left(
    \begin{matrix}
     A_1&      &     &&  &           \\
      &     \ddots  & && &              \\
       &      &   A_{2l}   &&&\\
       &&&0&&\\
     &&&&\ddots&\\ 
     &&&&&0\\     
    \end{matrix}    \right)
  \qquad\text{and}\qquad  Y_n=\left(
   \begin{matrix}
     B_1&      &     &&  &              \\
      &     \ddots   & && &                \\
       &      &   B_{2m}        &&&\\         &&&0&&\\
     &&&&\ddots&\\ 
     &&&&&0\\     
    \end{matrix}    \right)
\end{equation*}
where $A_i=\lambda_iR$ for $i=1,\dots,2l$, $B_i=\rho_i T$ for $i=1,\dots,2m$, where
$$R=\left(
    \begin{matrix}
    1& 0   \\
      0&   0  
    \end{matrix}    \right)
  \qquad\text{and}\qquad T= \left(
       \begin{matrix}
    t& \sqrt{(1-t)t}   \\
      \sqrt{(1-t)t}&    1-t 
    \end{matrix}    \right) $$

We see that 
\begin{equation*}
P(X_n,Y_n)=\left(
    \begin{matrix}
     P(A_1,B_1)&      &       &&&  &&& \\
      &     \ddots  &   &&&         &&&    \\
       &      &   P(A_{2l},B_{2l})    &&& &&&\\
        &&&  P(0,B_{2l+1})&      &     &&& \\
       &&&   &     \ddots  &           &&& \\
         &&&  &      &   P(0,B_{2m})   &&& \\
                &&& &&&  0      &      \\
    &&&   &&&   &     \ddots  &            \\
  &&&       &&&  &      & 0  
    \end{matrix}  
    \right)
\end{equation*}
which means that if $P(A_i,B_i)$ is invertible, for all $1\leq i\leq 2l$, then  $P(X_n,Y_n)$ has nullity at most $2n-4l=2n(1-2(1-t))=2n(2t-1)$. 

Let us fix $1\leq i\leq 2l$ and notice that $A_i$ and $B_i$ satisfy the relations 
$$A_iB_iA_i=\lambda^2_i \rho_i tA_i~~, ~~B_iA_iB_i=\lambda_i \rho_i^2 tB_i~~,~~A_i^k=\lambda_i^kA_i~~ and ~~B_i^k=\rho_i^kB_i.$$ 
Using these relations iteratively we arrive at the following:
\begin{eqnarray*}P(A_i,B_i)&=&\Subs[P_1(t,\lambda_i,\rho_i)] RT+ \Subs[P_2(t,\lambda_i,\rho_i)] TR\\ 
&+& \Subs[P_3(t,\lambda_i,\rho_i)]R +\Subs[P_4(t,\lambda_i,\rho_i)T.
\end{eqnarray*}
Finally, for any $a_1,a_2,a_3,a_4\in\mathbb{C}$, the determinant of  $a_1RT+ a_2 TR + a_3R +a_4T$ is given by $(t-1)(ta_1a_2-a_3a_4)$. Thus $\det [P(A_i,B_i)]\neq0 $ whenever the determinant condition holds for $\lambda_i$ and $\rho_i$.

For part (2), if $P(0,\rho)\neq0$ then $P(0,B_i)$ for $2l<i\leq 2m$ has at least one non-zero eigenvalue and then the nullity decreases in at least $2(m-l)$. Thus the nullity is at most 
$$2n-4l-2(m-l)=2n-2m+2l=2n(1-(1-s)-(1-t))=2n(s+t-1).$$  
However, by Proposition \ref{unatoms1} the size of the nullity is at least $2n(s+t-1)$.

Finally, part (3) follows by Proposition \ref{diagonalatoms1}.
\end{proof}

\subsection{A rational function}

To end this section let us consider two different rational expressions in $x$ and $y$,
$R_1(x,y)=xy^{-2}x+yx^{-2}y$ and $R_2(x,y)=xy^{-2}x+x^{-1}y^2x^{-1}$.

We consider $X$ and $Y$  free random variables which are strictly positive. i.e. with support on $(0,\infty)$. We are interested in the atoms of  $R_1(X,Y)$ and $R_2(X,Y)$.
In such a case, $R_1(X,Y)$ and $R_2(X,Y)$ are well-defined due to Theorem \ref{thm:comparison-rational functions} (for example, we can compare $X,Y$ to their classically independent counterparts, which are clearly well-defined).

By considering the case where $X$ and $Y$ are commuting we see that similar as for polynomials, the atomic parts of $R_1(X,Y)$ and $R_2(X,Y)$ are both contained in the set \begin{equation} \label{spectra example rational}
\{\rho^2/\lambda^2+\lambda^2/\rho^2\mid \rho \text{ is an atom of }X,~ \lambda\text{ is an atom of } Y\}.\end{equation}

Moreover, for given $\rho$ and $\lambda$ the functions $f,g:\mathbb{R}^+\to\mathbb{R}$, defined by $f(x)=\rho^2/x^2+x^2/\rho^2$, and $g(x)=x^2/\lambda^2+\lambda^2/x^2$, are both injective, Thus, arguing in the same way as in the proof of Proposition  \ref{diagonalbound2} one can prove that  $R_i(X,Y)$ has an atom at $a$ if and only if there exist $\lambda$ and $\rho$ such that $\mu_X(\{\lambda\})+\mu_Y(\{\rho\})>1$ and $a=\rho^2/\lambda^2+\lambda^2/\rho^2$ and the mass at $a$ is given by $\mu_X(\{\lambda\})+\mu_X(\{\rho\})-1$. In particular, their atoms are inside $[2,\infty]$.

We see then that $R_1(X,Y)$ and $R_2(X,Y)$ have exactly the same atoms with the same weights.  However, the full distributions $\mu_{R_1(X,Y)}$ and $\mu_{R_2(X,Y)}$ may differ, as one can see by calculating their second moments.

\section{Examples}
\label{sect:commutator}
In this final section we show how to use the above results in two specific but important examples, the selfajoint polynomials $P_1(x,y)=i[x,y]=i(xy-yx)$ and $P_2(x,y)=\{x,y\}=xy+yx$. For $X$ and $Y$ free selfadjoint variables  we completely characterize the atoms of $P_1(X,Y)$ and $P_2(X,Y)$ in terms of the atoms of $X$ and $Y$.

\subsection{Commutator}

The following theorem describes exactly the size of the atoms of the commutator, $P_1(x,y)=i[x,y]=i(xy-yx)$ of free random variables. 

\begin{theorem}Let $X$ and $Y$ be free random variables. Let $t$ and s be the size of the largest atom of $X$ and $Y$, respectively, i.e., $$t=\max\{\mu_X(\{a\})\mid a\in\mathbb{R}\}\qquad\text{and\qquad}s=\max\{\mu_Y(\{b\})\mid b\in\mathbb{R}\}.$$
Then 
\begin{enumerate}\item  $i(XY-YX)$ has no atoms outside of $0$.
\item $i(XY-YX)$ has an atom at $0$ of size given by $max(2t-1,2s-1,0)$.
\end{enumerate}
\end{theorem}
\begin{proof}

Part 1) is easy. By taking $X_n$ and $Y_n$ diagonal matrices, we see that, since $X_nY_n-Y_nX_n=0$, the only possible atom of $i[X,Y]$ is at 0. 

We will show that the size of this atom is given by $r=\max(2t-1,2s-1,0)$, where $t$ and $s$ are the sizes of the largest atoms of $X$ and $Y$, respectively.

Now, already Corollary \ref{unatoms4} tells us that  $\mu(\{0\})\geq r$.  Hence we only need to show that $\mu(\{0\})\leq \max(2t-1,2s-1,0).$ For this we will use Proposition \ref{bound from matrices} on block diagonal matrices.

Notice that, for any $\lambda,\rho\in\mathbb{R}$, we have that $i[X,Y]=i[X+\lambda,Y+\rho]$, so we may assume without loss of generality that the largest of the atoms of $X$ and $Y$ are at $0$. Also, since $X$ and $Y$ play the same role we may assume that $s\leq t$. Thus we want to show that the atom at zero is smaller than $2t-1$ if $t>1/2$ and $0$ otherwise.

We divide the proof in three cases according to the possible values of $t$ and $s$.

\textbf{Case 1. $t+s\geq 1$.}  This follows from Theorem \ref{criteria 1} since $P_1$ satisfies the strong determinant condition. 

\textbf{Case 2. $t\leq1/2$.} Our aim is to show that $i(XY-YX)$ has no atom at 0.
We assume that $X_n,Y_n\in M_{2n}$. Let $\{\lambda_i\}^{2n}_{i=1}$ be the multiset of eigenvalues of $X$ (counted with multiplicity). Since $s<1/2$, there are no eigenvalues with multiplicity larger than $n$, and thus we may reorder the eigenvalues so that $\lambda_{2i-1}\neq\lambda_{2i}$. Similarly, if $\{\rho_i\}^{2n}_{i=1}$ is the multiset of eigenvalues of $Y_n$ we may assume that $\rho_{2i-1}\neq\rho_{2i}$

Now, consider the block-diagonal matrices \begin{equation*}
X_n=\left(
    \begin{matrix}
     A_1&      &                  \\
      &     \ddots  &                \\
       &      &   A_n                
    \end{matrix}    \right), \qquad
   Y_n=\left(
   \begin{matrix}
     B_1&      &                  \\
      &     \ddots  &                \\
       &      &   B_n                
    \end{matrix}    \right)
\end{equation*}
where $$A_i=\left(
    \begin{matrix}
    \lambda_{2i-1}& 0   \\
      0&   \lambda_{2i} 
    \end{matrix}    \right)
  \qquad\text{and}\qquad  B_i=\frac{1}{2}\left(
       \begin{matrix}
    \rho_{2i-1}+\rho_{2i}
 & \rho_{2i-1}-\rho_{2i}
 \\
      \rho_{2i-1}-\rho_{2i}
&    \rho_{2i-1}+\rho_{2i}
    \end{matrix}    \right) .$$

As above
\begin{equation*}
[X_n,Y_n]=\left(
    \begin{matrix}
     [A_1,B_1]&      &                  \\
      &     \ddots  &                \\
       &      &   [A_n,B_n]                
    \end{matrix}    \right)
\end{equation*}
and thus it is enough to prove that $[A_i,B_i]$ is invertible, for all $i$.  

This follows from  \begin{equation*}
[A_i,B_i]=\frac{1}{2}\left(
    \begin{matrix}
     0&  (\lambda_{2i}-\lambda_{2i-1})(\rho_{2i}-\rho_{2i-1})        \\
    -(\lambda_{2i}-\lambda_{2i-1})(\rho_{2i}-\rho_{2i-1})   &     0  &            
    \end{matrix}    \right),
\end{equation*}
whose determinant is $(\lambda_{2i}-\lambda_{2i-1})^2(\rho_{2i}-\rho_{2i-1})^2\neq0$. 

\textbf{Case 3. $t>1/2, t+s<1$.} Let $m=t(2n)$, which we may assume to be an integer. 
Our aim is to show that $i(X_nY_n-Y_nX_m)$ has nullity $2m-2n=(2t-1)2n$, for some $X_n$ and $Y_n$ assuming that $X_n$ has nullity $m$ and $Y_n$ has all eigenspaces of size at most $n$, since $s<1/2$. Again if $\{\rho_i\}^{2n}_{i=1}$ is the multiset of eigenvalues of $Y_n$ we may assume that $\rho_{2i-1}\neq\rho_{2i}$

On the other hand, for the eigenvalues of $X_n$, since $t>1/2$ and thus $m>n$, we may assume that $\lambda_{2i}=0$, for $i=1,\dots,n$, and that $\lambda_{2i+1}\neq0$ for $i=1,\dots,2n-m$ and  $\lambda_{2i+1}=0$ for $i=2n-m+1,\dots,n$.

Now, consider again the block-diagonal matrices \begin{equation*}
X_n=\left(
    \begin{matrix}
     A_1&      &                  \\
      &     \ddots  &                \\
       &      &   A_n                
    \end{matrix}    \right),\qquad
   Y_n=\left(
   \begin{matrix}
     B_1&      &                  \\
      &     \ddots  &                \\
       &      &   B_n                
    \end{matrix}    \right)
\end{equation*}
where $$A_i=\left(
    \begin{matrix}
    \lambda_{2i-1}& 0   \\
      0&   \lambda_{2i} 
    \end{matrix}    \right)
  \qquad\text{and}\qquad  B_i=\frac{1}{2}\left(
       \begin{matrix}
    \rho_{2i-1}+\rho_{2i}
 & \rho_{2i-1}-\rho_{2i}
 \\
      \rho_{2i-1}-\rho_{2i}
&    \rho_{2i-1}+\rho_{2i}
    \end{matrix}    \right) .$$

Now, on one hand, from the calculation above $[A_i,B_i]$ is invertible for $i=1,...,2n-m$.  On the other hand, $[A_i,B_i]=0$ for $i=2n-m+1,\dots,n$, because $A_i=0$. Thus, the nullity of $[X_n,Y_n]$ is $2(m-n)$, as desired.

\end{proof}

\begin{remark}
\begin{enumerate}
    \item  When the variables $X_1,\ldots X_n$ are symmetric, the result may be deduced from Equation (1.29) in \cite{NS98} and the description of atoms for the free additive \cite{BV98} and  multiplicative \cite{Bel03} convolutions.
\item    
In \cite{LW21} the authors study  the closure properties of the class of freely infinitely divisible laws, $ID(\boxplus)$. They show that all the quadratic forms $Q(x_1, \ldots,x_d)$ such that $Q(x_1, \ldots,x_d)\in ID(\boxplus)$ whenever $X_1,\ldots,X_d\in ID(\boxplus)$ coincide with the linear span of commutators of free random variables. They also show that their distribution is symmetric. Since measures in $ID(\boxplus)$ have at most one atom, this tells that only possible atoms of $Q(X_1,\ldots,X_d)$ are at $0$ whenever $X_1,\ldots,X_d$ belong to $ID(\boxplus)$. Our results show that this is true for any free variables $X_1,\ldots,X_d$.
\end{enumerate}

\end{remark}

\subsection{Anticommutator}

Now we consider the atoms of the anticommutator $P_2(x,y)=\{x,y\}=xy+yx$, for free random variable.

\begin{theorem} \label{thm:anticommutator} Let $X$ and $Y$ be free random variables and let $Z=XY+YX$.

\noindent i) 
The size of the atom at $0$ of $Z$  is given by 
$$l:=\max\{2t-1,2s-1,s+u-1,t+r-1,0\},$$ where
\begin{enumerate}\item $t$ is the size of the atom at $0$ of $X$;
\item $s$ is the size of the atom at $0$ of $Y$;
\item  $u$ is the size of the largest atom outside of  $0$ of $X$;
\item $r$ is the size of the largest atom outside of  $0$ of $Y$.
\end{enumerate}

\noindent  ii) For any $a\neq0$, $Z$ has an atom at $a$ if and only if there exist weights $s(a)$ and $t(a)$ such that $t(a)+s(a)-1>0$, $X$ has an atom at $\lambda$ of size $s(a)$ and $Y$ has an atom at $\rho$ of size $t(a)$ and  $2 \lambda\rho=a$. 

Furthermore, the size of the atom of $Z$ at $a\neq0$ is given by $\max\{s(a)+t(a)-1,0\}$.
\end{theorem}

\begin{proof}

The size of atoms outside of $0$ may be calculated from the commutative case. Indeed for $\lambda,\rho\neq0$ we see that $P(\lambda,\rho)=2\lambda\rho$ satisfies the hypothesis of Theorem \ref{diagonalbound2}.

The atom at $0$ is more complicated and we need to use $2\times2$ matrices.  We will prove that the atom at $0$ cannot be bigger than $l$. The other inequality follows from Proposition \ref{unatoms1} and Corollary \ref{unatoms4}.

Before continuing we notice the following calculations for $2\times2$ matrices which we will use repeatedly.
For complex numbers $\lambda_i,\lambda_j$ and $\rho_i,\rho_j$, consider the matrices
\begin{equation*}
A(\lambda_i,\lambda_j)=\left(
    \begin{matrix}
      \lambda_i &  0          \\
    0  &     \lambda_j                 
    \end{matrix}
    \right)
\end{equation*}
and
\begin{equation*}
B_1(\rho_i,\rho_j)=\left(
    \begin{matrix}
      \rho_i&   0           \\
    0   &  \rho_j     
    \end{matrix}
    \right),\quad
B_2(\rho_i,\rho_j)=\frac{1}{2}\left(
    \begin{matrix}
      \rho_i+\rho_j&        \rho_i-\rho_j              \\
           \rho_i-\rho_j  &  \rho_i+\rho_j 
    \end{matrix}
    \right).
  \end{equation*}
A simple calculation shows that $\{A,B_1\}$ is invertible if $\lambda_i,\lambda_j,\rho_i,\rho_j$ are non-zero; $\{A,B_2\}$ is invertible if $\lambda_i\neq0$, $\lambda_j=0$ and $\rho_i\neq\rho_j$.

We take $X_n,Y_n\in M_{n}(\mathbb{C})$,
where the non-zero eigenvalues of $X_n$ are denoted by $\lambda_1,\dots,\lambda_{(1-t)n}$, the non-zero atom of $Y_n$ with the largest weight $r$ is denoted by $\rho$ and the remaining eigenvalues of $Y_n$ are denoted by $\rho_1,\dots,\rho_{(1-r)n}$.
We consider the different cases depending on $t$ and $r$.
Without loss of generality we assume that $t\geq s$.

\noindent \textbf{Case 1. } $t\leq 1/2$, $r+t\leq1$.
We divide $X_n$ into two parts, that is,
$$X_n=
\begin{pmatrix}
X_1 &0\\
0 &X_2
\end{pmatrix},\quad
Y_n=
\begin{pmatrix}
Y_1 &0\\
0 &Y_2\\
\end{pmatrix},$$
where
$X_1,Y_1\in M_{2tn}$ and $X_2,Y_2\in M_{(1-2t)n}$ are defined as follows:
$$
X_1=\begin{pmatrix}
A(\lambda_1,0)\\
& \ddots\\
& &A(\lambda_{tn},0)
\end{pmatrix},
$$
$X_2$ is a diagonal matrix that contains the remaining non-zero eigenvalues.
If $r\geq t$, then we define
$$
Y_1=\begin{pmatrix}
B_2(\rho,\rho_1)\\
& \ddots\\
& & B_2(\rho,\rho_{tn})
\end{pmatrix},
$$
where we may assume $Y_1$ contains all zero eigenvalues of $Y_n$ as $s\leq t$.
Then we have $\Null(\{X_1,Y_1\})=0$.
If $r<t$, then we have $r<1/2$.
In this case, we can order the eigenvalues of $Y_n$ such the any pair of adjacent eigenvalues is distinct and thus in particular $\text{Nullity}(\{X_1,Y_1\})=0$.
In both case, we let $Y_2$ be a diagonal matrix consisting of the remaining non-zero eigenvalues.
Therefore we have $\Null(\{X_2,Y_2\})=0$.

\noindent \textbf{Case 2.} $t\leq1/2$, $r+t>1$.
We divide $X_n$ and $Y_n$ into three parts, that is,
$$X_n=
\begin{pmatrix}
X_1 &0 &0\\
0 &X_2 &0\\
0 &0 &X_3 
\end{pmatrix},\quad
Y_n=
\begin{pmatrix}
Y_1 &0 &0\\
0 &Y_2 &0\\
0 &0 &Y_3 
\end{pmatrix}$$
where $X_1,Y_1\in M_{(r+t-1)n}$, $X_2,Y_2\in M_{(2-2r)n}$ and $X_3,Y_3\in M_{(r-t)n}$ are defined as follows:
$$
X_2=\begin{pmatrix}
A(\lambda_1,0)\\
& \ddots\\
& &A(\lambda_{(1-r)n},0)
\end{pmatrix},
$$
$X_1=\0$, $X_3$ is a diagonal matrix that contains the remaining non-zero eigenvalues and
$$
Y_2=\begin{pmatrix}
B_2(\rho,\rho_1)\\
& \ddots\\
& & B_2(\rho,\rho_{(1-r)n})
\end{pmatrix}
$$
(which is possible since $1-r<r$),
$Y_1$, $Y_3$ are diagonal matrices such that $Y_3$ contains no zero eigenvalues (which is possible since we can let $Y_1$ and $Y_2$ together has $tn$ many zero eigenvalues while we know $s\leq t$).
Hence we see that $\Null(\{X_2,Y_2\})=\Null(\{X_3,Y_3\})=0$ and thus $\Null(\{X_n,Y_n\})=(r+t-1)n$.

\noindent \textbf{Case 3.} $t\geq1/2$, $r\geq t$.
We write $t=(2t-1)+(r-t)+(1-r)$ and divide $X_n$ into three parts that contain correspondingly many zero eigenvalues.
Consider the matrices
$$X_n=\left(
    \begin{matrix}
      X_1&      0 &0         \\
   0    &  X_2  &0 \\
    0    & 0 & X_3 
    \end{matrix}
    \right),
    \qquad
    Y_n=\left(
    \begin{matrix}
      Y_1&      0 &0         \\
   0    &  Y_2  &0 \\
    0    & 0 &Y_3  
    \end{matrix}
    \right) $$
where $X_1,Y_1\in M_{(2t-1)n}$, $X_2,Y_2\in M_{(2r-2t)n}$ and $X_3,Y_3\in M_{(2-2r)n}$ are defined as follows: $X_1=\0$,
$$
X_2=\left(
    \begin{matrix}
      A(\lambda_1,0)&  & \\
      & \ddots &\\
 & &A( \lambda _{(r-t)n},0)
    \end{matrix}
    \right),\quad
X_3=\left(
    \begin{matrix}
      A(\lambda_{(r-t)n+1},0)&  & \\
      & \ddots &\\
 &  &A( \lambda _{(1-t)n},0)
    \end{matrix}
    \right)
$$
and
$$  Y_1=Y_2=\left(
    \begin{matrix}
     \rho &    &                \\
    & \ddots &  \\
   &      & \rho
    \end{matrix}
    \right),\quad
Y_3=\left(
    \begin{matrix}
    B_2 (\rho,\rho_1) &    &                \\
    & \ddots &  \\
   &      & B_2(\rho,\rho_{(1-r)n})
    \end{matrix}
    \right).
$$
Thus one sees that $$\text{Nullity}(\{X_1,Y_1\})=(2t-1)n,\quad \text{Nullity}(\{X_2,Y_2\})=(r-t)n,\quad \text{Nullity}(\{X_3,Y_3\})=0,$$ 
and so $\text{Nullity}(\{X_n,Y_n\})=(r+t-1)n$.

\noindent \textbf{Case 4.} $t\geq1/2$, $r<t$. 
We proceed similarly as in Case 3 but we do not need $X_2,Y_2$.
Correspondingly, we let $X_3,Y_3$ have size $(2-2t)n$.
If $s\geq 1-t$ or $r\geq 1-t$, then we can build $Y_3$ with block matrix $B_2(0,\rho_i)$ or $B_2(\rho,\rho_i)$ respectively.
Hence $\text{Nullity}(\{X_3,Y_3\})=0$.
Otherwise we have $s<1/2$ and $r<1/2$, in which case we can order the eigenvalues of $Y_n$ such the any pair of adjacent eigenvalues is distinct and thus in particular $\text{Nullity}(\{X_3,Y_3\})=0$.
So the desired result follows similarly as in Case 3.
\end{proof}

\begin{remark}
One may want to include both calculations in this section in the more general situation $ e^{it}XY+e^{-it}YX$. It turns out that, whenever  $ \cos(t)\neq 0$, the situation is exactly the same as in the anticommutator case, by just changing $2\mu\lambda$ to $2 cos(t) \mu\lambda $.  Indeed, considering $\tilde X=e^{it}X$ and using the same notation as in Theorem \ref{thm:anticommutator}, one sees that $AB_1+B_1^*A$ is invertible if $\lambda_i,\lambda_j,\rho_i,\rho_j$ are non-zero and $AB_2+B_2^*A$ is invertible if $\lambda_i\neq0$, $\lambda_j=0$ and $\rho_i\neq\rho_j$. The rest of the proof then works mutatis mutandis.
\end{remark}

\subsection{Multiplicative commutator}

Now we consider the rational function $XYX^{-1}Y^{-1}$ where X and Y are free invertible elements. The particular case, when $X$ and $Y$ are unitaries is of special interest but for the calculation here, it is not needed.

Since when $X$ and $Y$ are commutative, clearly $XYX^{-1}Y^{-1}=1$, then the only possible atom is at $1$. In principle we could try to find some $2\times 2$ matrix to have a similar calculations as for the commutator, however the following calculation shows that the  size of the atom at $0$ of $XY-YX$ is the same as the size of the atom at $1$ for  $XYX^{-1}Y^{-1}$. 

$$\rank(XY-YX)=\rank((XY-YX)X^{-1}Y^{-1})=\rank(XYX^{-1}Y^{-1}-1).$$

Thus arriving to the following conclusion.

\begin{corollary} Let $X$ and $Y$ be free random variables that are invertible. Let $t$ and $s$ be the size of the largest atoms of $X$ and $Y$, then the only possible atom of  $XYX^{-1}Y^{-1}$ is at $1$ with mass $max(2t-1,2s-1,0)$.
\end{corollary}

\subsection{ Universal Covers of Graphs}

Now we present an application to the problem of relating the spectrum of the adjacency matrix $A(G)$ of a graph with the point spectrum of the density of states of its universal cover (or universal covering tree). We thank Jorge Garza-Vargas for leading our attention to this possible application. 

Let $G=(V(G),E(G))$ be a finite simple undirected connected graph. The universal cover of $G$, that we denote by  $T(G)$, is the graph constructed as follows:

\begin{enumerate}
    \item 
 Fix any vertex $v\in V(G)$. The set of vertices of $T(G)$ is given by non-backtracking walks in $G$ starting at $v$, including the empty walk. Here, a non-backtracking walk means a finite sequence of vertices $v_1,v_2,\cdots,v_n$, so that $v_i\sim v_{i+1}$, for $i=1,\ldots,n-1$ and $v_i\neq v_{i+2}$ for $i=1,\ldots,n-2$.

 \item The set of edges of $T(G)$ is given by the pair of non-backtracking walks such that one of them can be obtained by appending a vertex. 
\end{enumerate}

$T(G)$ is independent of $v$, up to isomorphism. $T(G)$ is called the universal cover because it is the unique graph, up to isomorphism,  that covers every other cover of $G$.  

One can associate to $T(G)$ a probability measure known as the density of states of $T(G)$, we refer for precise definitions to \cite{BGVM20} and references therein. The relation with our results comes from the fact, that as observed in \cite{BC19} and \cite{GVK19}, if $G$ has adjacency matrix $A=(A_{i,j})_{i,j}$ and $\{u_{i,j}\}_{0\leq i,j\leq n}$ is a family of Haar distributed unitaries with $u_{i,j}=u_{j,i}^{-1}$ and $\{u_{i,j}\}_{i>j}$ freely independent, then the matrix $U(A)$ defined by  $U(A)_{i,j}=A_{i,j}u_{i,j}$ has the same distribution as the density of states associated to the universal cover $T(G)$.

\begin{prop}\cite[Theorem 3.2]{BGVM20} \label{prop:spectrumgraphs}The point spectrum of $T(G)$ is contained in the set of eigenvalues of $A(G)$. Moreover if $\mu$ denotes the density of states of $T(G)$ and $\mu_{A(G)}$ denotes the empirical eigenvalue distribution of $A(G)$
then $\mu_{T(G)}^p\leq\mu_{A(G)}$. 
\end{prop}

Here we want to see how this fact may deduced from  Theorem \ref{thm:comparison-rational functions}. To do this, we consider the matrix $R(A)$ on the variables $X:=(x_{i,j})_{i>j}$ given by $R(A)(x)_{i,j}=A_{i,j}x_{i,j}$  if $i>j$,  $R(A)(X)_{i,j}=A_{i,j}x^{-1}_{j,i}$ if $i<j$ and $R(A)(X)_{i,i}=0$.

Now, $R(A)(X)$ does not belong to $M_n(\C\left<x_{2,1},\dots,x_{n,n-1}\right>)$, nor is a rational expression in the variables $\{x_{ij}\}_{i>j}$. However by following the same line of proof of part (2) in Theorem \ref{thm:comparison-rational functions},  one may arrive to the same conclusion for matrices with entries which are rational expressions (in this case $x_{i,j}^{-1}$).

Now we only need to observe that the variables $u_{ij}$ are normal and $\mu^p_{u_{i,j}}\leq \mu^p_{1}=\delta_1$, since the variables  $u_{i,j}$ are absolutely continuous.  So $U=(u_{i,j})_{i>j}$ and $\mathbf{1}=(1)_{i>j}$ satisfy the hypothesis of Theorem \ref{thm:comparison-rational functions} and we conclude that $\mu^p_{R(U)}\leq\mu^p_{R(\mathbf{1})}$. Proposition \ref{prop:spectrumgraphs}, i.e $\mu^p_{U(A)}\leq \mu^p_{A}=\mu_A$, now follows from the observation that $R(U)=U(A)$ and $R(\mathbf{1})=A$.

\bibliographystyle{amsalpha}
\bibliography{References}
\end{document}